\documentclass[11pt]{article}
\usepackage{amssymb,latexsym,amsmath,amsbsy,amsthm,amsxtra,amsgen,graphicx,makeidx,dsfont,longtable}
\numberwithin{equation}{section}

\newfont{\msbm}{msbm10 at 11pt}

\newcommand {\Z} {\mbox{\msbm Z}}

\newcommand {\N} {\mbox{\msbm N}}
\newcommand {\1} {\mathds{1}}

\newfont{\msbmsm}{msbm10 at 8pt}

\newtheorem{Theo}{Theorem}[section]
\newtheorem{Lemma}[Theo]{Lemma}
\newtheorem{Cor}[Theo]{Corollary}
\newtheorem{Prop}[Theo]{Proposition}

\newtheorem{Rmk}[Theo]{Remark}

\def\eps{\varepsilon}
\def\Var{\textup{Var}}

\begin{document}
\title{Rigorous results for a population model with selection II: \\ genealogy of the population}
\author{Jason Schweinsberg\thanks{Supported in part by NSF Grant DMS-1206195} \\
University of California at San Diego}
\maketitle



\vspace{-.6in}
\begin{abstract}
We consider a model of a population of fixed size $N$ undergoing selection.  Each individual acquires beneficial mutations at rate $\mu_N$, and each beneficial mutation increases the individual's fitness by $s_N$.  Each individual dies at rate one, and when a death occurs, an individual is chosen with probability proportional to the individual's fitness to give birth.  Under certain conditions on the parameters $\mu_N$ and $s_N$, we show that the genealogy of the population can be described by the Bolthausen-Sznitman coalescent.  This result confirms predictions of Desai, Walczak, and Fisher (2013), and Neher and Hallatschek (2013).
\end{abstract}

\section{Introduction}

In population genetics, one is often interested in understanding the genealogical structure of a population.  That is, we take a sample of individuals from a population at some time and trace their ancestral lines backwards in time.  As we trace the ancestral lines backwards in time, the lineages will merge until eventually all sampled individuals are traced back to one common ancestor.  For many standard population models, including the classical Moran model \cite{moran}, the genealogy of the population is best described by a process known as Kingman's coalescent, which was introduced in \cite{king82}.  Kingman's coalescent is the coalescent process in which only two lineages ever merge at one time and each pair of lineages merges at rate one.  

For populations undergoing selection, Kingman's coalescent does not always provide an adequate description of the genealogy of the population.  If one individual acquires a beneficial mutation which then spreads rapidly to a large fraction of the population, many ancestral lines could merge nearly at once because they all get traced back to the individual that acquired the beneficial mutation.  As a result, the genealogy of the population is best described by a coalescent process that permits more than two lineages to merge at one time.  Such processes, known as coalescents with multiple mergers or $\Lambda$-coalescents, were introduced by Pitman \cite{pit99} and Sagitov \cite{sagitov} and have been studied extensively in the probability literature in recent years.  For previous work in which coalescents with multiple mergers were used to describe the genealogy of populations undergoing selection, see \cite{bbs, bdmm06, bdmm07, dwf13, ds05, nh13}.

In this paper, we will consider the following population model.  The population has fixed size $N$.  Each individual independently acquires mutations at times of a Poisson process with rate $\mu_N$.  All mutations are assumed to be beneficial, and the fitness of each individual depends on how many mutations the individual has acquired, relative to the mean of the population.  More precisely, let $X_j(t)$ be the number of individuals with $j$ mutations at time $t$, which we call type $j$ individuals, and let $$M(t) = \frac{1}{N} \sum_{j=0}^{\infty} j X_j(t)$$ be the average number of mutations carried by the individuals in the population at time $t$.  Then the fitness of an individual with $j$ mutations at time $t$ is defined to be $$\max\big\{0, 1 + s_N(j - M(t))\big\}.$$  Note that the parameter $s_N$ measures the selective advantage that an individual gets from each mutation.  As in the Moran model, each individual independently lives for an exponentially distributed time with mean one.  When an individual dies, it gets replaced by a new individual whose parent is chosen at random from the population.  The probability that a particular individual is chosen as the parent is proportional to that individual's fitness, and the new individual inherits all of its parent's mutations.

This model was studied in great detail using nonrigorous methods by Desai and Fisher \cite{df07}, who obtained results concerning the rate of adaptation, meaning the rate at which the mean fitness $M(t)$ grows as a function of time, as well as the distribution of the fitnesses of individuals in the population at a given time.  See also \cite{brw08, rbw08, yec10} for related results, and see \cite{yec10} for a good summary of the literature on this model and closely related models.  The genealogy of the population in this model has been studied only within the past few years.  Desai, Walczak, and Fisher \cite{dwf13} argued that the genealogy of the population can be described by a process called the Bolthausen-Sznitman coalescent, which we will define precisely in section \ref{mainres}.  Neher and Hallatschek \cite{nh13} arrived at the same conclusion for a slightly different model.

This model was also studied in detail in the paper \cite{schI}, which contains rigorous proofs of the results of Desai and Fisher \cite{df07} concerning the rate of adaptation and the distribution of fitnesses of individuals in the population.  In the present paper, which is a sequel to \cite{schI}, we build on the techniques developed in \cite{schI} to provide a mathematically rigorous description of the genealogy of the population.  We confirm nonrigorous predictions presented in \cite{dwf13, nh13} and show that the genealogy of the population is given by the Bolthausen-Sznitman coalescent, under suitable conditions on the parameters $s_N$ and $\mu_N$.

The rest of this paper is organized as follows.  In section \ref{mainres}, we state precisely our assumptions and the main result of the paper, which is Theorem \ref{boszthm} below.  In section \ref{heursec}, we give a heuristic argument that explains the ideas behind why Theorem \ref{boszthm} is true, and we make some connections with other results in the literature.  In section \ref{reviewsec}, we summarize the results from \cite{schI} that will be needed in the present paper.  The remaining sections are devoted to proving Theorem \ref{boszthm}.

\section{Assumptions and Main Result}\label{mainres}

We first define the following two quantities, which were also used in \cite{schI} and which are important for scaling the process correctly:
\begin{equation}\label{kNaN}
k_N = \frac{\log N}{\log(s_N/\mu_N)}, \hspace{.5in}a_N = \frac{\log(s_N/\mu_N)}{s_N}.
\end{equation}
As we will see below, $k_N$ is the natural scale for the number of mutations because the difference in the number of mutations carried by the fittest individual in the population and an individual of average fitness is typically within a constant multiple of $k_N$.  Also, we will see that $a_N$ is the natural time scale on which to study the process.

We will need the following assumptions on the parameters $s_N$ and $\mu_N$, which are identical to the three assumptions that appeared in \cite{schI}:

\bigskip
{\bf A1}:  We have ${\displaystyle \lim_{N \rightarrow \infty} \frac{k_N}{\log (1/s_N)} = \infty}$.

\bigskip
{\bf A2}:  We have ${\displaystyle \lim_{N \rightarrow \infty} \frac{k_N \log k_N}{\log(s_N/\mu_N)} = 0.}$

\bigskip
{\bf A3}:  We have ${\displaystyle \lim_{N \rightarrow \infty} s_N k_N = 0}$.

\bigskip
\noindent 
Dividing A3 by A1, we get
\begin{equation}\label{sto0}
\lim_{N \rightarrow \infty} s_N = 0.
\end{equation}
Also, as noted in \cite{schI}, these assumptions imply that for all $a > 0$, we have
\begin{equation}\label{musN}
\lim_{N \rightarrow \infty} \frac{\mu_N}{s_N^a} = \lim_{N \rightarrow \infty} \frac{1}{\mu_N N^a} = 0,
\end{equation}
which means the mutation rate $\mu_N$ tends to zero faster than any power of $s_N$ but more slowly than any power of $1/N$.  

In view of (\ref{sto0}), assumption A1 implies that $\lim_{N \rightarrow \infty} k_N = \infty$.  This means that the difference between the number of mutations carried by the fittest individual and the number carried by an individual of average fitness tends to infinity as $N \rightarrow \infty$.  Because each additional mutation adds $s_N$ to the fitness of an individual, assumption A3 implies that the difference in fitness between these two individuals tends to zero as $N \rightarrow \infty$.  As discussed in \cite{schI}, assumption A2 ensures that mutations do not happen too fast for the analysis in this paper and \cite{schI} to be valid.  Understanding how the population evolves under faster mutation rates is an important question for future work.

Although the parameters $\mu_N$ and $s_N$ depend on $N$, we will drop the subscripts and write $\mu$ and $s$ throughout the rest of the paper to lighten notation.

Before stating the main result, we need to define the Bolthausen-Sznitman coalescent, which was introduced in \cite{bosz98}.  The Bolthausen-Sznitman coalescent is a continuous-time Markov chain $(\Pi(t), t \geq 0)$ taking its values in the set of partitions of $\{1, \dots, n\}$.  It is defined by the property that $\Pi(0) = \{\{1\}, \dots, \{n\}\}$ is the partition of $1, \dots, n$ into singletons, and then whenever the partition has $b$ blocks, each possible transition that involves merging $k$ of the blocks into one, where $2 \leq k \leq b$, happens at rate
\begin{equation}\label{lambk}
\lambda_{b,k} = \int_0^1 y^{k-2} (1-y)^{b-k} \: dy,
\end{equation}
and these are the only possible transitions.  A more detailed construction of the Bolthausen-Sznitman coalescent will be given shortly in section \ref{boszsec}.

\begin{Theo}\label{boszthm}
Assume A1-A3 hold.  Fix positive real numbers $t$ and $T$ such that $t > 0$ and $T > t + 2$.  Fix a positive integer $n$, and sample $n$ individuals at random from the population at time $a_N T$.  For $0 \leq u \leq t+1$, let $\Pi_N(u)$ be the partition of $\{1, \dots, n\}$ such that $i$ and $j$ are in the same block of the partition if and only if the $i$th and $j$th sampled individuals have the same ancestor in the population at time $a_N(T - u)$.  Then
\begin{equation}\label{part1}
\lim_{N \rightarrow \infty} P(\Pi_N(1) = \{\{1\}, \dots, \{n\}\}) = 1.
\end{equation}
Also, the finite-dimensional distributions of $(\Pi_N(1 + u), 0 \leq u \leq t)$ converge as $N \rightarrow \infty$ to the finite-dimensional distributions of the Bolthausen-Sznitman coalescent.
\end{Theo}

Note that Theorem \ref{boszthm} stipulates that with probability tending to one as $N \rightarrow \infty$, the sampled individuals at time $a_N T$ will all be descended from different ancestors at time $a_N(T - 1)$.  However, as the ancestral lines are traced back further, the merging of these ancestral lines obeys the law of the Bolthausen-Sznitman coalescent.  This result also appears in \cite{dwf13}, where it was obtained by nonrigorous methods.

\section{Heuristics and Background}\label{heursec}

\subsection{The Bolthausen-Sznitman coalescent}\label{boszsec}

Recall that the Bolthausen-Sznitman coalescent is the coalescent process whose transition rates are given by (\ref{lambk}).  Pitman \cite{pit99} showed how to construct the Bolthausen-Sznitman coalescent from a Poisson process.  We give a variation of this construction here.  Consider a Poisson process on $[0, \infty) \times (0, 1] \times [0,1]^n$ with intensity $$dt \times y^{-2} \: dy \times dz_1 \times \dots \times dz_n.$$
Let $\Pi(0) = \{\{1\}, \dots, \{n\}\}$ be the partition of $1, \dots, n$ into singletons.  If $(t, y, z_1, \dots, z_n)$ is a point of the Poisson process, and if the blocks of the partition $\Pi(t-)$, ranked in order by their smallest elements, are $B_1, \dots, B_b$, then $\Pi(t)$ is the partition obtained from $\Pi(t-)$ by merging together all of the blocks $B_i$ for which $z_i \leq y$.

Informally, this means that if $(t, y)$ are the first two coordinates of a point of the Poisson process, then at time $t$ we have a so-called $y$-merger, in which each block independently participates in the merger with probability $y$.  If $\Pi(t-)$ has $b$ blocks, then for $2 \leq k \leq b$, the probability that a particular set of $k$ blocks merges into one is $y^k (1-y)^{b-k}$, which allows us to recover the formula (\ref{lambk}) for the transition rates.

To see that the construction above is well-defined, note that a point $(t, y, z_1, \dots, z_n)$ of the Poisson process can only produce a merger at time $t$ if at least two of $z_1, \dots, z_n$ are less than or equal to $y$.  The rate at which such points appear is bounded above by $$\int_0^1 y^{-2} \cdot \binom{n}{2} y^2 \: dy < \infty.$$  Therefore, only finitely many such points will appear in any bounded time interval, and the construction above can be carried out by considering these points in order by their time coordinate.

We now give a heuristic argument to explain when the Bolthausen-Sznitman coalescent should be expected to describe the genealogy of a population.  Note that if a population has size $S$ and then a new large family of size $Sx$ suddenly appears, then the fraction of the population belonging to the large family will be $x/(1+x)$.  Consequently, if we are tracing ancestral lines backwards in time, approximately a fraction $x/(1+x)$ of the lineages will coalesce around the time that this family appears.  That is, we will have a $y$-merger with $y = x/(1+x)$.  For the Bolthausen-Sznitman coalescent, we can see from the Poisson process construction above that $y$-mergers with $y \geq x/(1+x)$ occur at rate
\begin{equation}\label{xmerger}
\int_{x/(1+x)}^1 y^{-2} \: dy = x^{-1}.
\end{equation}
Therefore, the Bolthausen-Sznitman coalescent will describe the genealogy of a population when families of size $Sx$ or larger appear at a rate proportional to $x^{-1}$.

\subsection{A heuristic argument for Theorem \ref{boszthm}}

In this subsection, we give a short approximate calculation to suggest why Theorem \ref{boszthm} should be true.  For $j \in \N$, let
\begin{equation}\label{taujdef}
\tau_j = \inf\bigg\{t: X_{j-1}(t) \geq \frac{s}{\mu} \bigg\}
\end{equation}
be the first time that there are at least $s/\mu$ individuals in the population with $j-1$ mutations.  It was shown in \cite{schI} that typically no individual acquires a $j$th mutation until after time $\tau_j$.  We write for now $q_j = j - M(\tau_j),$
which is the difference between $j$ and the mean number of mutations carried by the individuals in the population at time $\tau_j$.  As argued in \cite{df07, schI}, shortly after time $\tau_j$, the number of type $j-1$ individuals in the population is growing approximately exponentially at the rate $s(q_j - 1)$, which means that when $t$ is slightly larger than $\tau_j$, we have
\begin{equation}\label{Xj-1approx}
X_{j-1}(t) \approx \frac{s}{\mu} e^{s(q_j - 1)(t - \tau_j)}.
\end{equation}
Because each type $j-1$ individual independently acquires mutations at rate $\mu$, at time $u$ we have type $j$ individuals appearing due to a mutation at rate $\mu X_{j-1}(u)$.  If such a mutation happens at time $u$, then because type $j$ individuals have a selective advantage of approximately $sq_j$ over the rest of the population, the expected number of descendants of this mutation alive at time $t$ is approximately $e^{sq_j(t - u)}$.  Therefore, using (\ref{Xj-1approx}),
\begin{equation}\label{Xjapprox}
X_j(t) \approx \int_{\tau_j}^t \mu \cdot \frac{s}{\mu} e^{s(q_j - 1)(u - \tau_j)} \cdot e^{sq_j(t - u)} \: du = s e^{sq_j(t - \tau_j)} \int_{\tau_j}^t e^{-s(u - \tau_j)} \: du \approx e^{sq_j(t - \tau_j)}.
\end{equation}

Usually, the type $j$ individuals will belong to many small families.  That is, many type $j-1$ individuals will acquire mutations, each of which will become the ancestor of only a small fraction of the type $j$ population.  
In that case, the approximation in (\ref{Xjapprox}) will be valid.  However, occasionally there can be an unusually early mutation, when a type $j-1$ individual acquires a $j$th mutation much sooner than expected.  When this occurs, the descendants of the new type $j$ individual can eventually constitute a significant fraction of the type $j$ individuals in the population.  These unusually large families can lead to multiple mergers of ancestral lines, as many lineages get traced back to the individual that got the early mutation.

To estimate the probability that this happens, we approximate $q_j - 1$ by $q_j$ in (\ref{Xj-1approx}) to see that at time $u$, mutations from type $j-1$ to type $j$ are occurring at rate approximately $s e^{s q_j (u - \tau_j)}.$  If such a mutation does occur, then the number of descendants of this mutation behaves like a supercritical branching process with deaths at rate $1$ and births at rate $1 + sq_j$.  Such a branching process survives with probability approximately $s q_j$ and, conditional on survival, the size of the population after it has evolved for time $t - u$ is approximately $$\frac{W}{s q_j} e^{s q_j (t - u)},$$ where $W$ has an exponential distribution with mean one.  In particular, a successful mutation that occurs at time $$u = \tau_j + \frac{1}{s q_j} \log \bigg( \frac{1}{sq_j} \bigg) + v$$ has approximately $$W e^{-sq_jv} e^{sq_j(t - \tau_j)}$$ descendants in the population at time $t$.  Write $S = e^{sq_j(t - \tau_j)}$, which from (\ref{Xjapprox}) is approximately the number of type $j$ individuals at time $t$ that do not come from unusually early mutations.  By integrating over the possible times when the mutation could occur, we see that the probability that there will be a mutation that is the ancestor of at least $Sx$ type $j$ individuals at time $t$ is approximately
\begin{equation}\label{earlymut}
\int_{-\infty}^{\infty} s e^{sq_j [\log(1/sq_j)/sq_j + v]} \cdot sq_j \cdot P(W e^{-sq_jv} > x) \: dv 
= \int_{-\infty}^{\infty} s e^{sq_jv} e^{-xe^{sq_j v}} \: dv = \frac{1}{q_jx}.
\end{equation}
Note that the factor of $x^{-1}$ on the right-hand side of (\ref{earlymut}) matches the right-hand side of (\ref{xmerger}).

Consider now what happens when we sample $n$ individuals from the population at time $a_N T$ and trace their ancestral lines backwards in time.  As noted in \cite{schI}, one type will dominate the population at a typical time, so with high probability, the sampled individuals will all have the same type, which we will call type $\ell$.  With high probability, the sampled individuals will be descended from distinct type $\ell$ ancestors at time $\tau_{\ell+1}$.  Because we will see that the time between when type $\ell$ individuals originate and when they become the dominant type in the population is approximately $a_N$, this means the ancestral lines will most likely not merge when they are traced back from time $a_N$ to time $a_N(T - 1)$, which leads to the result (\ref{part1}).

As we trace the lineages further back, with high probability they get traced back to type $\ell-1$ ancestors at time $\tau_{\ell}$, then to type $\ell-2$ ancestors at time $\tau_{\ell-1}$, and so on.  At each stage of this process, there is a small probability that a group of ancestral lines will merge together because they get traced back to an individual that acquired an unusually early mutation.  Because of the agreement between (\ref{xmerger}) and (\ref{earlymut}), these mergers follow the same dynamics, in the limit as $N \rightarrow \infty$, as the Bolthausen-Sznitman coalescent.

The explanation given here for the appearance of the Bolthausen-Sznitman coalescent is very similar to that given by Desai, Walczak, and Fisher \cite{dwf13} and by Neher and Hallatschek \cite{nh13}, though these authors did not work directly from the Poisson process construction of the Bolthausen-Sznitman coalescent.  
 
\subsection{Comparison with branching Brownian motion}
 
Theorem \ref{boszthm} resembles the main result of \cite{bbs}, in which the authors confirmed nonrigorous predictions of Brunet, Derrida, Mueller, and Munier \cite{bdmm06, bdmm07} and showed that the Bolthausen-Sznitman coalescent describes the genealogy in a different population model involving selection.  In \cite{bbs}, the population was modeled by branching Brownian motion with absorption, in which initially there is some configuration of particles in $(0, \infty)$, each particle independently moves according to Brownian motion with drift $-\nu_N$, each particle divides into two at rate one, and particles are killed upon reaching the origin.  In this model, the particles represent individuals in a population, the position of a particle corresponds to the fitness of the individual, branching events represent births, and killing at the origin models the deaths of individuals whose fitness is too low.  It was shown in \cite{bbs} that if the initial configuration of particles and the drift parameter $\nu_N$ are chosen so that the number of particles in the system stays comparable to $N$, then the genealogy of this population is given by the Bolthausen-Sznitman coalescent.

One difference between the model in \cite{bbs} and the model studied in this paper is that for branching Brownian motion with absorption, all individuals have the same birth rate, while individuals with low fitness are killed.  In the model considered here, all individuals have the same death rate, while individuals with higher fitness are more likely to give birth.  In part because of this difference, the two population models behave quite differently in many respects.  For example, in the branching Brownian motion model, the speed of evolution is measured by the drift $\nu_N$ required to maintain a stable population size, which is $$\nu_N = \sqrt{2 - \frac{2 \pi^2}{(\log N + 3 \log \log N)^2}}.$$  That is, as $N \rightarrow \infty$, the speed of evolution tends to the limiting value $\sqrt{2}$ at the rate of $(\log N)^{-2}$.  This kind of behavior was first observed by Brunet and Derrida \cite{bd97} and was verified rigorously for other probabilistic models in \cite{bg10, maillard, mmq11}.  However, as shown in \cite{df07, schI}, the population model studied in the present paper does not show this behavior.  Also, for branching Brownian motion with absorption, once the particles reach a sort of equilibrium, the density of particles near $y$ is roughly proportional to $$e^{-\nu_N y} \sin \bigg( \frac{\pi y}{L_N} \bigg),$$ where $L_N = (\log N + 3 \log \log N)/\sqrt{2}$.  This is again quite different from the results for the model studied in this paper, where the distribution of fitnesses has a Gaussian-like shape; see, for example, \cite{beer07, df07, rwc03, schI, yec10}.  Finally, for branching Brownian motion with absorption, if two particles are sampled at some time, then the time that one has to go back a common ancestor of these two particles is comparable to $(\log N)^3$, as compared with the time scaling by $a_N$ in Theorem \ref{boszthm}.  Yet, in spite of these differences, we find that the Bolthausen-Sznitman coalescent describes the genealogy in both models.

\subsection{Connection with multitype branching processes}

We mention here how the appearance of the Bolthausen-Sznitman coalescent in this model could have been predicted from known results about multitype branching processes.  Consider a two-type Yule process in which type 1 individuals give birth to type 1 individuals at rate $\lambda$ and to type 2 individuals at rate $\mu$, and type 2 individuals give birth to type 2 individuals at rate $\lambda + s$.  If we say that type 2 individuals belong to the same family when they are descended from the same mutation, then the sizes of type 2 families at some large time $t$ can be approximated by the points of a Poisson process on $(0, \infty)$ with intensity $C x^{-1-\alpha}$, where $C$ is a constant and $\alpha = \lambda/(\lambda + s)$; see Theorem 3 of \cite{dm10} and the following corollary.  This implies that the total number of type 2 individuals has approximately a stable law of index $\alpha$ and that the distribution of the family sizes, normalized to sum to one, is the Poisson-Dirichlet distribution with parameters $(\alpha, 0)$, which was introduced in \cite{py97}.

If $(\Pi(t), t \geq 0)$ is the Bolthausen-Sznitman coalescent, then the distribution of the block sizes of $\Pi(t)$, normalized to sum to one, converges as $n \rightarrow \infty$ to the Poisson-Dirichlet distribution with parameters $(e^{-t}, 0)$, as shown in \cite{pit99}.  Thus, the appearance of stable laws in the work Durrett and Moseley \cite{dm10} and Durrett, Foo, Leder, Mayberry, and Michor \cite{dflmm}, who studied a multitype branching process model for tumor progression, and the appearance of the Poisson-Dirichlet distribution in the work of Leviyang \cite{leviyang}, who studied the coalescence of HIV lineages in a similar model, strongly suggest that the Bolthausen-Sznitman coalescent should describe the genealogy in similar models when the selective advantage $s$ is tending to zero.  This conjecture is confirmed by Theorem \ref{boszthm} above.  Indeed, the work \cite{dflmm, dm10, leviyang}, which appeared before the work of Desai, Walczak, and Fisher \cite{dwf13} and Neher and Hallatschek \cite{nh13}, served as the original motivation for the present paper.

\section{Review of results from \cite{schI}}\label{reviewsec}

The population model considered in this paper was also studied extensively in \cite{schI}, and in the present paper, we will make heavy use of some of the results and techniques developed in \cite{schI}.  In this section, we will state the results from \cite{schI} that we will need.  

\subsection{Evolution of type $j$ individuals}

We first present some results summarizing how the type $j$ individuals evolve.  Let $\eps > 0$, $\delta > 0$, and $T > 1$.
Recall the definition of $k_N$ from (\ref{kNaN}), and let $$k^* = \max\bigg\{j \in \N: j < k_N + \frac{2k_N \log k_N}{\log(s/\mu)}\bigg\}.$$  Note from assumption A2 that $(2k_N \log k_N)/\log(s/\mu) \rightarrow 0$ as $N \rightarrow \infty$.  As discussed in \cite{schI}, for $j \leq k^*$, individuals of type $j$ appear in the population very quickly.  To understand the evolution of the type $j$ individuals for $j \geq k^* + 1$, define
\begin{equation}\label{bdef}
b = \log \bigg(\frac{24000 \,T}{\delta^2 \eps} \bigg).
\end{equation}
Also, define $\tau_j$ as in (\ref{taujdef}), and then set
\begin{displaymath}
q_j^* = \left\{
\begin{array}{ll} j - k_N & \mbox{ if } a_N - 2a_N/k_N \leq \tau_j \leq a_N + 2a_N/k_N \\
j - M(\tau_j) & \mbox{ otherwise }
\end{array} \right.
\end{displaymath}
and 
\begin{equation}\label{qjdef}
q_j = \max\{1, q_j^*\}.
\end{equation}
Next, let
\begin{equation}\label{xijdef}
\xi_j = \max \bigg\{\tau_j, \: \tau_j + \frac{1}{sq_j} \log \bigg( \frac{1}{sq_j} \bigg) + \frac{b}{sq_j} \bigg\},
\end{equation}
as in \cite{schI}.  Every type $j$ individual at time $t$ has an ancestor that acquired a $j$th mutation before time $t$.  If this $j$th mutation occurred at or before time $\xi_j$, we call the individual an early type $j$ individual.  When an individual gets its $j$th mutation, we call this a type $j$ mutation, and we call such a mutation an early type $j$ mutation if it occurs at or before time $\xi_j$.  Let $X_{j,1}(t)$ denote the number of early type $j$ individuals at time $t$, and let $X_{j,2}(t)$ denote the number of other type $j$ individuals at time $t$, which means $$X_j(t) = X_{j,1}(t) + X_{j,2}(t).$$  For $t \geq 0$, let $$G_j(t) = s(j - M(t)) - \mu,$$ which represents the growth rate of the type $j$ individuals in the population at time $t$.  For $j \geq k^*+1$, let
\begin{equation}\label{gammajdef}
\gamma_j = \tau_j + a_N
\end{equation}
and
$$\tau_j^* = \tau_j + \frac{a_N}{4Tk_N}.$$

Proposition \ref{mainjprop} collects several results related to how the type $j$ individuals evolve.  The first four parts of the proposition are identical to Proposition 3.3 of \cite{schI}, except for the last statement of part 1, which comes instead from Lemma 8.18 of \cite{schI}.  The first two parts of the proposition describe how the type $j$ individuals emerge before time $\tau_{j+1}$.  Part 3 describes the evolution of the type $j$ individuals after time $\tau_{j+1}$ but before the type $j$ individuals start to get close to extinction.  Part 4 bounds the extinction time for the type $j$ individuals, as well as the size of the type $j$ population as it nears extinction.  Part 5 of the proposition, which is Remark 6.9 in \cite{schI},  demonstrates that nearly all individuals in the population have type $j$ between times $\gamma_j$ and $\gamma_{j+1}$.  Finally, part 6, which is a combination of parts 1 and 3 of Proposition 3.6 in \cite{schI}, bounds the difference between $\tau_j$ and $\tau_{j+1}$.

\begin{Prop}\label{mainjprop}
There exist positive constants $C_1$ and $C_2$, depending on $\delta$, $\eps$, and $T$, such that if $N$ is sufficiently large, then the following statements all hold with probability at least $1 - \eps$:
\begin{enumerate}
\item For all $j \geq k^* + 1$ and all $t \in [\tau_j^*, \tau_{j+1}] \cap [0, a_N T]$, we have
\begin{equation}\label{prop21}
X_{j,1}(t) \leq C_1 \exp \bigg( \int_{\tau_j}^t G_j(v) \: dv \bigg).
\end{equation}
Also, $X_{j,1}(t) \leq s/2\mu$ for all $t \leq \tau_j^* \wedge a_N T$, and no early type $j$ individual acquires a type $j+1$ mutation until after time $\tau_{j+1} \wedge a_N T$.  Moreover, no individual that gets a $j$th mutation at or before time $\tau_j$ has a descendant alive in the population at time $\tau_j^*$.

\item For all $j \geq k^* + 1$ and all $t \in [\tau_j^*, \tau_{j+1}] \cap [0, a_N T]$, we have
\begin{equation}\label{prop22}
(1 - 4 \delta) \exp \bigg( \int_{\tau_j}^t G_j(v) \: dv \bigg) \leq X_{j,2}(t) \leq (1 + 4 \delta) \exp \bigg( \int_{\tau_j}^t G_j(v) \: dv \bigg).
\end{equation}
Moreover, the upper bound holds for all $t \in [\xi_j, \tau_{j+1}] \cap [0, a_N T]$.

\item For all $j \geq k^* + 1$ and all $t \in [\tau_{j+1}, \gamma_{j+K}] \cap [0, a_N T]$, we have
\begin{equation}\label{prop23}
\frac{(1 - \delta)s}{\mu} \exp \bigg( \int_{\tau_{j+1}}^t G_j(v) \: dv \bigg) \leq X_j(t) \leq \frac{(1 + \delta)s}{\mu} \exp \bigg( \int_{\tau_{j+1}}^t G_j(v) \: dv \bigg).
\end{equation}

\item Let $K = \lfloor k_N/4 \rfloor$.  For all $j \geq k^* + 1$, we have
\begin{equation}\label{prop24}
X_j(t) \leq \frac{k_N^2 s}{\mu} \exp \bigg( \int_{\tau_{j+1}}^t G_j(v) \: dv \bigg)
\end{equation}
for all $t \in [\gamma_{j + K}, a_N T]$.  Also, for all $j \geq k^* + 1$ such that $\gamma_{j+\lceil 17 k_N \rceil} < a_N T$, we have $X_j(t) = 0$ for all $t \geq \gamma_{j+\lceil 17 k_N \rceil}$.

\item For all $j \geq k^*+1$, we have $$\frac{1}{N} \sum_{i=j+1}^{\infty} X_i(t) \leq C_2 e^{-s(\gamma_{j+1} - t)} + \frac{s}{N \mu}$$ for all $t \in [(4/s) \log k_N, \gamma_{j+1}] \cap [0, a_N T]$ and $$\frac{1}{N} \sum_{i=0}^{j-1} X_i(t) \leq C_2 e^{-s(t - \gamma_j)}$$ for all $t \in [\gamma_j, \gamma_{j+K}] \cap [0, a_N T].$

\item We have $\tau_{k^*+1} \leq 2a_N/k_N$.  Also, for all $j \geq k^*+1$ such that either $\tau_j + 2a_N/k_N \leq a_N T$ or $\tau_{j+1} \leq a_N T$, we have
\begin{equation}\label{tauspacing}
\frac{a_N}{3k_N} \leq \tau_{j+1} - \tau_j \leq \frac{2a_N}{k_N}.
\end{equation}
More precisely,
$$\int_{\tau_j/a_N}^{\tau_{j+1}/a_N} q(t) \: dt \leq \frac{1 + 2 \delta}{k_N}$$ and
$$\int_{\tau_j/a_N}^{\tau_{j+1}/a_N} (q(t) + \1_{\{t \in [1, \gamma_{k^* + 1}/a_N)\}}) \: dt \geq \frac{1- 2\delta}{k_N},$$
where $q$ is the function defined later in (\ref{qdef}).
\end{enumerate}
\end{Prop}

\begin{Rmk}\label{Jrem}
{\em Let
\begin{equation}\label{Jdef}
J = 3 k_N T + k^* + 1.
\end{equation}
As noted in Remark 3.7 of \cite{schI}, when (\ref{tauspacing}) holds, we have
$$\tau_J > \tau_J - \tau_{k^* + 1} \geq \frac{a_N}{3k_N} (J - (k^* + 1)) \wedge a_N T = a_N T,$$ and furthermore when the statement of part 1 of Proposition \ref{mainjprop} also holds, no individual of type $J+1$ or higher can appear until after time $a_N T$.}
\end{Rmk}

The next proposition contains some bounds related to the quantities $G_j(t)$ and $q_j$ that are important for the analysis that follows.  The first three parts of the proposition come from Lemma 8.8 of \cite{schI}.  The fourth part is part of Lemma 6.1 of \cite{schI}, and the fifth comes from Lemmas 8.25 and 8.26 in \cite{schI}.

\begin{Prop}\label{newGq}
There is a positive constant $C_3$, depending on $\eps$, $\delta$, and $T$, such that if $N$ is sufficiently large, then the following statements all hold for all $j$ such that $k^*+1 \leq j \leq J$ with probability at least $1 - \eps$:
\begin{enumerate}
\item If $\tau_j > a_N + 2a_N/k_N$ and $t \in [\tau_j, \tau_{j+1} \wedge a_N T]$, then $s(q_j - C_3) \leq G_j(t) \leq s (q_j + C_3)$.

\item If $t \in [\tau_j, \tau_{j+1} \wedge a_N T]$, then $(1 - 2 \delta) s k_N \leq G_j(t) \leq G_j(t) + \mu \leq (e + 2 \delta) s k_N$.

\item If $\tau_j \leq a_N T$, then $(1 - 2 \delta) k_N \leq q_j \leq (e + 2 \delta) k_N.$

\item If $\tau_{j+1} \leq a_N T$, then $\exp\big(\int_{\tau_j}^{\tau_{j+1}} G_j(v) \: dv\big) \leq 2s/\mu$.

\item If $j \geq k^* + 1 + K$, then
$$e^{-\int_{\tau_{j+1}}^u G_j(v) \: dv} \leq \left\{
\begin{array}{ll} e^{-sk_N(u - \tau_{j+1})/5} & \mbox{ if }u \in [\tau_{j+1}, \gamma_{j-K}] \cap [0, a_N T]  \\
(s/\mu)^{-k_N/241} & \mbox{ if }u \in [\gamma_{j-K}, \gamma_{j+K}] \cap [0, a_N T].
\end{array} \right.$$
\end{enumerate}
\end{Prop}

Let $\Lambda$ be the event that the six statements in Proposition \ref{mainjprop} and the five statements of Proposition \ref{newGq} all hold.  Note that the event $\Lambda$ depends $\eps$, $\delta$, $T$, and $N$.  Then Propositions \ref{mainjprop} and \ref{newGq} imply that
\begin{equation}\label{problam}
P(\Lambda) > 1 - 2 \eps
\end{equation}
if $N$ is sufficiently large.  We now define a random time $\zeta$, which we interpret as being the first time that one of the statements of Proposition \ref{mainjprop} or Proposition \ref{newGq} fails to hold.  Write ${\bf X}(t) = (X_0(t), X_1(t), \dots)$, and let $({\cal F}_t, t \geq 0)$ denote the natural filtration of the population process $({\bf X}(t), t \geq 0)$.  Then define $$\zeta = \inf\{t: P(\Lambda|{\cal F}_t) = 0\}.$$  Since Propositions \ref{mainjprop} and \ref{newGq} only describe the behavior of the process up to time $a_N T$, the event $\Lambda$ is equivalent to the event $\{\zeta > a_N T\}$, which in turn is equivalent to the event $\{\zeta = \infty\}$.  Note that the definition given here for $\zeta$ is not quite the same as the definition in \cite{schI} because in \cite{schI} some additional properties were listed that are not relevant for the present work, and some of the properties listed above were derived from others.  Nevertheless, the idea is the same in both papers.  Namely, if $t < \zeta$, then all of the properties specified in Propositions \ref{mainjprop} and \ref{newGq} hold through time $t$.

\subsection{Selective advantage of the fittest individuals}

The result below, which is Theorem 1.1 of \cite{schI}, gives an asymptotic result for the difference in fitness between the fittest individual in the population and an individual of average fitness.

\begin{Prop}\label{Qthm}
For $t \geq 0$, let
\begin{equation}\label{Qdef}
Q(t) = \max\{j: X_j(t) > 0\} - M(t).
\end{equation}  
Assume A1-A3 hold.  There is a unique bounded function $q: [0, \infty) \rightarrow [0, \infty)$ such that
\begin{equation}\label{qdef}
q(t) = \left\{
\begin{array}{ll} e^t & \mbox{ if }0 \leq t < 1  \\
\int_{t-1}^t q(u) \: du & \mbox{ if }t \geq 1.
\end{array} \right.
\end{equation}
If $S$ is a compact subset of $(0,1) \cup (1, \infty)$, then
\begin{equation}\label{mainQres}
\sup_{t \in S} \bigg| \frac{Q(a_N t)}{k_N} - q(t) \bigg| \rightarrow_p 0,
\end{equation}
where $\rightarrow_p$ denotes convergence in probability as $N \rightarrow \infty$.
\end{Prop}

The next proposition collects some properties of the function $q$.  All of these results are part of Lemma 7.2 of \cite{schI} except for (\ref{qlip}), which follows from (\ref{qbounds}) and the definition of $q$.

\begin{Prop}\label{Qlem}
The function $q$ defined in (\ref{qdef}) is continuous on $[0, 1) \cup (1, \infty)$, and $$\lim_{t \rightarrow \infty} q(t) = 2.$$
Also,
\begin{equation}\label{qbounds}
1 \leq q(t) \leq e \hspace{.2in}\mbox{ for all }t \geq 0
\end{equation}
and if $t < u$ with $1 \notin (t, u]$, then
\begin{equation}\label{qlip}
|q(u) - q(t)| \leq e(u-t).
\end{equation}
\end{Prop}

\subsection{A useful martingale}

Here we review the construction of a martingale that was central to the analysis in \cite{schI} and will be important again in the present paper.  As in \cite{schI}, let $F_j(t)$ be the fitness of a type $j$ individual at time $t$, which is $\max\{0, 1 + s(j - M(t))\}$, divided by the sum of the fitnesses of all individuals in the population at time $t$, which is $N$ if every individual's fitness is strictly positive.  Remark \ref{Jrem} and assumption A3 imply that if $N$ is sufficiently large, then every individual's fitness is strictly positive at time $t$ for all $t < \zeta$, in which case
\begin{equation}\label{Fjeq}
F_j(t) = \frac{1 + s(j - M(t))}{N}.
\end{equation}
To define birth and death rates, we follow closely the discussion in \cite{schI} and observe that there are three ways that the number of type $j$ individuals could change at time $t$:
\begin{enumerate}
\item Each type $j-1$ individual acquires a $j$th mutation at rate $\mu$.  Therefore, at time $t$, the rate at which a type $j$ individual appears due to a mutation is $\mu X_{j-1}(t-)$, where we adopt the convention that $X_{-1}(t) = 0$ for all $t \geq 0$ so that our formulas are valid when $j = 0$.

\item The number of type $j$ individuals could increase by one at time $t$ due to a birth.  This happens if one of the $N - X_j(t-)$ other individuals dies at time $t$, which happens at rate $N - X_j(t-)$ because each individual dies at rate one, and if the new individual born has type $j$, which happens with probability $X_j(t-) F_j(t-)$.  Therefore, we define the birth rate
\begin{equation}\label{Bjeq}
B_j(t) = (N - X_j(t)) F_j(t).
\end{equation}

\item The number of type $j$ individuals could decrease at time $t$ due to a mutation or death.  The rate at which one of the type $j$ individuals becomes type $j+1$ due to a mutation is $\mu X_j(t-)$.  Death events that reduce the number of type $j$ individuals happen at rate $X_j(t-)(1 - X_j(t-)F_j(t-))$ because there are $X_j(t-)$ type $j$ individuals each dying at rate one, and when a death occurs, the probability that the new individual born does not have type $j$ is $1 - X_j(t-)F_j(t-)$.  Therefore, we define the death rate
\begin{equation}\label{Djeq}
D_j(t) = 1 + \mu - X_j(t)F_j(t).
\end{equation}
\end{enumerate}
For all $t \geq 0$ and $j \in \Z^+$, let $$G_j^*(t) = B_j(t) - D_j(t).$$  One can easily check that whenever (\ref{Fjeq}) holds, we have $G_j^*(t) = G_j(t)$.  Also, as shown in section 5.2 of \cite{schI}, whenever (\ref{Fjeq}) holds and $j \leq J$, we can see, using assumption A3, that for sufficiently large $N$,
\begin{equation}\label{BD3}
B_j(t) + D_j(t) = \frac{(N - 2 X_j(t))(1 + s(j - M(t)))}{N} + 1 + \mu \leq 2 + sJ + \mu \leq 3.
\end{equation}

The result below is Proposition 4.1 of \cite{schI}.  The martingale defined in this proposition is similar to the one obtained in section 4 of \cite{dm11}.

\begin{Prop}\label{Zmart}
For all $t \geq 0$ and $j \in \Z^+$, let
\begin{equation}\label{Zjdef}
Z_j(t) = e^{-\int_0^t G_j^*(v) \: dv} X_j(t) - \int_0^t \mu X_{j-1}(u) e^{-\int_0^u G_j^*(v) \: dv} \: du - X_j(0).
\end{equation}
Then $(Z_j(t), t \geq 0)$ is a mean zero martingale with
$$\Var(Z_j(t)) = E \bigg[ \int_0^t e^{-2 \int_0^u G_j^*(v) \: dv} (\mu X_{j-1}(u) + B_j(u) X_j(u) + D_j(u) X_j(u)) \: du \bigg].$$  
\end{Prop}

We will sometimes need to apply the result of Proposition \ref{Zmart} to only a subset of the type $j$ individuals in the population.  If $\kappa$ and $\gamma$ are stopping times with respect to $({\cal F}_t, t \geq 0)$ such that $0 \leq \kappa \leq \gamma$, then for $t \geq 0$ and $j \in \Z^+$, let $X_j^{\kappa, \gamma}(t)$ be the number of type $j$ individuals in the population at time $t$ that are descended from individuals that acquired a $j$th mutation during the time interval $(\kappa, \gamma]$.  Let $B_j^{\kappa, \gamma}(t)$ and $D_j^{\kappa, \gamma}(t)$ denote the expressions on the the right-hand sides of (\ref{Bjeq}) and (\ref{Djeq}) with $X_j^{\kappa, \gamma}(t)$ in place of $X_j(t)$.  The result below is Corollary 4.4 of \cite{schI}.

\begin{Cor}\label{ZmartCor4}
Let $\kappa$ and $\gamma$ be stopping times with $\kappa \leq \gamma$.  For $t \geq \kappa$, let $$Z_j^{\kappa, \gamma}(t) = e^{-\int_{\kappa}^t G_j^*(v) \: dv} X_j^{\kappa, \gamma}(t) - \int_{\kappa}^{t \wedge \gamma} \mu X_{j-1}(u) e^{-\int_{\kappa}^u G_j^*(v) \: dv} \: du.$$  Then $(Z_j^{\kappa, \gamma}(\kappa + t), t \geq 0)$ is a mean zero martingale and
\begin{align}
&\Var(Z_j^{\kappa, \gamma}(\kappa + t)|{\cal F}_{\kappa}) \nonumber \\
&\hspace{.1in}= E \bigg[ \int_{\kappa}^{\kappa + t} e^{-2 \int_{\kappa}^u G_j^*(v) \: dv} (\mu X_{j-1}(u) \1_{u \in (\kappa, \gamma]} + B_j^{\kappa, \gamma}(u)X_j^{\kappa, \gamma}(u) + D_j^{\kappa, \gamma}(u)X_j^{\kappa, \gamma}(u)) \: du \bigg| {\cal F}_{\kappa} \bigg].  \nonumber
\end{align}
Furthermore, if $\tau$ is a stopping time with $\kappa \leq \tau$, then $(Z_j^{\kappa, \gamma}((\kappa + t) \wedge \tau), t \geq 0)$ is a mean zero martingale, and $\Var(Z_j^{\kappa, \gamma}((\kappa + t) \wedge \tau)|{\cal F}_{\kappa})$ is obtained by replacing $\kappa + t$ with $(\kappa + t) \wedge \tau$ in the integral above.
\end{Cor}

Finally, suppose $\kappa$ is a stopping time with respect to $({\cal F}_t, t \geq 0)$ and $S$ is a set of type $j$ individuals alive at time $\kappa$.  Then for $t \geq \kappa$, let $X^S_j(t)$ be the number of type $j$ individuals in the population at time $t$ that are descended from one of the individuals in the set $S$, and let $B^S_j(t)$ and $D^S_j(t)$ the expressions on the right-hand sides of (\ref{Bjeq}) and (\ref{Djeq}) with $X^S_j(t)$ in place of $X_j(t)$.  Then, the same reasoning used to establish Proposition \ref{Zmart} and Corollary \ref{ZmartCor4} yields the following corollary.

\begin{Cor}\label{ZmartCor3}
Let $\kappa$ be a stopping time, and let $S$ be a set of type $j$ individuals in the population at time $\kappa$.
For $t \geq \kappa$, let $$Z^S_j(t) = e^{-\int_{\kappa}^t G_j^*(v) \: dv} X^S_j(t) - X^S_j(\kappa).$$  Then $(Z^S_j(\kappa + t), t \geq 0)$ is a mean zero martingale and
 $$\Var(Z^S_j(\kappa + t)|{\cal F}_{\kappa}) = E \bigg[  \int_{\kappa}^{\kappa + t} e^{-2 \int_{\kappa}^u G_j^*(v) \: dv} (B^S_j(u) X^S_j(u) + D^S_j(u) X^S_j(u)) \: du \bigg| {\cal F}_{\kappa} \bigg].$$
Furthermore, if $\tau$ is a stopping time with $\kappa \leq \tau$, then $(Z_j^S((\kappa + t) \wedge \tau), t \geq 0)$ is a mean zero martingale, and $\Var(Z_j^S((\kappa + t) \wedge \tau)|{\cal F}_{\kappa})$ is obtained by replacing $\kappa + t$ with $(\kappa + t) \wedge \tau$ in the integral above.
\end{Cor}

\begin{Rmk}\label{randj}
{\em By the Strong Markov Property of the population process $({\bf X}(t), t \geq 0)$, the results of Corollaries \ref{ZmartCor4} and \ref{ZmartCor3} hold even when the type $j$ is random, as long as $j$ is ${\cal F}_{\kappa}$-measurable.}
\end{Rmk}

\section{Tracing the ancestral lines back to time $a_N(T-1)$}

The rest of the paper is devoted to the proof of Theorem \ref{boszthm}.  Throughout the proof, we will fix $\eps > 0$, $\delta > 0$, $t > 0$, and $T > t + 2$.  We will also assume that $\eps < 1$ and
\begin{equation}\label{deldef}
\delta < \max \bigg\{ \frac{1}{100}, \frac{T - (t + 2)}{40T}, \frac{1}{19T}, \eps^3 \bigg\}.
\end{equation}
The event $\Lambda$ is defined as in section 4 for these choices of $\eps$, $\delta$, and $T$, and for the constants $C_1$, $C_2$, and $C_3$ from Propositions \ref{mainjprop} and \ref{newGq}.

We sample $n$ individuals at random from the population at time $a_N T$ and randomly label these individuals with the integers $1, \dots, n$.  We then trace the ancestral lines of these individuals back to time $a_N(T - (t + 1))$.  Recall that if $0 \leq u \leq t-1$, then $\Pi_N(u)$ is the partition of $\{1, \dots, n\}$ such that $i$ and $j$ are in the same block of $\Pi_N(u)$ if and only if the individuals in the sample labelled $i$ and $j$ have the same ancestor at time $a_N(T - u)$. 

For $1 \leq i \leq n$ and $0 \leq u \leq a_N T$, let $U_i(u)$ be the number of mutations carried by the individual at time $u$ that is the ancestor of the individual labelled $i$ at time $a_N T$.  For $1 \leq i \leq n$ and $1 \leq j \leq U_i(a_N T)$, let
\begin{equation}\label{Vijdef}
V_{i,j} = \inf\{u: U_i(u) = j\}
\end{equation}
be the time when the $j$th mutation appears on the $i$th lineage.  For $i,j \in \{1, \dots, n\}$, let
\begin{equation}\label{Tijdef}
T_{i,j} = \sup\{u: \mbox{ the $i$th and $j$th sampled individuals have the same ancestor at time $u$}\}
\end{equation}
denote the coalescence time of $i$ and $j$.

Throughout the rest of the paper, we use $C$ to denote a positive constant that does not depend on $\delta$, $\eps$, or $T$ but whose value may change from line to line.  Recall that the numbered constants $C_1$, $C_2$, and $C_3$ do depend on $\delta$, $\eps$, and $T$.  We will say that a statement holds ``for sufficiently large $N$" if there is a positive integer $N_0$, possibly depending on $\eps$, $\delta$, and $T$, such that the statement holds for all $N \geq N_0$.

\subsection{The types of the individuals sampled at time $a_N T$}

Part 5 of Proposition \ref{mainjprop} implies that, between times $\gamma_j$ and $\gamma_{j+1}$, the fraction of individuals in the population having type $j$ is very close to one, except for times very close to the boundary of this interval.  Consequently, when we take a sample from the population at time $a_N T$, typically either all individuals will have the same type, or else all individuals will have one of two types.  The result below is a weaker form of this statement.

\begin{Lemma}\label{sameL}
Let
\begin{equation}\label{Ldef}
L = \inf\bigg\{j: \tau_j \geq a_N(T - 1) - \frac{3 a_N}{k_N} \bigg\}.
\end{equation}
Then
$$\lim_{N \rightarrow \infty} P \big( \Lambda \cap \big\{U_i(a_N T) \notin \{L, L+1, \dots, L+9\} \mbox{ for some }i \in \{1, \dots, n\} \big\} \big) = 0.$$
\end{Lemma}

\begin{proof}
It follows from equation (\ref{tauspacing}) that on the event $\Lambda$, we have $\tau_L \leq a_N(T - 1) - a_N/k_N$ and $\tau_{L+10} \geq a_N(T-1) + a_N/3k_N$.  Therefore, using (\ref{gammajdef}), on $\Lambda$ we have $\gamma_L \leq a_N T - a_N/k_N$ and $\gamma_{L+10} \geq a_N T + a_N/3k_N$.  Therefore, by part 5 of Proposition \ref{mainjprop}, on $\Lambda$ we have
\begin{equation}\label{bigi}
\frac{1}{N} \sum_{\ell=L+10}^{\infty} X_{\ell}(a_N T) \leq C_2 e^{-s(\gamma_{L+10} - a_N T)} + \frac{s}{N \mu} \leq C_2 \bigg( \frac{s}{\mu} \bigg)^{-1/3k_N} + \frac{s}{N \mu},
\end{equation}
which tends to zero as $N \rightarrow \infty$ because $(1/3k_N) \log(s/\mu) \rightarrow \infty$ as $N \rightarrow \infty$ by assumption A2, and $s/(N \mu) \rightarrow 0$ as $N \rightarrow \infty$ by (\ref{musN}).  Likewise, by part 5 of Proposition \ref{mainjprop}, on $\Lambda$ we have
\begin{equation}\label{smalli}
\frac{1}{N} \sum_{\ell=0}^{L-1} X_{\ell}(a_N T) \leq C_2 e^{-s(a_N T - \gamma_L)} \leq C_2 e^{-s a_N/k_N} = C_2 \bigg( \frac{s}{\mu} \bigg)^{-1/k_N},
\end{equation}
which tends to zero as $N \rightarrow \infty$.  Because the expressions in (\ref{bigi}) and (\ref{smalli}) both tend to zero as $N \rightarrow \infty$, we conclude that on $\Lambda$, the fraction of individuals in the population at time $a_N T$ having between $L$ and $L + 9$ mutations tends to one as $N \rightarrow \infty$.  Because the $n$ individuals are sampled at random from the population, the result follows.
\end{proof}

\subsection{The types of the ancestors at time $a_N(T - 1)$}

Lemma \ref{sameL} implies that with high probability all individuals sampled at time $a_N T$ will have between $L$ and $L + 9$ mutations.  Lemma \ref{noearlymut} below shows that for $\ell \in \{L, L + 1, \dots, L + 9\}$, with high probability the type $\ell$ individuals in the sample will all be descended from type $\ell$ individuals at time $\tau_{\ell+1}$.

\begin{Lemma}\label{noearlymut}
We have $$\lim_{N \rightarrow \infty} P\big(\Lambda \cap \big\{ U_i(\tau_{U_i(a_N T) + 1}) \neq U_i(a_N T) \mbox{ for some }i \in \{1, \dots, n\} \big\} \big) = 0.$$
\end{Lemma}

\begin{proof}
Choose $\ell \in \{L, L + 1, \dots, L + 9\}$.  Recall from Corollary \ref{ZmartCor4} that $X_{\ell}^{\tau_{\ell+1}, a_N T}(a_N T)$ denotes the number of type $\ell$ individuals at time $a_N T$ that are descended from an individual that got its $\ell$th mutation during the time interval $(\tau_{\ell+1}, a_N T]$.  Equivalently, this is the number of type $\ell$ individuals at time $a_N T$ whose ancestor in the population at time $\tau_{\ell+1}$ does not have type $\ell$.  Because each individual in the population at time $a_N T$ has probability $n/N$ of being in the sample, we therefore have
\begin{equation}\label{condmut}
P \big( \Lambda \cap \big\{ U_i(\tau_{\ell + 1}) \neq U_i(a_N T) = \ell \mbox{ for some }i \in \{1, \dots, n\} \big\} \big| {\cal F}_{a_N T} \big) \leq \frac{n X_{\ell}^{\tau_{\ell+1}, a_N T}(a_N T) \1_{\Lambda}}{N}.
\end{equation}
It suffices to show that the expected value of the right-hand side of (\ref{condmut}) tends to zero as $N \rightarrow \infty$.

By Corollary \ref{ZmartCor4} and Remark \ref{randj}, on $\Lambda$,
\begin{equation}\label{nem1}
e^{-\int_{\tau_{\ell+1}}^{a_N T} G_{\ell}(v) \: dv} X_{\ell}^{\tau_{\ell+1}, a_NT}(a_N T) = \int_{\tau_{\ell+1}}^{a_N T} \mu X_{\ell-1}(u) e^{-\int_{\tau_{\ell+1}}^u G_{\ell}(v) \: dv} \: du + Z_{\ell}^{\tau_{\ell+1}, a_N T}(a_N T),
\end{equation}
where $Z_{\ell}^{\tau_{\ell+1}, a_N T}(\tau_{\ell+1} + t, t \geq 0)$ is a mean zero martingale.  Note that (\ref{tauspacing}) implies that on $\Lambda$, we have $\gamma_{\ell-1+K} > a_N T$ if $N$ is sufficiently large, and therefore from (\ref{prop23}) and from part 4 of Proposition \ref{newGq}, we get for $u \in [\tau_{\ell+1}, a_N T]$,
\begin{align*}
\mu X_{\ell-1}(u) e^{-\int_{\tau_{\ell+1}}^u G_{\ell}(v) \: dv} &\leq (1 + \delta) s e^{\int_{\tau_{\ell}}^u G_{\ell-1}(v) \: dv} e^{-\int_{\tau_{\ell+1}}^u G_{\ell}(v) \: dv} \\
&= (1 + \delta) s e^{\int_{\tau_{\ell}}^{\tau_{\ell+1}} G_{\ell}(v) \: dv} e^{-s(u - \tau_{\ell})} \\
&\leq \frac{2(1 + \delta) s^2}{\mu} \: e^{-s(u - \tau_{\ell})}.
\end{align*}
It follows that on $\Lambda$, if $N$ is sufficiently large,
\begin{equation}\label{nem2}
\int_{\tau_{\ell+1}}^{a_N T} \mu X_{\ell-1}(u) e^{-\int_{\tau_{\ell+1}}^u G_{\ell}(v) \: dv} \: du \leq \frac{2(1 + \delta)s}{\mu} \: e^{-s(\tau_{\ell+1} - \tau_{\ell})}.
\end{equation}
Now on $\Lambda$, by (\ref{tauspacing}), we have
\begin{equation}\label{nem4}
e^{-s(\tau_{\ell+1} - \tau_{\ell})} \leq e^{-a_N s/3k_N} = \bigg( \frac{s}{\mu} \bigg)^{-1/3k_N}.
\end{equation}
Also, on $\Lambda$ we have $a_N T \in [\tau_{\ell+1}, \gamma_{\ell + K}]$ if $N$ is sufficiently large and therefore, by (\ref{prop23}),
\begin{equation}\label{nem3}
e^{-\int_{\tau_{\ell+1}}^{a_N T} G_{\ell}(v) \: dv} X_{\ell}(a_N T) \geq \frac{(1 - \delta) s}{\mu}.
\end{equation}
Combining (\ref{nem1}), (\ref{nem2}), (\ref{nem4}), and (\ref{nem3}), and using that $X_{\ell}(a_N T) \leq N$, we get that for sufficiently large $N$,
\begin{align}\label{nem5}
E \bigg[ \frac{X_{\ell}^{\tau_{\ell+1}, a_N T}(a_N T) \1_{\Lambda}}{N} \bigg] &\leq E \bigg[ \frac{e^{-\int_{\tau_{\ell+1}}^{a_N T} G_{\ell}(v) \: dv} X_{\ell}^{\tau_{\ell+1}, a_N T}(a_N T) \1_{\Lambda}}{e^{-\int_{\tau_{\ell+1}}^{a_N T} G_{\ell}(v) \: dv} X_{\ell}(a_N T)} \bigg] \nonumber \\
&\leq E \bigg[ \frac{2(1 + \delta)(s/\mu)^{1 - 1/3k_N} + Z_{\ell}^{\tau_{\ell + 1}, a_N T}(a_N T)}{(1 - \delta)(s/\mu)} \bigg] \nonumber \\
&= \frac{2(1 + \delta)}{1 - \delta} \bigg( \frac{s}{\mu} \bigg)^{-1/3k_N}.
\end{align}
Because $(1/3k_N) \log (s/\mu) \rightarrow \infty$ as $N \rightarrow \infty$ by assumption A2, the expression on the right-hand side of (\ref{nem5}) tends to zero as $N \rightarrow \infty$.  The lemma follows by taking expectations of both sides in (\ref{condmut}).
\end{proof}

\subsection{Coalescence between times $a_N(T-1)$ and $a_N T$}

Our next goal is to show that for $\ell \in \{L, L+1, \dots, L+9\}$, the type $\ell$ individuals in the sample at time $a_N T$ all come from distinct ancestors at time $\tau_{\ell+1}$.  That is, the lineages do not coalesce as they are traced back from time $a_N T$ to time $\tau_{\ell + 1}$.  The precise statement is given in Lemma \ref{noearlycoal} below.  Because $\gamma_{\ell + 1} - \tau_{\ell + 1} = a_N$, this observation is very close to the statement (\ref{part1}) that none of the lineages coalesce when they are traced back $a_N$ time units.  We first establish the following preliminary lemma, which is more general than what is needed for the proof of Lemma \ref{noearlycoal} but will also be used later to prove Lemma \ref{coalj2}.

\begin{Lemma}\label{maxXj}
Suppose $k^*+1+K \leq j \leq J$.  Randomly label the type $j$ individuals at time $\tau_{j+1}$ by the integers $1, 2, \dots, \lceil s/\mu \rceil$.  For $t \geq \tau_{j+1}$, let $X_j^i(t)$ denote the number of type $j$ individuals at time $t$ that are descended from the individual labelled $i$ at time $\tau_{j+1}$.  Let $\gamma = \gamma_{j + K} \wedge \zeta \wedge a_N T$, and let
\begin{equation}\label{Mijdef}
R_{i,j} = \sup_{t \in [\tau_{j+1}, \gamma)} \frac{X_j^i(t)}{X_j(t)}.
\end{equation}
Then
\begin{equation}\label{sumEM}
E \bigg[ \sum_{i=1}^{\lceil s/\mu \rceil} R_{i,j}^2 \bigg] \leq \frac{C \mu}{s^2 k_N}.
\end{equation}
\end{Lemma}

\begin{proof}
By Corollary \ref{ZmartCor3} applied when $S$ consists only of the individual labelled $i$ at time $\tau_{j+1}$, for $i = 1, 2, \dots, \lceil s/\mu \rceil$ and $t \geq \tau_{j+1}$, we have
\begin{equation}\label{Xieq}
X_j^i(t \wedge \gamma) = e^{\int_{\tau_{j+1}}^{t \wedge \gamma} G_j(v) \: dv}(1 + Z_{j}^i(t)),
\end{equation}
where $(Z_j^i(\tau_{j+1} + t), t \geq 0)$ is a mean zero martingale.  
Now suppose $t \in [\tau_{j+1}, \gamma)$.  Using (\ref{Xieq}) and (\ref{prop23}),
$$\bigg( \frac{X_j^i(t)}{X_j(t)} \bigg)^2 \leq \frac{\mu^2}{(1 - \delta)^2 s^2} (1 + Z_j^i(t))^2.$$
Taking the supremum of both sides over $t \in [\tau_{j+1}, \gamma)$, then taking expectations and using that $(a + b)^2 \leq 2a^2 + 2b^2$, we get
\begin{equation}\label{EMi}
E[R_{i,j}^2] \leq \frac{2 \mu^2}{(1 - \delta)^2 s^2} \bigg(1 + E \bigg[\sup_{t \in [\tau_{j+1}, \gamma)} (Z_j^i(t))^2 \bigg] \bigg).
\end{equation}
By the $L^2$ Maximum Inequality for martingales, Corollary \ref{ZmartCor3}, and the reasoning used to derive (\ref{BD3}),
$$E \bigg[\sup_{t \in [\tau_{j+1}, \gamma)} (Z_j^i(t))^2 \bigg| {\cal F}_{\tau_{j+1}} \bigg] \leq 4 E \bigg[ \int_{\tau_{j+1}}^{\gamma} e^{-2 \int_{\tau_{j+1}}^u G_{j}(v) \: dv} \cdot 3 X_{j}^i(u) \: du \bigg| {\cal F}_{\tau_{j+1}} \bigg].$$  Combining this result with (\ref{Xieq}) gives
\begin{equation}\label{varZi}
E \bigg[\sup_{t \in [\tau_{j+1}, \gamma)} (Z_j^i(t))^2 \bigg| {\cal F}_{\tau_{j+1}} \bigg] \leq 12 E \bigg[ \int_{\tau_{j+1}}^{\gamma} e^{-\int_{\tau_{j+1}}^u G_{j}(v) \: dv} (1 + Z_{j}^{i}(u)) \: du \bigg| {\cal F}_{\tau_{j+1}} \bigg].
\end{equation}
Note that $1 + Z_j^i(u) \geq 0$ for all $u \in [\tau_{j+1}, \gamma)$ by (\ref{Xieq}).  Therefore,
by part 5 of Proposition \ref{newGq} and the fact that $(Z_j^i(\tau_{j+1}+t), t \geq 0)$ is a mean zero martingale,
\begin{align}\label{ipt1}
&E \bigg[ \int_{\tau_{j+1}}^{\gamma_{j-K} \wedge \gamma} e^{-\int_{\tau_{j+1}}^u G_{j}(v) \: dv} (1 + Z_{j}^i(u)) \: du \bigg| {\cal F}_{\tau_{j+1}} \bigg] \nonumber \\
&\hspace{1.5in}\leq E \bigg[ \int_{\tau_{j+1}}^{\infty} e^{-sk_N(u - \tau_{j+1})/5} (1 + Z_{j}^i(u)) \: du \bigg| {\cal F}_{\tau_{j+1}} \bigg] \nonumber \\
&\hspace{1.5in}= \int_{\tau_{j+1}}^{\infty} e^{-sk_N(u - \tau_{j+1})/5} \: du \nonumber \\
&\hspace{1.5in}= \frac{5}{sk_N}.
\end{align}
Also, using part 5 of Proposition \ref{newGq} again and the fact that $\gamma - \gamma_{j - K} \wedge \gamma \leq (2a_N/k_N)(2K) \leq a_N$ for sufficiently large $N$ by (\ref{tauspacing}),
\begin{align}\label{ipt2}
&E \bigg[ \int_{\gamma_{j-K} \wedge \gamma}^{\gamma} e^{-\int_{\tau_{j+1}}^u G_{j}(v) \: dv} (1 + Z_{j}^i(u)) \: du \bigg| {\cal F}_{\tau_{j+1}} \bigg] \nonumber \\
&\hspace{1.5in} \leq E \bigg[ \int_{\gamma_{j - K} \wedge \gamma}^{\gamma} \bigg( \frac{s}{\mu} \bigg)^{-k_N/241} (1 + Z_{j}^{(i)}(u)) \: du \bigg| {\cal F}_{\tau_{j+1}} \bigg] \nonumber \\
&\hspace{1.5in} \leq a_N \bigg( \frac{s}{\mu} \bigg)^{-k_N/241}.
\end{align}
Because $s k_N \cdot a_N (s/\mu)^{-k_N/241} \rightarrow 0$ as $N \rightarrow \infty$, as can easily be seen by taking logarithms, equations (\ref{varZi}), (\ref{ipt1}), and (\ref{ipt2}) imply that
$$E \bigg[\sup_{t \in [\tau_{j+1}, \gamma)} (Z_j^i(t))^2 \bigg| {\cal F}_{\tau_{j+1}} \bigg] \leq 12 \bigg( \frac{5}{sk_N} + a_N \bigg( \frac{s}{\mu} \bigg)^{-k_N/241} \bigg) \leq \frac{C}{sk_N}$$ for sufficiently large $N$.  Therefore, using (\ref{EMi}), we get for sufficiently large $N$,
$$E \bigg[ \sum_{i=1}^{\lceil s/\mu \rceil} R_{i,j}^2 \bigg] \leq \bigg\lceil \frac{s}{\mu} \bigg\rceil \frac{2 \mu^2}{(1 - \delta)^2 s^2} \cdot \bigg(1 + \frac{C}{sk_N} \bigg).$$  The result follows because $sk_N \rightarrow 0$ as $N \rightarrow \infty$ by assumption A3.
\end{proof}

Note that in the statement of Lemma \ref{noearlycoal} below, we consider only the lineages labelled $1$ and $2$ to simplify notation.  This is sufficient because individuals are sampled uniformly at random.  To bound the probability that the event in question occurs for some pair of lineages, we may simply multiply the probability that the event occurs for the lineages $1$ and $2$ by $\binom{n}{2}$.

\begin{Lemma}\label{noearlycoal}
We have
$$\lim_{N \rightarrow \infty} P\big(\Lambda \cap \{ U_1(a_N T) = U_2(a_N T) = \ell \mbox{ and }T_{1,2} \geq \tau_{\ell+1} \mbox{ for some }\ell \} \big) = 0.$$
\end{Lemma}

\begin{proof}
We know from Lemma \ref{noearlymut} that with probability tending to one as $N \rightarrow \infty$, on $\Lambda$ all type $\ell$ individuals sampled at time $a_N T$ have type $\ell$ ancestors at time $\tau_{\ell+1}$.
Therefore, it suffices to show that 
\begin{align}\label{probtobound}
&\lim_{N \rightarrow \infty} P\big(\Lambda \cap \{ U_1(a_N T) = U_2(a_N T) = U_1(\tau_{\ell+1}) = U_2(\tau_{\ell+1}) = \ell \nonumber \\
&\hspace{3in} \mbox{ and }T_{1,2} \geq \tau_{\ell+1} \mbox{ for some }\ell \} \big) = 0.
\end{align}
That is, we need to show it is unlikely that the first two individuals in the sample are both type $\ell$ individuals that are descended from the same type $\ell$ individual at time $\tau_{\ell+1}$. 

Randomly label the type $\ell$ individuals at time $\tau_{\ell+1}$ by the integers $1, 2, \dots, \lceil s/\mu \rceil$.  Let $X^i_{\ell}(t)$ denote the number of type $\ell$ individuals at time $t$ descended from the $i$th type $\ell$ individual in the population at time $\tau_{\ell+1}$.  Since each individual at time $a_N T$ is equally likely to be sampled,
\begin{align}\label{ptb1}
&P\big(\Lambda \cap \{ U_1(a_N T) = U_2(a_N T) = U_1(\tau_{j+1}) = U_2(\tau_{j+1}) = \ell\} \cap \{T_{1,2} \geq \tau_{\ell+1} \} \big| {\cal F}_{a_N T} \big) \nonumber \\
&\hspace{3.2in} =\sum_{i=1}^{\lceil s/\mu \rceil} \frac{X_{\ell}^i(a_N T)(X_{\ell}^i(a_N T) - 1) \1_{\Lambda}}{N (N-1)}.
\end{align}

By Lemma \ref{sameL}, it suffices to consider $\ell \in \{L, L+1, \dots, L+9\}$.
Part 6 of Proposition \ref{mainjprop} implies that on $\Lambda$, we have $\tau_{k^*+1+K} \leq 2a_N(K+1)/k_N$, and therefore $L \geq k^*+1+K$ for sufficiently large $N$.  Also, in view of (\ref{tauspacing}), on $\Lambda$ we have $\gamma_{L+9+K} \geq a_N T$.  Therefore, applying Lemma \ref{maxXj}, and noting that the probability of a change in the population at exactly time $a_N T$ is zero, for each fixed positive integer $\ell$ we have
\begin{equation}\label{ptb2}
E \bigg[ \sum_{i=1}^{\lceil s/\mu \rceil} \frac{X_{\ell}^i(a_N T)(X_{\ell}^i(a_N T) - 1)\1_{\{L \leq \ell \leq L+9\} \cap \Lambda}}{N (N-1)} \bigg] \leq E \bigg[ \sum_{i=1}^{\lceil s/\mu \rceil} R_{i,\ell}^2 \bigg] \leq \frac{C \mu}{s^2 k_N}.
\end{equation}
Taking expectations of both sides of (\ref{ptb1}) and then using (\ref{ptb2}) and the fact that $L + 9 \leq J$ on $\Lambda$ by Remark \ref{Jrem}, we get that the probability in (\ref{probtobound}) is bounded above by $C J \mu/(s^2 k_N)$, which tends to zero as $N \rightarrow \infty$ by (\ref{musN}).  Thus, (\ref{probtobound}) holds, which implies the result of the lemma.
\end{proof}

\begin{Rmk}\label{LL9Rmk}
{\em It follows from Lemmas \ref{sameL} and \ref{noearlymut} that with probability tending to one as $N \rightarrow \infty$, we have $U_i(\tau_{L+10}) = U_i(a_N T)$ for all $i \in \{1, \dots, n\}$.  Because individuals in the population model inherit all of their parents mutations, two lineages can only coalesce if they have the same type.  That is, we must have $U_i(T_{i,j}) = U_j(T_{i,j})$ for $i,j \in \{1, \dots, n\}$.  It therefore follows from Lemma \ref{noearlycoal} that with probability tending to one as $N \rightarrow \infty$, no lineages coalesce as they are traced back from time $a_N T$ to time $\tau_{L + 10}$.  The fact that the probability of coalescence between times $\tau_L$ and $\tau_{L+10}$ tends to zero as $N \rightarrow \infty$, which would imply (\ref{part1}), will be established later.}
\end{Rmk}

\section{Tracing the ancestral lines between times $\tau_j$ and $\tau_{j+1}$}

Lemmas \ref{noearlymut} and \ref{noearlycoal} show that the type $\ell$ individuals in the sample at time $a_N T$ are typically descended from distinct type $\ell$ ancestors at time $\tau_{\ell+1}$.  In this subsection, we consider tracing these ancestral lines back further in time.  In particular, we focus on what happens when lineages are traced back from time $\tau_{j+1}$ to $\tau_j$.  We establish that with high probability, type $j$ individuals at time $\tau_{j+1}$ are descended from type $j-1$ individuals at time $\tau_j$, and lineages will only coalesce when many type $j$ lineages are traced back to an individual that acquired its $j$th mutation before the time $\xi_j$ defined in (\ref{xijdef}).

\subsection{Approximating $\tau_j$ by the fixed time $\tau_j^*$}

We define here some fixed times $\tau_j^*$ that approximate the random times $\tau_j$.  Let $\tau_{k^*+1}^* = 0$.  For integers $j \geq k^* + 1$, let
\begin{equation}\label{tjstardef}
\tau_{j+1}^* = \tau_j^* + \frac{a_N}{k_N q(\tau_j^*/a_N)},
\end{equation}
where $q$ is the function defined in Proposition \ref{qdef}.  Because $1 \leq q(u) \leq e$ for all $u \geq 0$ by Proposition \ref{Qlem}, we have
\begin{equation}\label{tauspacing2}
\frac{a_N}{e k_N} \leq \tau_{j+1}^* - \tau_j^* \leq \frac{a_N}{k_N}.
\end{equation}
For $u \in (0, T]$, let $j^*(u) = \max\{j: \tau_j^* \leq a_N u\}$ and $j'(u) = \max\{j: \tau_j \leq a_N u\}$.  The lemma below shows that $\tau_j^*$ is a good approximation to $\tau_j$.

\begin{Lemma}\label{taustar}
Fix $u \in (0, T]$.  On the event $\{\zeta > a_N u\}$, we have
\begin{equation}\label{tslem1}
|j^*(u) - j'(u)| \leq 9 \delta T k_N.
\end{equation}
Likewise, let $j \in \{k^* + 1, \dots, J\}$.  On the event $\{\tau_j < \zeta \wedge a_N T\}$, we have
\begin{equation}\label{tslem2}
|\tau^*_j - \tau_j| \leq 10 \delta a_N T.
\end{equation}
\end{Lemma}

\begin{proof}
Suppose $j \in \{k^* + 1, \dots, J\}$ and $\tau_{j+1} < \zeta \wedge a_N T$.  By part 6 of Proposition \ref{mainjprop},
\begin{equation}\label{tsu0}
\frac{1 - 2 \delta}{k_N} - \int_{\tau_j/a_N}^{\tau_{j+1}/a_N} \1_{\{u \in [1, \gamma_{k^*+1}/a_N)\}} \: du \leq \int_{\tau_j/a_N}^{\tau_{j+1}/a_N} q(u) \: du \leq \frac{1 + 2 \delta}{k_N}.
\end{equation}
Therefore, if $u \in (0, T]$ and $\zeta > a_N u$, then, using that $\tau_{k^*+1}/a_N \leq 2/k_N$ and $u - \tau_{j'(u)}/a_N \leq 2/k_N$ by part 6 of Proposition \ref{mainjprop} and that $q(v) \leq e$ for all $v \in [0, u]$ by Proposition \ref{Qlem}, we have
\begin{align}\label{tsu1}
\int_0^u q(v) \: dv &\leq \int_0^{\tau_{k^*+1}/a_N} q(u) \: du + \frac{(1 + 2 \delta)(j'(u) - (k^* + 1))}{k_N} + \int_{\tau_{j'(u)}/a_N}^u q(v) \: dv \nonumber \\
&\leq \frac{(1 + 2 \delta)(j'(u) - (k^* + 1))}{k_N} + \frac{4e}{k_N}.
\end{align}
Likewise, using that $\gamma_{k^*+1}/a_N - 1 \leq 2/k_N$ by part 6 of Proposition \ref{mainjprop}, the lower bound in (\ref{tsu0}) implies that
\begin{equation}\label{tsl1}
\int_0^u q(v) \: dv \geq \frac{(1 - 2 \delta)(j'(u) - (k^* + 1))}{k_N} - \frac{2}{k_N}.
\end{equation}

By definition, 
\begin{equation}\label{ts1}
\int_{\tau_j^*/a_N}^{\tau_{j+1}^*/a_N} q \bigg( \frac{\tau_j^*}{a_N} \bigg) \: du = \frac{1}{k_N}.
\end{equation}
By (\ref{qlip}) and (\ref{tauspacing2}), if $u \in [\tau_j^*/a_N, \tau_{j+1}^*/a_N)$ and $\tau_{j+1}^*/a_N < \zeta$, then $$\bigg|q(u) - q\bigg(\frac{\tau_j^*}{a_N}\bigg) \bigg| \leq \frac{e(\tau^*_{j+1} - \tau^*_j)}{a_N} \leq \frac{e}{k_N}$$ unless $1 \in (\tau_j^*/a_N, u]$.  Combining this observation with (\ref{ts1}) and (\ref{tauspacing2}), we get
\begin{equation}\label{ts2}
\bigg( \frac{1}{k_N} - \frac{e}{k_N^2} \bigg) \1_{\{1 \notin (\tau_j^*/a_N, \tau_{j+1}^*/a_N]\}} \leq \int_{\tau_j^*/a_N}^{\tau_{j+1}^*/a_N} q(u) \: du \leq \frac{1}{k_N} + \frac{e}{k_N^2}.
\end{equation}
Now (\ref{tauspacing2}) and (\ref{ts2}) imply that
\begin{align}
\label{tsu2}
\int_0^u q(v) \: dv &= \int_{\tau_{k^*+1}/a_N}^{\tau_{j^*(u)}/a_N} q(v) \: dv + \int_{\tau_{j^*(u)}/a_N}^u q(v) \: dv \nonumber \\
&\leq \frac{(1 + e/k_N)(j^*(u) - (k^*+1))}{k_N} + \frac{e}{k_N}
\end{align}
and
\begin{equation}\label{tsl2}
\int_0^u q(v) \: dv \geq \frac{(1 - e/k_N)(j^*(u) - (k^* + 1) - 1)}{k_N}.
\end{equation}

Combining (\ref{tsl1}) and (\ref{tsu2}) gives, for sufficiently large $N$,
\begin{align*}
j'(u) - (k^*+1) &\leq \frac{(1 + e/k_N)(j^*(u) - (k^*+1))}{1 - 2 \delta} + \frac{e + 2}{1 - 2 \delta} \\
&\leq (1 + 3 \delta)(j^*(u) - (k^* + 1)) + 5.
\end{align*}
Rearranging this expression, and using that $j^*(u) - (k^* + 1) \leq (a_N u)(ek_N/a_N) \leq eT k_N$ by (\ref{tauspacing2}), we get for sufficiently large $N$,
\begin{equation}\label{ts3}
j'(u) - j^*(u) \leq 3 \delta (j^*(u) - (k^* + 1)) + 5 \leq 9 \delta T k_N.
\end{equation}
Likewise, combining (\ref{tsu1}) and (\ref{tsl2}), we get for sufficiently large $N$,
\begin{align*}
j'(u) - (k^* + 1) &\geq \frac{(1 - e/k_N)(j^*(u) - (k^*+1) - 1)}{1 + 2 \delta} - \frac{4e}{1 + 2 \delta} \\
&\geq (1 - 3 \delta)(j^*(u) - (k^* + 1)) - (4e + 1).
\end{align*}
Rearranging, and again using that $j^*(u) - (k^* + 1) \leq eT k_N$, we get
\begin{equation}\label{ts4}
j'(u) - j^*(u) \geq - 3 \delta (j^*(u) - (k^* + 1)) - (4e+1) \geq -9 \delta T k_N.
\end{equation}
The result (\ref{tslem1}) follows from (\ref{ts3}) and (\ref{ts4}).

Finally, to prove (\ref{tslem2}), note that on the event $\{\tau_j < \zeta \wedge a_N T\}$, we have $j^*(\tau_j^*/a_N) = j$ and $j'(\tau_j/a_N) = j$.  Therefore, using (\ref{tslem1}), we have
$$|j^*(\tau_j/a_N) - j^*(\tau_j^*/a_N)| = |j^*(\tau_j/a_N) - j'(\tau_j/a_N)| \leq 9 \delta T k_N.$$
Since $|j^*(\tau_j/a_N) - j^*(\tau_j^*/a_N)|$ is the number of points $\tau_i^*$ that land between $\tau_j$ and $\tau_j^*$,
it now follows from (\ref{tauspacing2}) that for sufficiently large $N$,
$$|\tau_j - \tau_j^*| \leq (9 \delta T k_N + 1) \cdot \frac{a_N}{k_N} \leq 10 \delta a_N T,$$ which matches (\ref{tslem2}).
\end{proof}

Define the fixed positive integers 
$$j_1 = j^*(T - (t+1)) - \lfloor 9 \delta T k_N \rfloor, \hspace{.4in}j_2 = j^*(T - 1 + 19/k_N) + \lfloor 9 \delta T k_N \rfloor,$$ and let $$I = \{j \in \N: j_1 \leq j \leq j_2\}.$$  The next result shows that, when tracing ancestral lines back from time $a_N(T-1)$ to time $a_N(T - (t+1))$, we only need to consider time intervals $[\tau_j, \tau_{j+1}]$ for $j \in I$. 

\begin{Lemma}\label{Ilem}
On the event $\Lambda$, for sufficiently large $N$, we have
\begin{equation}\label{jstarprime}
a_N(T - (t + 1)) - 10 \delta a_N T \leq \tau_{j_1}^* \leq \tau_{j_2}^* \leq a_N(T - 1) + 10 \delta a_N T
\end{equation}
and
\begin{equation}\label{jprime}
a_N + \frac{2a_N}{k_N} < \tau_{j_1} < a_N(T - (t + 1)).
\end{equation}
Also, $L+9 \leq j_2 \leq J$ and $\tau_{j_2+1} < a_N T$.   Furthermore, the cardinality of $I$ is at most $3T k_N$.
\end{Lemma}

\begin{proof}
Throughout the proof, we will work on the event $\Lambda$.  Using (\ref{tauspacing2}), we get that for sufficiently large $N$, $$\tau_{j_1}^* \geq a_N(T - (t + 1)) - (9 \delta T k_N + 1) \cdot \frac{a_N}{k_N} \geq a_N(T - (t + 1)) - 10 \delta a_N T$$ and
$$\tau_{j_2}^* \leq a_N \bigg(T - 1 + \frac{19}{k_N} \bigg) + (9 \delta T k_N) \cdot \frac{a_N}{k_N} \leq a_N(T - 1) + 10 \delta a_N T.$$  We have now proved (\ref{jstarprime}).

By (\ref{tslem1}), we have $j_1 \leq j'(T - (t + 1))$, and thus $\tau_{j_1} \leq a_N(T - (t + 1))$, which is the upper bound in (\ref{jprime}).  To get the lower bound, note that (\ref{tslem2}) and (\ref{tauspacing2}) give $$\tau_{j_1} \geq \tau^*_{j_1} - 10 \delta a_N T \geq \tau^*_{j^*(T - (t+1))} - (9 \delta T k_N) \cdot \frac{a_N}{k_N} - 10 \delta a_N T.$$  Since (\ref{tauspacing2}) implies $\tau^*_{j^*(u)} \geq a_N u - a_N/k_N$ for $u \in (0, T]$, it follows, using (\ref{deldef}), that for sufficiently large $N$,
\begin{align}\label{jp2}
\tau_{j_1} &\geq a_N(T - (t+1)) - \frac{a_N}{k_N} - 19 \delta a_N T \nonumber \\
&> a_N + a_N(T - (t + 2) - 20 \delta T) \nonumber \\
&\geq a_N + a_N \bigg( \frac{T - (t+2)}{2} \bigg).
\end{align}
The lower bound in (\ref{jprime}) follows because $\lim_{N \rightarrow \infty} k_N = \infty$.

Next, note that by (\ref{tauspacing}) and (\ref{Ldef}), we have $\tau_{L+10} \leq a_N(T-1) + 19 a_N/k_N$.  By (\ref{tslem1}), we have $j_2 \geq j'(T - 1 + 19/k_N)$, which means $\tau_{j_2+1} > a_N(T - 1 + 19/k_N) \geq \tau_{L+10}$ and thus $j_2 \geq L + 9$.  Also, by (\ref{tauspacing2}), the number of times $\tau_i^*$ between $a_N(T - 1 + 19/k_N)$ and $a_N T$ is at least $(k_N/a_N)(a_N(1 - 19/k_N)) - 1 = k_N - 20$.  Therefore, using (\ref{tslem1}), $$j'(T) \geq j^*(T) - 9 \delta T k_N \geq j^*(T - 1 + 19/k_N) + k_N - 20 - 9 \delta T k_N \geq j_2 + k_N - 20 - 18 \delta T k_N,$$ which is greater than $j_2 + 1$ for sufficiently large $N$ because $\delta < 1/19T$ by (\ref{deldef}).  It follows that $\tau_{j_2 + 1} < a_N T$.

Finally, by Remark \ref{Jrem}, we have $j_2 + 1 \leq J$.  Also, we have and $j_1 \geq k^* + 1$ for sufficiently large $N$ by (\ref{jp2}), so $j_2 - j_1 + 1 \leq 3Tk_N$, which is equivalent to the last statement of the lemma.
\end{proof}

\subsection{The types of the ancestors at time $\tau_j$}

Lemma \ref{taujmut} below establishes that with high probability, the type $\ell$ individuals in the sample get traced back to type $\ell-1$ individuals at time $\tau_{\ell}$, then to type $\ell - 2$ individuals at time $\tau_{\ell - 1}$, and so on until we have traced the lineages back to time $a_N(T - (t + 1))$.  We begin with the following preliminary result.

\begin{Lemma}\label{Kjlem2}
Let $j \in I$.  Let $K_j$ be the number of type $j$ individuals in the population at time $\tau_{j+1}$ whose ancestor in the population at time $\tau_j$ does not have type $j-1$.  Then
$$E[K_j \1_{\{\tau_{j+1} < \zeta\}}] \leq 5 \bigg( \frac{s}{\mu} \bigg)^{1 - 1/3k_N}.$$
\end{Lemma}

\begin{proof}
By parts 1 and 6 of Proposition \ref{mainjprop}, on $\{\tau_{j+1} < \zeta\}$, no individual of type $j$ or higher in the population at time $\tau_j$ has a descendant alive in the population at time $\tau_{j+1}$.  Therefore, $K_j$ is the number of type $j$ individuals at time $\tau_{j+1}$ whose ancestor at time $\tau_j$ has type less than $j-1$.  Such an individual must be descended from an individual that gets its $(j-1)$st mutation after time $\tau_j$.  We will therefore consider the number of type $j-1$ individuals at times $t \geq \tau_j$ that are descended from individuals that acquired their $(j-1)$st mutation between times $\tau_j$ and $\tau_{j+1}$.  Following Corollary \ref{ZmartCor4}, we denote the number of such individuals by $X_{j-1}^{\tau_j, \tau_{j+1}}(t)$.  Then, writing $\zeta_j = \zeta \wedge \tau_{j+1}$,
\begin{align}\label{Xstar1}
&e^{-\int_{\tau_j}^{(\tau_j + u) \wedge \zeta_j} G_{j-1}(v) \: dv} X_{j-1}^{\tau_j, \tau_{j+1}}((\tau_j + u) \wedge \zeta_j) \nonumber \\
&\hspace{0.6in} = \int_{\tau_j}^{(\tau_j + u) \wedge \zeta_j} \mu X_{j-2}(w) e^{-\int_{\tau_j}^w G_{j-1}(v) \: dv} \: dw + Z_{j-1}^{\tau_j, \tau_{j+1}}((\tau_j + u) \wedge \zeta_j),
\end{align}
where $(Z_j^{\tau_j, \tau_{j+1}}(\tau_j + u), u \geq 0)$ is a mean zero martingale.  By (\ref{prop23}), on the event $\{\zeta_j > \tau_j\}$, we have for $t \geq 0$,
\begin{align}\label{Xstar2}
\int_{\tau_j}^{(\tau_j + u) \wedge \zeta_j} \mu X_{j-2}(w) e^{-\int_{\tau_j}^w G_{j-1}(v) \: dv} \: du &\leq \int_{\tau_j}^{(\tau_j + u) \wedge \zeta_j} (1 + \delta) s e^{\int_{\tau_{j-1}}^w G_{j-2}(v) \: dv} e^{-\int_{\tau_j}^w G_{j-1}(v) \: dv} \: dw \nonumber \\
&= (1 + \delta) s e^{\int_{\tau_{j-1}}^{\tau_j} G_{j-2}(v) \: dv} \int_{\tau_j}^{(\tau_j + u) \wedge \zeta_j} e^{-s(w - \tau_j)} \: dw \nonumber \\
&\leq (1 + \delta) e^{\int_{\tau_{j-1}}^{\tau_j} G_{j-2}(v) \: dv}.
\end{align}
By (\ref{tauspacing}) and part 4 of Proposition \ref{newGq}, on $\{\zeta_j > \tau_j\}$ we have
\begin{equation}\label{Xstar3}
e^{\int_{\tau_{j-1}}^{\tau_j} G_{j-2}(v) \: dv} = e^{-s(\tau_j - \tau_{j-1})} e^{\int_{\tau_{j-1}}^{\tau_j} G_{j-1}(v) \: dv} \leq e^{-s(a_N/3k_N)} \bigg( \frac{2s}{\mu} \bigg) \leq 2 \bigg( \frac{s}{\mu} \bigg)^{1-1/3k_N}.
\end{equation}
Taking conditional expectations on both sides of (\ref{Xstar1}) and then using (\ref{Xstar2}) and (\ref{Xstar3}) gives 
\begin{equation}\label{Xstar5}
E \big[ e^{-\int_{\tau_j}^{(\tau_j + u) \wedge \zeta_j} G_{j-1}(v) \: dv} X_{j-1}^{\tau_j, \tau_{j+1}}((\tau_j + u) \wedge \zeta) \big| {\cal F}_{\tau_j} \big] \leq 2 (1 + \delta) \bigg( \frac{s}{\mu} \bigg)^{1-1/3k_N}.
\end{equation}

For $u \geq \tau_j$, let $X_j^*(u)$ denote the the number of type $j$ individuals in the population at time $u$ that got their $(j-1)$st mutation after time $\tau_j$.  Note that $K_j = X_j^*(\tau_{j+1})$ on $\{\tau_{j+1} < \zeta\}$.  By the reasoning that leads to Corollary \ref{ZmartCor4}, we get
\begin{equation}\label{Xstar4}
e^{-\int_{\tau_j}^{u \wedge \zeta_j} G_j(v) \: dv} X_j^*(u \wedge \zeta_j) = \int_{\tau_j}^{u \wedge \zeta_j} \mu X_{j-1}^{\tau_j, \tau_{j+1}}(w) e^{-\int_{\tau_j}^w G_j(v) \: dv} \: dw + Z_j^*(u \wedge \zeta_j),
\end{equation}
where $(Z_j^*(\tau_{j+1} + u), u \geq 0)$ is a mean zero martingale.  By part 4 of Proposition \ref{newGq}, on $\{\tau_{j+1} < \zeta\}$, we have $$e^{\int_{\tau_j}^{\tau_{j+1}} G_j(v) \: dv} \leq \frac{2s}{\mu}.$$
Therefore, using that the expression in (\ref{Xstar4}) is nonnegative,
$$X_j^*(\tau_{j+1}) \1_{\{\tau_{j+1} < \zeta\}} \leq \frac{2s}{\mu} \bigg( \int_{\tau_j}^{\infty} \mu X_{j-1}^{\tau_j, \tau_{j+1}}(w) e^{-\int_{\tau_j}^w G_{j-1}(v) \: dv} e^{-s(w - \tau_j)} \1_{\{w \leq \zeta_j\}} \: dw + Z_j^*(\zeta_j) \bigg).$$
Taking conditional expectations of both sides and using Fubini's Theorem and (\ref{Xstar5}),
\begin{align*}
E[ X_j^*(\tau_{j+1}) \1_{\{\tau_{j+1} < \zeta\}} | {\cal F}_{\tau_j}] &\leq \frac{2s}{\mu} \bigg( \int_{\tau_j}^{\infty} 2(1 + \delta) \mu \bigg( \frac{s}{\mu} \bigg)^{1 - 1/3k_N} e^{-s(w - \tau_j)} \: dw \bigg) \\
&= 4(1 + \delta) \bigg( \frac{s}{\mu} \bigg)^{1 - 1/3k_N}.
\end{align*} 
Taking expectations of both sides gives the result of the lemma.
\end{proof}

\begin{Lemma}\label{taujmut}
We have
$$\lim_{N \rightarrow \infty} P \big( \Lambda \cap \big\{ U_i(\tau_{j}) \neq j - 1 \mbox{ for some }i \in \{1, \dots, n\} \mbox{ and }j \in I \mbox{ with }j \leq U_i(a_N T) \big\} \big) = 0.$$
\end{Lemma}

\begin{proof}
Fix $i \in \{1, \dots, n\}$.  Suppose $\Lambda$ occurs and $U_i(\tau_j) \neq j-1$ for some $j \in I \mbox{ with }j \leq U_i(a_N T)$.  Then either $U_i(\tau_{U_i(a_N T) + 1}) \neq U_i(a_N T)$, an event whose probability tends to zero as $N \rightarrow \infty$ by Lemma \ref{noearlymut}, or else there is an integer $j \in I$ with $j < U_i(a_N T)$ such that $U_i(\tau_j) \neq j-1$ and $U_i(\tau_{j+1}) = j$.  Therefore, to prove the lemma, it suffices to show that
\begin{equation}\label{stsmut}
\lim_{N \rightarrow \infty} \sum_{j \in I} P\big(\Lambda \cap \{U_i(\tau_j) \neq j-1\} \cap \{U_i(\tau_{j+1}) = j\} \big) = 0.
\end{equation}

Fix $j \in I$.  Recall from Lemma \ref{Kjlem2} that $K_j$ is the number of type $j$ individuals in the population at time $\tau_{j+1}$ whose ancestor in the population at time $\tau_j$ does not have type $j-1$.  Note that the probability, conditional on ${\cal F}_{\tau_{j+1}}$, that a randomly chosen type $j$ individual at time $\tau_{j+1}$ is not descended from a type $j-1$ individual at time $\tau_j$ is $K_j/\lceil s/\mu \rceil$.
Also, conditional on ${\cal F}_{\tau_{j+1}}$, the $\lceil s/\mu \rceil$ type $j$ individuals at time $\tau_{j+1}$ are equally likely to be the ancestor of the $i$th individual in the sample taken at time $a_N T$.  Therefore, since $K_j$ is ${\cal F}_{\tau_{j+1}}$-measurable, on the event $\tau_{j+1} < \zeta$ we have
$$P\big(\{U_i(\tau_j) \neq j-1\} \cap \{U_i(\tau_{j+1}) = j\}|{\cal F}_{\tau_{j+1}}\big) = P(U_i(\tau_{j+1}) = j|{\cal F}_{\tau_{j+1}}) \cdot \frac{K_j}{\lceil s/\mu \rceil} \leq \frac{\mu K_j}{s}.$$  Therefore, multiplying both sides by $\1_{\{\tau_{j+1} < \zeta\}}$, taking expectations, and using Lemma \ref{Kjlem2}, we get
$$P\big(\{U_i(\tau_j) \neq j-1\} \cap \{U_i(\tau_{j+1}) = j\} \cap\{\tau_{j+1} < \zeta\} \big) \leq \frac{\mu}{s} E[K_j \1_{\{\tau_{j+1} < \zeta\}}] \leq 5 \bigg( \frac{s}{\mu} \bigg)^{-1/3k_N}.$$  Since the cardinality of $I$ is at most $3Tk_N$ by Lemma \ref{Ilem}, it follows that the sum of the probabilities on the left-hand side of (\ref{stsmut}) is at most $$15 T k_N \bigg( \frac{s}{\mu} \bigg)^{-1/3k_N}.$$
To check that this expression goes to zero as $N \rightarrow \infty$, we consider the logarithm.  Note that $\log (k_N(s/\mu)^{-1/3k_N}) = \log k_N - (1/3k_N) \log (s/\mu),$ which tends to $-\infty$ as $N \rightarrow \infty$ by assumption A2.  In view of the discussion before equation (\ref{stsmut}), the result of the lemma follows.
\end{proof}

\subsection{Coalescence between times $\tau_j$ and $\tau_{j+1}$}\label{coaljj}

We next consider the merging of ancestral lines between times $\tau_j$ and $\tau_{j+1}$.  It will suffice to consider the lineages labelled $1$ and $2$.  In view of Lemma \ref{taujmut}, we may also assume these lineages have type $j$ at time $\tau_{j+1}$ and type $j-1$ at time $\tau_j$, which will occur with high probability.  Recall the definitions of $V_{i,j}$ and $T_{i,j}$ from (\ref{Vijdef}) and (\ref{Tijdef}).  Also, let $V_j = \min\{V_{1,j}, V_{2,j}\}$ and $V_j^* = \max\{V_{1,j}, V_{2,j}\}$.  Because only lineages of the same type can coalesce, there are only three ways that these lineages could coalesce between times $\tau_{j-1}$ and $\tau_j$:
\begin{enumerate}
\item Two lineages at time $\tau_{j+1}$ could be traced back to one individual that acquires its $j$th mutation between times $\xi_j$ and $\tau_{j+1}$.  That is, $\xi_j < V_{1,j} = V_{2,j} < T_{1,2} \leq \tau_{j+1}$.

\item Two lineages at time $\tau_{j+1}$ could be traced back to one individual that acquires its $j$th mutation before time $\xi_j$.  That is, $\tau_j < V_{1,j} = V_{2,j} < \xi_j$ and $V_{1,j} = V_{2,j} < T_{1,2} \leq \tau_{j+1}$.

\item Two lineages at time $\tau_{j+1}$ could be descended from different type $j$ mutations between times $\tau_j$ and $\tau_{j+1}$, but then the two type $j-1$ lineages could coalesce before time $\tau_j$.  That is, $\tau_j \leq T_{1,2} < V_j < V_j^* < \tau_{j+1}$.
\end{enumerate}
We will now show that only coalescence events of the second type need to be considered.  Lemma \ref{coalj1} rules out case 1 above, and Lemma \ref{coalj2} rules out case 3.

\begin{Lemma}\label{coalj1}
Define the event $$A_j' = \{U_1(\tau_{j+1}) = U_2(\tau_{j+1}) = j\} \cap \{U_1(\tau_j) = U_2(\tau_j) = j-1\} \cap \{ \xi_j < V_{1,j} = V_{2,j} < T_{1,2} < \tau_{j+1}\}.$$  For sufficiently large $N$, we have
$$P\bigg( \Lambda \cap \bigcup_{j \in I} A_j' \bigg) \leq C T e^{-b}.$$
\end{Lemma}

\begin{proof}
Fix $j \in I$.  Let $H_j$ be the number of type $j$ mutations between times $\xi_j$ and $\tau_{j+1}$.  Let $0 < \kappa_1 < \kappa_2 < \dots < \kappa_{H_j}$ denote the times at which these mutations occur.  Let $X_{j,2,i}(u)$ be the number of type $j$ individuals at time $u$ descended from the individual that acquires its $j$th mutation at time $\kappa_i$.  This means that $$X_{j,2}(\tau_{j+1}) = \sum_{i=1}^{H_j} X_{j,2,i}(\tau_{j+1}).$$  Conditional on $X_{j,2,1}(\tau_{j+1}), \dots, X_{j,2,H_j}(\tau_{j+1})$, the probability that two randomly chosen individuals at time $\tau_{j+1}$ are descended from the same individual that gets its $j$th mutation between times $\xi_j$ and $\tau_{j+1}$ is $$\frac{1}{\lceil s/\mu \rceil (\lceil s/\mu \rceil - 1)} \sum_{i=1}^{H_j} X_{j,2,i}(\tau_{j+1}) (X_{j,2,i}(\tau_{j+1}) - 1) \leq \frac{\mu^2}{s^2} \sum_{i=1}^{H_j} X_{j,2,i}(\tau_{j+1})^2.$$  Since, conditional on ${\cal F}_{\tau_{j+1}}$, each of the $\lceil s/\mu \rceil$ type $j$ individuals at time $\tau_{j+1}$ is equally likely to be the ancestor of an individual in our sample at time $a_N T$, it follows that on $\tau_{j+1} < \zeta$,
\begin{equation}\label{Aj1}
P(A_j'|{\cal F}_{\tau_{j+1}}) \leq \frac{\mu^2}{s^2} \sum_{i=1}^{H_j} X_{j,2,i}(\tau_{j+1})^2.
\end{equation}
Therefore, multiplying both sides by $\1_{\{\tau_{j+1} < \zeta\}}$ and taking expectations,
\begin{equation}\label{Ajnew}
P(A_j' \cap \{\tau_{j+1} < \zeta\}) \leq \frac{\mu^2}{s^2} E \bigg[ \bigg( \sum_{i=1}^{H_j} X_{j,2,i}(\tau_{j+1})^2 \bigg)\1_{\{\tau_{j+1} < \zeta\}} \bigg].
\end{equation}

We now bound the expectation on the right-hand side of (\ref{Ajnew}).
Write $\zeta_j = \zeta \wedge \tau_{j+1}$.  By Corollary \ref{ZmartCor3} applied with $\kappa_i$ playing the role of $\kappa$ and the single type $j$ individual that acquires its $j$th mutation at time $\kappa_i$ playing the role of $S$, we get
\begin{equation}\label{Aj2}
X_{j,2,i}(u \wedge \zeta_j) = e^{\int_{\kappa_i}^{u \wedge \zeta_j} G_j(v) \: dv} (1 + Z_{i,j}(u)),
\end{equation}
where $(Z_{i,j}(\kappa_i + u), u \geq 0)$ is a mean zero martingale.  Therefore, using part 4 of Proposition \ref{newGq}, we get that on $\{\kappa_i < \tau_{j+1} < \zeta\}$,
\begin{align}\label{Aj3}
X_{j,2,i}(\tau_{j+1}) &= e^{\int_{\tau_j}^{\tau_{j+1}} G_j(v) \: dv} e^{-\int_{\tau_j}^{\kappa_i} G_j(v) \: dv}(1 + Z_{i,j}(\tau_{j+1})) \nonumber \\
&\leq \frac{2s}{\mu} e^{-\int_{\tau_j}^{\kappa_i} G_j(v) \: dv} (1 + Z_{i,j}(\tau_{j+1})).
\end{align}
Corollary \ref{ZmartCor3} combined with (\ref{BD3}) and (\ref{Aj2}) gives that on $\{\kappa_i < \tau_{j+1}\}$,
\begin{align*}
\Var(Z_{i,j}(\tau_{j+1})|{\cal F}_{\kappa_i}) &\leq 3 E \bigg[ \int_{\kappa_i}^{\tau_{j+1} \wedge \zeta_j} e^{-2 \int_{\kappa_i}^u G_j(v) \: dv} X_{j,2,i}(u) \: du \bigg| {\cal F}_{\kappa_i} \bigg] \\
&= 3 E \bigg[ \int_{\kappa_i}^{\tau_{j+1} \wedge \zeta_j} e^{-\int_{\kappa_i}^u G_j(v) \: dv}(1 + Z_{i,j}(u)) \: du \bigg| {\cal F}_{\kappa_i} \bigg].
\end{align*}
Because $k^*+1 \leq j \leq J$ by Lemma \ref{Ilem}, it follows from part 2 of Proposition \ref{newGq} that for sufficiently large $N$, if $v \in [\tau_j, \tau_{j+1}]$ and $v < \zeta$, then $G_j(v) \geq (1 - 2 \delta) s k_N$. 
Therefore, on $\{\kappa_i < \tau_{j+1}\}$,
\begin{equation}\label{Aj4}
\Var(Z_{i,j}(\tau_{j+1})|{\cal F}_{\kappa_i}) \leq 3 E \bigg[ \int_{\kappa_i}^{\infty} e^{-sk_N(1 - 2 \delta)(u - \kappa_i)} (1 + Z_{i,j}(u)) \: du \bigg| {\cal F}_{\kappa_i} \bigg] = \frac{3}{sk_N(1 - 2 \delta)}.
\end{equation}
From (\ref{Aj3}) and (\ref{Aj4}), we get that on $\{\kappa_i < \tau_{j+1}\}$,
$$E\big[X_{j,2,i}^2(\tau_{j+1}) \1_{\{\tau_{j+1} < \zeta\}} \big| {\cal F}_{\kappa_i} \big] \leq \frac{4 s^2}{\mu^2}  e^{-2 \int_{\tau_j}^{\kappa_i} G_j(v) \: dv} \bigg(1 + \frac{3}{sk_N (1 - 2 \delta)} \bigg).$$
By assumption A3, the second term inside the parentheses dominates when $N$ is large. 
Also, by part 1 of Proposition \ref{newGq}, we have $s(q_j - C_3) \leq G_j(v) \leq s(q_j + C_3)$ if $\tau_j \leq v < \zeta_j$.  Therefore, for sufficiently large $N$, on $\{\kappa_i < \tau_{j+1}\}$,
\begin{equation}\label{Aj5}
E\big[X_{j,2,i}^2(\tau_{j+1}) \1_{\{\tau_{j+1} < \zeta\}} \big| {\cal F}_{\kappa_i} \big] \leq \frac{C s}{\mu^2 k_N} e^{-2 s(q_j - C_3)(\kappa_i - \tau_j)}.
\end{equation}

Next, we condition on ${\cal F}_{\tau_j}$.   Using (\ref{prop23}) followed by part 1 of Proposition \ref{newGq}, the rate at which type $j$ mutations are appearing at time $u$, provided that $\tau_j \leq u < \zeta_j$, is
\begin{equation}\label{Aj9}
\mu X_{j-1}(u) \leq (1 + \delta) s e^{\int_{\tau_j}^{u} G_{j-1}(v) \: dv} \leq (1 + \delta) s e^{s(q_j + C_3 - 1) (u - \tau_j)}.
\end{equation}
Therefore, using (\ref{Aj5}), (\ref{Aj9}), and (\ref{tauspacing}), we get that on $\{\tau_j < \zeta\}$,
\begin{align}\label{Aj6}
&E \bigg[ \bigg( \sum_{j=1}^{H_j} X_{j,2,i}(\tau_{j+1})^2 \bigg) \1_{\{\tau_{j+1} < \zeta\}} \bigg| {\cal F}_{\tau_j} \bigg] \nonumber \\
&\hspace{.5in}\leq \int_{\xi_j}^{(\tau_j + 2a_N/k_N) \wedge \zeta} (1 + \delta) s e^{s(q_j + C_3 - 1) (u - \tau_j)} \cdot \frac{Cs}{\mu^2 k_N} e^{-2s(q_j - C_3)(u - \tau_j)} \: du \nonumber \\
&\hspace{.5in}\leq \frac{C s^2}{\mu^2 k_N} \int_{\xi_j}^{(\tau_j + 2a_N/k_N) \wedge \zeta} e^{-s(q_j - 3C_3 + 1)(u - \tau_j)} \: du \nonumber \\
&\hspace{.5in}\leq \frac{C s}{\mu^2 k_N (q_j - 3C_3 + 1)} \cdot e^{-s(q_j - 3 C_3 + 1)(\xi_j - \tau_j)}.
\end{align}
Note that $e^{-s q_j(\xi_j - \tau_j)} = s q_j e^{-b}$ for sufficiently large $N$ by (\ref{xijdef}).  Also, $q_j \geq (1 - 2 \delta) k_N$ on $\{\tau_j < \zeta\}$ for sufficiently large $N$ by part 3 of Proposition \ref{newGq}, so
\begin{equation}\label{Aj8}
s(\xi_j - \tau_j) = \frac{1}{q_j} \bigg( \log \bigg( \frac{1}{s q_j} \bigg) + b \bigg) \leq \frac{1}{(1 - 2 \delta) k_N} \bigg( \log \bigg( \frac{1}{s} \bigg) + b \bigg) \rightarrow 0
\end{equation}
as $N \rightarrow \infty$ by assumption A1.  Therefore, $e^{(3C_3 - 1) s (\xi_j - \tau_j)} \rightarrow 1$ as $N \rightarrow \infty$.  Combining these observations with (\ref{Aj6}), and using that $\tau_{j+1} < \zeta$ implies $\tau_j < \zeta$ in view of (\ref{tauspacing}), we get that for sufficiently large $N$,
\begin{equation}\label{Aj7}
E \bigg[ \bigg( \sum_{j=1}^{H_j} X_{j,2,i}(\tau_{j+1})^2 \bigg) \1_{\{\tau_{j+1} < \zeta\}} \bigg| {\cal F}_{\tau_j} \bigg] \leq \frac{C s^2 e^{-b}}{\mu^2 k_N}.
\end{equation}
Finally, we can take expectations of both sides in (\ref{Aj7}) and combine the result with (\ref{Ajnew}) and the fact that the cardinality of $I$ is at most $3k_N T$ by Lemma \ref{Ilem} to obtain the result of the lemma.
\end{proof}

\begin{Lemma}\label{coalj2}
Recall that $V_j = \min\{V_{1,j}, V_{2,j}\}$.  Define the event
$$A_j^* = \{U_1(\tau_j) = U_2(\tau_j) = j-1 \} \cap \{\tau_j \leq T_{1,2} < V_j < \tau_{j+1}\}.$$  Then
$$\lim_{N \rightarrow \infty} P \bigg( \Lambda \cap \bigcup_{j \in I} A_j^* \bigg) = 0.$$
\end{Lemma}

\begin{proof}
Fix $j \in I$.  Randomly label the type $j - 1$ individuals at time $\tau_j$ by $1, 2, \dots, \lceil s/\mu \rceil$.  For $t \geq \tau_j$, let $X_{j-1}^i(t)$ denote the number of type $j-1$ individuals at time $t$ descended from the type $j-1$ individual labelled $i$ at time $\tau_j$.  

Let ${\cal C}_j$ be the $\sigma$-field generated by the random variables $V_{1,j}, V_{2,j}, X_{j-1}^1(V_j-), \dots, X_{j-1}^{\lceil s/\mu \rceil}(V_j-)$ and the event $\{V_j < \zeta\}$.  The only way that $A_j^*$ can occur is if the first two lineages get traced back to distinct type $j-1$ ancestors at time $V_j-$ and then merge between times $\tau_j$ and $V_j-$.  Conditional on ${\cal C}_j$, we know that one of the type $j-1$ individuals at time $V_j-$ will get a $j$th mutation at time $V_j$, but all of the type $j-1$ individuals at time $V_j-$ are equally likely to be ancestors of individuals in our sample at time $a_N T$.  Therefore, using the notation from (\ref{Mijdef}),
$$P(A_j^* \1_{\{V_j < \zeta\}}|{\cal C}_j) \leq \bigg( \sum_{i=1}^{\lceil s/\mu \rceil} \frac{X_{j-1}^i(V_j-) (X_{j-1}^i(V_j-) - 1)}{X_{j-1}(V_j-) (X_{j-1}(V_j-) - 1)} \bigg)  \1_{\{V_j < \zeta\}} \leq \sum_{i=1}^{\lceil s/\mu \rceil} R_{i,j-1}^2.$$
It follows from part 6 of Proposition \ref{mainjprop} that $\tau_{k^*+K+1} < (K+1)(2a_N/k_N) \leq a_N$ for sufficiently large $N$, and therefore by (\ref{jprime}), we have $j_1 - 1 \geq k^*+1+K$.  Thus,
summing over $j \in I$, taking expectations of both sides, and applying Lemma \ref{maxXj} and Lemma \ref{Ilem}, we get
\begin{equation}\label{Ajstar}
\sum_{j \in I} P(A_j^* \cap \{V_j < \zeta\}) \leq \sum_{j \in I} \frac{C \mu}{s^2 k_N} \leq \frac{C \mu T}{s^2}.
\end{equation}
The right-hand side of (\ref{Ajstar}) tends to zero as $N \rightarrow \infty$ by (\ref{musN}).  The result of the lemma follows because if $\Lambda \cap A_j^*$ occurs for some $j \in I$, then $V_j < \tau_{j+1} < \zeta$ by Lemma \ref{Ilem}.
\end{proof}

\section{Coupling with a branching process between times $\tau_j$ and $\tau_{j+1}$}

Recall from Lemmas \ref{coalj1} and \ref{coalj2} and the discussion before Lemma \ref{coalj1} that we have shown that all possible coalescence events have low probability, except for the possibility that type $j$ lineages at time $\tau_{j+1}$ could be descended from the same type $j$ mutation between times $\tau_j$ and $\xi_j$.  In this section, we study these early type $j$ mutations in depth.  The strategy here will 
to couple the descendants of these mutations with a supercritical branching process.

\subsection{Review of results on continuous-time branching processes}

Consider a continuous-time birth and death process $(Z(t), t \geq 0)$ in which each individual independently dies at rate $\nu > 0$ and gives birth to a new individual at rate $\lambda > \nu$.  Assume $Z(0) = 1$.  Using results in \cite{athney}, one can show that
\begin{equation}\label{survivetot}
P(Z(t) > 0) = \frac{\lambda - \nu}{\lambda - \nu e^{-(\lambda - \nu) t}},
\end{equation}
which is also stated as part of Lemma 8.16 of \cite{schI}.  Let $q$ denote the probability that the population goes extinct by time $t$.  By letting $t \rightarrow \infty$ in (\ref{survivetot}), we get
\begin{equation}\label{survivalprob}
1 - q = \frac{\lambda - \nu}{\lambda}.
\end{equation}

Let $W(t) = e^{-(\lambda - \nu)t} Z(t)$.  It is well-known (see, for example, section 7 of Chapter III in \cite{athney}) that $(W(t), t \geq 0)$ is a martingale, and there is a random variable $W$ such that
\begin{equation}\label{Wlim}
\lim_{t \rightarrow \infty} W(t) = W \hspace{.2in}\mbox{a.s.},
\end{equation}
where $W$ is zero on the event that the branching process goes extinct and is almost surely strictly positive on the event that the branching process survives forever.  In this instance, it is also known that the conditional distribution of $W$ given that the branching process survives forever is the exponential distribution with rate parameter $1 - q$, so that if $x \geq 0$, then
\begin{equation}\label{Wexp}
P(W > x) = (1 - q) e^{-(1 - q) x}.
\end{equation}
This can be derived from results in \cite{athney} and is also worked out, for example, in \cite{dm10}.  Recall that if $S$ has an exponential distribution with parameter $\lambda$, then $E[S] = 1/\lambda$ and $E[S^2] = 2/\lambda^2$.  Because $P(W > 0) = 1-q$, it follows that $E[W] = 1$ and $\Var(W) \leq E[W^2] = 2/(1-q)$.  We will need the following result concerning the rate of convergence of $W(t)$ to $W$.  

\begin{Lemma}\label{bplem2}
For all $\eta > 0$ and $t > 0$, we have $$P(|W(t) - W| > \eta) \leq \frac{2 e^{-(\lambda - \nu) t}}{\eta^2 (1 - q)}.$$
\end{Lemma}

\begin{proof}
Conditional on $Z(t)$, we can consider separately the descendants of the $Z(t)$ individuals at time $t$ to see that $$W = e^{-(\lambda - \nu) t} \sum_{i=1}^{Z(t)} W_i,$$ where the random variables $W_1, \dots, W_{Z(t)}$ are independent and have the same distribution as $W$ (see section 10 in Chapter III of \cite{athney}).  It follows that $$E[W|Z(t)] = e^{-(\lambda - \nu) t} Z(t) E[W] = W(t)$$ and $$\Var(W|Z(t)) = e^{-2(\lambda - \nu) t} Z(t) \Var(W) \leq \frac{2 e^{-(\lambda - \nu)t } W(t)}{1 - q}.$$  Therefore, by Chebyshev's Inequality,
\begin{equation}\label{Wrate}
P(|W - W(t)| > \eta|Z(t)) \leq \frac{\Var(W|Z(t))}{\eta^2} \leq \frac{2 e^{-(\lambda - \nu) t} W(t)}{\eta^2(1 - q)}.
\end{equation}
Because $E[W(t)] = 1$, the result follows by taking expectations of both sides in (\ref{Wrate}).
\end{proof}

\subsection{A branching process coupling between times $\tau_j$ and $\tau_{j+1}$}\label{bpcouplesec}

We will assume now that $j \in I$, which by Lemma \ref{Ilem} ensures that $\tau_{j+1} < a_N T$ on $\Lambda$.
Recalling Corollary \ref{ZmartCor4}, we will let
$$X_j'(t) = X_j^{\tau_j, \xi_j}(t)$$ denote
the number of type $j$ individuals at time $t$ that are descended from individuals that acquired a $j$th mutation during the time interval $(\tau_j, \xi_j]$.  Note that $X_j'(t) = X_{j,1}(t)$, as long as there are no type $j$ mutations before time $\tau_j$.
We say there is a pure birth event at time $t$ if $X_j'(t) = X_j'(t-) + 1$ and a pure death event at time $t$ if $X_j'(t) = X_j'(t-) - 1$.  We say there is a birth and death event at time $t$ if one of the $X_j'(t-)$ individuals at time $t-$ gives birth and another dies, so that $X_j'(t) = X_j'(t-)$.  Let $B_j'(t)$ and $D_j'(t)$ denote the expressions in (\ref{Bjeq}) and (\ref{Djeq}) respectively  with $X_j'(t)$ in place of $X_j(t)$.  Recall from the discussion surrounding (\ref{Fjeq}), (\ref{Bjeq}), and (\ref{Djeq}) that if $t \in [\tau_j, \tau_{j+1} \wedge \zeta)$, then the rate at which a particular type $j$ individual gives birth as part of a pure birth event is
\begin{equation}\label{newB}
B_j'(t) = \bigg(1 - \frac{X_j'(t)}{N} \bigg)(1 + s(j - M(t))),
\end{equation}
while the rate at which a particular type $j$ individual is involved in a pure death event is
\begin{equation}\label{newD}
D_j'(t) = 1 + \mu - \frac{X_j'(t)}{N}(1 + s(j - M(t))).
\end{equation}
Also, the rate at which a particular type $j$ individual gives birth as part of a birth and death event and the rate at which a particular type $j$ individual dies as part of a birth and death event are both equal to
\begin{equation}\label{Odef}
O_j(t) = \frac{X_j'(t)}{N}(1 + s(j - M(t))).
\end{equation}
We write $B_j^*(t) = B_j'(t) + O_j(t) = 1 + s(j - M(t))$ and $D_j^*(t) = D_j'(t) + O_j(t) = 1 + \mu$ for the total birth and death rates respectively.  The following lemma gives upper and lower bounds on these birth and death rates.  The lemma also gives a bound on the rate of type $j$ mutations, which will correspond to immigration in our branching process.

\begin{Lemma}\label{birthdeathlem}
There is a positive constant $C_4$ such that for sufficiently large $N$,
if $X_j'(t) \leq s/2 \mu$ and $t \in [\tau_j, \tau_{j+1} \wedge \zeta)$, then the following hold:
\begin{equation}\label{deathbounds}
1 - s \leq D_j'(t) \leq D_j^*(t) = 1 + \mu,
\end{equation}
\begin{equation}\label{birthbounds}
1 - sq_j - C_4 s \leq B_j'(t) \leq B_j^*(t) \leq 1 + sq_j + C_4s,
\end{equation}
\begin{equation}\label{immbounds}
(1 - \delta) s e^{s(q_j - C_4)(t - \tau_j)} \leq \mu X_{j-1}(t) \leq (1 + \delta) s e^{s(q_j + C_4)(t - \tau_j)}.
\end{equation}
\end{Lemma}

\begin{proof}
Suppose $X_j'(t) \leq s/2 \mu$ and $t \in [\tau_j, \tau_{j+1} \wedge \zeta)$.  By (\ref{musN}), assumption A3, and the fact that $j \leq J$ by Lemma \ref{Ilem}, for sufficiently large $N$ we have
\begin{equation}\label{BDXj}
O_j(t) \leq \bigg( \frac{s}{2 \mu N} \bigg)(1 + s(j - M(t))) \leq \bigg( \frac{s}{2 \mu N} \bigg)(1 + sJ) \leq s.
\end{equation}
The result (\ref{deathbounds}) follows immediately from equations (\ref{newD}), (\ref{Odef}), and (\ref{BDXj}).

To bound the birth rate, note that since $G_j(t) = s(j - M(t)) - \mu$, we have $$1 + \mu + G_j(t) - O_j(t) = B_j'(t) \leq B_j^*(t) = 1 + \mu + G_j(t).$$  Since $s(q_j - C_3) \leq G_j(t) \leq s(q_j + C_3)$ for sufficiently large $N$ by (\ref{jprime}) and part 1 of Proposition \ref{newGq}, the inequality (\ref{birthbounds}) now follows from (\ref{BDXj}) and (\ref{musN}).

Finally, if $t \in [\tau_j, \tau_{j+1} \wedge \zeta)$, then since $G_{j-1}(t) = G_j(t) - s$, part 1 of Proposition \ref{newGq} gives
$s(q_j - C_3 - 1) \leq G_{j-1}(t) \leq s(q_j + C_3 - 1)$.  Now (\ref{immbounds}) follows from this observation and (\ref{prop23}).
\end{proof}

We will use the bounds in Lemma \ref{birthdeathlem} to obtain a coupling in which $(X_j'(t), t \geq \tau_j)$ is bounded between two branching processes with immigration.  More specifically, we will construct processes $(X_j^+(t), t \geq 0)$ and $(X_j^-(t), t \geq 0)$ such that 
\begin{equation}\label{bpcoupling}
X_j^-(t) \leq X_j'(t + \tau_j) \leq X_j^+(t)
\end{equation}
for $t < \kappa_j$, where
\begin{equation}\label{kappajdef}
\kappa_j = \inf\bigg\{u: X_j^+(u) \geq \frac{s}{2 \mu} \bigg\} \wedge \big((\tau_{j+1} \wedge \zeta) - \tau_j\big).
\end{equation}
The processes $(X_j^+(t), t \geq 0)$ and $(X_j^-(t), t \geq 0)$ evolve according to the following rules.  First, $X_j^+(t)$ is the size at time $t$ of a population for which, at time $t$:
\begin{itemize}
\item New immigrants appear at rate $\phi_j^+(t) = (1 + \delta) s e^{s(q_j + C_4)t} \1_{\{t \leq \xi_j - \tau_j\}}$.

\item Each individual gives birth to a new individual at rate $\lambda_j^+ = 1 + s(q_j + C_4)$.

\item Each individual dies at rate $\nu_j^+ = 1 - s$.
\end{itemize}
Likewise, for the process $(X_j^-(t), t \geq 0)$, at time $t$:
\begin{itemize}
\item New immigrants appear at rate $\phi_j^-(t) = (1 - \delta) s e^{s(q_j - C_4)t} \1_{\{t \leq \xi_j - \tau_j\}}$.

\item Each individual gives birth to a new individual at rate $\lambda_j^- = 1 + s(q_j - C_4)$.

\item Each individual dies at rate $\nu_j^- = 1 + \mu$.
\end{itemize}

To establish that a coupling can be achieved so that (\ref{bpcoupling}) holds, we will give an explicit construction of the processes $(X_j^+(t), t \geq 0)$ and $(X_j^-(t), t \geq 0)$.  To do this, we will construct a population in which individuals are colored red, yellow, and blue.  We will let $X_j^+(t)$ be the total number of individuals at time $t$, and we will let $X_j^-(t)$ be the total number of red individuals at time $t$.  For $t < \kappa_j$, the number of individuals at time $t$ that are red or yellow will equal $X_j'(\tau_j + t)$, which we will refer to as the number of individuals in the ``original population".  We will number the individuals in our population by the order in which they were born.

The construction will require the original population process $({\bf X}(t), t \geq 0)$, as well as additional Poisson processes.  For each $i \in \N$, we will have Poisson processes $N_{b,i,j}$ and $N_{d,i,j}$ to help construct births and deaths and an additional Poisson processes $N_{m,j}$ to handle immigration.  These will be Poisson processes on $[0, \infty) \times [0, \infty)$ with Lebesgue intensity, which will be independent of one another and of the original population process.  We will also need a sequence $(\beta_{\ell, j})_{\ell=1}^{\infty}$ of independent random variables which are uniformly distributed on $(0,1)$ and are independent of $({\bf X}(t), t \geq 0)$ and the above Poisson processes.

We first construct our population up to time $\kappa_j$.  Observe, as we go through the construction, that the red population has immigration, birth, and death rates of $\phi_j^-(t)$, $\lambda_j^-$, and $\nu_j^-$ respectively, the total population has immigration, birth, and death rates of $\phi_j^+(t)$, $\lambda_j^+$, and $\nu_j^+$ respectively, and the red and yellow individuals stay in one-to-one correspondence with the original population.  This construction is well-defined because Lemma \ref{birthdeathlem} ensures that the rates described below are positive and the probabilities indicated below are between zero and one.
\begin{itemize}
\item If a type $j$ mutation occurs in the original population at time $\tau_j + t$, then an immigrant appears at time $t$.  This will be the $\ell$th change in the population for some positive integer $\ell$.  We color this immigrant red if $\beta_{\ell, j} \leq \phi_j^-(t-)/(\mu X_{j-1}(t-))$, and otherwise we color it yellow.  A blue immigrant appears at time $t$ if the Poisson process $N_{m,j}$ has a point $(t, x)$ with $x \leq \phi_j^+(t-) - \mu X_{j-1}(t-)$.

\item If the $i$th individual at time $t-$ is blue, then it gives birth to a blue individual at time $t$ if the Poisson process $N_{b,i,j}$ has a point $(t, x)$ with $x \leq \lambda_j^+$ and dies at time $t$ if there is a point $(t, x)$ in $N_{d,i,j}$ with $x \leq \nu_j^+$.

\item Suppose the $i$th individual at time $t-$ is red.  If the corresponding individual in the original population gives birth at time $\tau_j + t$ as part of a pure birth event, then the $i$th individual gives birth at time $t$.  This will be the $\ell$th change in the population for some $\ell$, and the new individual born will be red if $\beta_{\ell, j} \leq \lambda_j^-/B_j'(t-)$ and otherwise will be yellow.  If the corresponding individual in the original population gives birth at time $\tau_j + t$ as part of a birth and death event, then the $i$th individual gives birth to a yellow individual at time $t$.  The $i$th individual also gives birth to a blue individual at time $t$ if the Poisson process $N_{b,i,j}$ has a point at $(t, x)$ with $x \leq \lambda_j^+ - B_j^*(t-)$.

If the corresponding individual in the original population dies at time $\tau_j + t$ as part of a pure death event, then this will lead to the $\ell$th change in the population for some $\ell$, and the $i$th individual dies at time $t$ if $\beta_{\ell, j} \leq \nu_j^+/D_j'(t-)$ and otherwise turns blue.  If the corresponding individual in the original population dies at time $t$ as part of a birth and death event, then the $i$th individual turns blue at time $t$.  The $i$th individual also turns yellow at time $t$ if $N_{d,i,j}$ has a point at $(t, x)$ with $x \leq \nu_j^- - D_j^*(t-)$. 

\item Suppose the $i$th individual at time $t-$ is yellow.  If the corresponding individual in the original population gives birth at time $\tau_j + t$ as part of either a pure birth or a birth and death event, then a new yellow individual is born at time $t$.  The $i$th individual also gives birth to a blue individual at time $t$ if the Poisson process $N_{b,i,j}$ has a point at $(t, x)$ with $x \leq \lambda_j^+ - B_j^*(t-)$.

If the corresponding individual in the original population dies at time $t$ as part of a pure death event, then this will be the $\ell$th change in the population for some $\ell$, and the $i$th individual dies at time $t$ if $\beta_{\ell} \leq \nu_j^+/D_j'(t-)$ and otherwise turns blue.  If the corresponding individual dies at time $t$ as part of a birth and death event, then the $i$th individual turns blue at time $t$.
\end{itemize}
At time $\kappa_j$, the coupling with the original population is broken, and we make all yellow individuals blue.  After time $\kappa_j$, the process evolves as follows:
\begin{itemize}
\item If $\kappa_j < t \leq \xi_j$, then a red immigrant appears at time $t$ if there is a point $(t, x)$ of $N_{m,j}$ with $x \leq \phi_j^-(t-)$ and a blue immigrant appears at time $t$ if there is a point $(t, x)$ of $N_{m,j}$ with $\phi_j^-(t) < x \leq \phi_j^+(t)$.

\item If the $i$th individual is blue, it gives birth to a blue individual at time $t$ if $N_{b,i,j}$ has a point $(t, x)$ with $x \leq \lambda_j^+$ and dies at time $t$ if there is a point $(t, x)$ in $N_{d,i,j}$ with $x \leq \nu_j^+$.

\item Suppose the $i$th individual is red.  Then the $i$th individual gives birth to a red individual at time $t$ if the Poisson process $N_{b,i,j}$ has a point $(t, x)$ with $x \leq \lambda_j^-$ and to a blue individual at time $t$ if $N_{b,i,j}$ has a point $(t, x)$ with $\lambda_j^- < x \leq \lambda_j^+$.  Also, the $i$th individual dies at time $t$ if the Poisson process $N_{d,i,j}$ has a point $(t, x)$ with $x \leq \nu_j^+$ and turns blue at time $t$ if $N_{d,i,j}$ has a point $(t, x)$ with $\nu_j^+ < x \leq \nu_j^-$. 
\end{itemize}

For $j \in I$, let ${\cal H}_j$ be the $\sigma$-field generated by ${\cal F}_{\tau_j}$ along with the Poisson processes $N_{b, i, h}$, $N_{d,i,h}$, and $N_{m,h}$ and the random variables $\beta_{\ell, h}$ for $h < j$.  Because the immigration, birth, and death rates $\phi_j^+$, $\phi_j^-$, $\lambda_j^+$, $\lambda_j^-$, $\nu_j^+$, and $\nu_j^-$ are all ${\cal H}_j$-measurable, conditional on ${\cal H}_j$, the processes $(X_j^+(t), t \geq 0)$ and $(X_j^-(t), t \geq 0)$ are continuous-time branching processes with immigration, in which the immigration rate varies with time.

Let $$\tau_j' = \tau_j + \frac{3}{sq_j} \log \bigg( \frac{1}{sq_j} \bigg).$$  Note that $\tau_j < \xi_j < \tau_j'$ for sufficiently large $N$.  In view of (\ref{tauspacing}) and part 3 of Proposition \ref{newGq}, along with the fact that
$\log(s/\mu)/\log(1/sk_N) \rightarrow \infty$ as $N \rightarrow \infty$ by (\ref{musN}), we have $\tau_j' < \tau_{j+1}$ on $\{\zeta > \tau_j'\}$ if $N$ is sufficiently large.  Lemma \ref{kappabound} below helps to bound the probability that $\kappa_j < \tau_j'- \tau_j$ and therefore helps to ensure that with high probability, (\ref{bpcoupling}) holds up to time $\tau_j' - \tau_j$.  We will need the following bound on the mean of the branching process.

\begin{Lemma}\label{expbp}
For sufficiently large $N$, on $\{\tau_j < \zeta\}$, we have $$E[X_j^+(\tau_j' - \tau_j)|{\cal F}_{\tau_j}] \leq \frac{C}{s^3 k_N^4} \log \bigg( \frac{1}{sk_N} \bigg).$$
\end{Lemma}

\begin{proof}
Standard calculations involving supercritical branching processes give
\begin{align*}
E[X_j^+(\tau_j' - \tau_j)|{\cal F}_{\tau_j}] &= \int_0^{\tau_j' - \tau_j} \phi_j^+(u) e^{(\lambda_j^+ - \nu_j^+)(\tau_j' - \tau_j - u)} \: du \\
&= (1 + \delta) s \int_0^{\xi_j - \tau_j} e^{s(q_j + C_4) u} e^{s(q_j + C_4 + 1)(\tau_j' - \tau_j - u)} \: du \\
&= (1 + \delta) s e^{s(q_j + C_4 + 1)(\tau_j' - \tau_j)} \int_0^{\xi_j - \tau_j} e^{-su} \: du.
\end{align*}
Now $s(C_4 + 1)(\tau_j' - \tau_j) \rightarrow 0$ as $N \rightarrow \infty$ by the reasoning in (\ref{Aj8}), and $e^{sq_j(\tau_j' - \tau_j)} = (sq_j)^{-3}$.  Also,
$$\int_0^{\xi_j - \tau_j} e^{-su} \: du = \frac{1 - e^{-s(\xi_j - \tau_j)}}{s} \leq \xi_j - \tau_j = \frac{1}{sq_j} \log \bigg( \frac{1}{sq_j} \bigg) + \frac{b}{sq_j}.$$
Since $q_j \geq (1 - 2 \delta) k_N$ on $\{\tau_j < \zeta\}$ by part 3 of Proposition \ref{newGq}, the result follows. 
\end{proof}

\begin{Lemma}\label{kappabound}
We have $$\lim_{N \rightarrow \infty} P \bigg( \Lambda \cap \bigcup_{j \in I} \{\kappa_j \leq \tau_j' - \tau_j\} \bigg) = 0.$$
\end{Lemma}

\begin{proof}
In view of Lemma \ref{Ilem}, for $j \in I$ we have $\tau_j' < \tau_{j+1} < \zeta$ on $\Lambda$.  Therefore, for $j \in I$, on $\Lambda$ the only way to have $\kappa_j \leq \tau_j' - \tau_j$ would be to have $X_j^+(t) > s/2 \mu$ for some $t \leq \tau_j' - \tau_j$.  Because $(X_j^+(t), t \geq 0)$ is a submartingale, it follows from Doob's Maximal Inequality and Lemma \ref{expbp} that
\begin{align}\label{kbe1}
P(\Lambda \cap \{\kappa_j \leq \tau_j' - \tau_j\}|{\cal F}_{\tau_j}) &\leq P \bigg( \sup_{0 \leq t \leq \tau_j' - \tau_j} X_j^+(t) > \frac{s}{2 \mu} \bigg| {\cal F}_{\tau_j} \bigg) \1_{\{\zeta > \tau_j\}} \nonumber \\
&\leq \frac{2 \mu}{s} E[X_j^+(\tau_j' - \tau_j)|{\cal F}_{\tau_j}] \1_{\{\zeta > \tau_j\}} \nonumber \\
&\leq \frac{C \mu}{s^4 k_N^4} \log \bigg( \frac{1}{sk_N} \bigg).
\end{align}
Summing over $j \in I$, and then using (\ref{musN}) and the fact that the cardinality of $I$ is at most $3T k_N$ by Lemma \ref{Ilem}, we obtain the result.
\end{proof}

\subsection{The probability that a family survives}

Here we use the branching process coupling introduced in the previous subsection to obtain upper and lower bounds on the probability that an individual will acquire a $j$th mutation before time $\xi_j$ and have descendants surviving a long time into the future.  

\begin{Lemma}\label{pmlem1}
Suppose $j \in I$, where $j$ is possibly random, and $\tau_j$ is a stopping time.  Define ${\cal H}_j$ as in subsection \ref{bpcouplesec}.  On the event $\{\tau_j < \zeta\}$, we have for sufficiently large $N$,
\begin{equation}\label{bpineq}
\frac{(1 - 2 \delta) e^b}{q_j} \leq P(X_j^-(\tau_j' - \tau_j) > 0|{\cal H}_j) \leq P(X_j^+(\tau_j' - \tau_j) > 0|{\cal H}_j) \leq \frac{(1 + 2 \delta) e^b}{q_j}.
\end{equation}
Also, letting $L_j^-$ and $L_j^+$ denote the numbers of immigrants in $(X_j^-(t), t \geq 0)$ and $(X_j^+(t), t \geq 0)$ respectively that have descendants alive at time $\tau_j' - \tau_j$, for sufficiently large $N$ on $\{\tau_j < \zeta\}$ we have
\begin{equation}\label{2immfam}
P(L_j^+ \geq 2|{\cal H}_j) \leq \frac{2e^{2b}}{q_j^2}.
\end{equation}
\end{Lemma}

\begin{proof}
Throughout the proof, we work on the event $\{\tau_j < \zeta\}$.
Because $X_j^-(t) \leq X_j^+(t)$ for all $t \geq 0$, the second inequality in (\ref{bpineq}) is obvious.  We now prove the third inequality.  By (\ref{survivetot}), the probability that an immigrant in the branching process $(X_j^+(t), t \geq 0)$ at time $u$ has descendants that survive until time $\tau_j' - \tau_j$ is 
$$\frac{\lambda_j^+ - \nu_j^+}{\lambda_j^+ - \nu_j^+ e^{-(\lambda_j^+ - \nu_j^+)(\tau_j' - \tau_j - u)}}.$$
Now $\lambda_j^+ - \nu_j^+ = s(q_j + C_4 + 1)$.  Also, for sufficiently large $N$,
$$\tau_j' - \xi_j = \frac{2}{s q_j} \log \bigg(\frac{1}{s q_j} \bigg) - \frac{b}{sq_j} \geq \frac{3}{2sq_j} \log \bigg( \frac{1}{sq_j} \bigg).$$ Therefore, if $u < \xi_j - \tau_j$, then
$$\nu_j^+ e^{-(\lambda_j^+ - \nu_j^+)(\tau_j' - \tau_j - u)} \leq (1 - s) e^{-s(q_j + C_4 + 1)(\tau_j' - \xi_j)} \leq e^{-sq_j(\tau_j' - \xi_j)} \leq (sq_j)^{3/2},$$ which, in view of part 3 of Proposition \ref{newGq} and assumption A3, implies that for sufficiently large $N$,
$$\frac{\lambda_j^+ - \nu_j^+}{\lambda_j^+ - \nu_j^+ e^{-(\lambda_j^+ - \nu_j^+)(\tau_j' - \tau_j - u)}} \leq \frac{s(q_j + C_4 + 1)}{1 + s(q_j + C_4) - (sq_j)^{3/2}} \leq s(q_j + C_4 + 1).$$  Therefore,
\begin{align*}
E[L_j^+|{\cal H}_j] &= \int_0^{\xi_j - \tau_j} \phi_j^+(u) \cdot \frac{\lambda_j^+ - \nu_j^+}{\lambda_j^+ - \nu_j^+ e^{-(\lambda_j^+ - \nu_j^+)(\tau_j' - \tau_j - u)}} \: du \\
&\leq \int_0^{\xi_j - \tau_j} (1 + \delta) s e^{s(q_j + C_4) u} \cdot s(q_j + C_4 + 1) \: du \\
&= (1 + \delta) s^2 (q_j + C_4 + 1) \bigg( \frac{e^{s(q_j + C_4) (\xi_j - \tau_j)} - 1}{s(q_j + C_4)} \bigg) \\
&\leq (1 + \delta) \bigg( \frac{q_j + C_4 + 1}{q_j + C_4} \bigg) s e^{s(q_j + C_4)(\xi_j - \tau_j)}.
\end{align*}
Because $e^{sq_j(\xi_j - \tau_j)} = e^b/(sq_j)$ and $C_4 s (\xi_j - \tau_j) \rightarrow 0$ as $N \rightarrow \infty$ by (\ref{Aj8}), it follows that for sufficiently large $N$,
\begin{equation}\label{Ljexp}
E[L_j^+|{\cal H}_j] \leq \frac{(1 + 2 \delta) e^b}{q_j}.
\end{equation}
The conditional Markov's Inequality now gives the third inequality in (\ref{bpineq}).  Because the conditional distribution of $L_j^+$ given ${\cal H}_j$ is Poisson, we have $P(L_j^+ \geq 2|{\cal H}_j) \leq (E[L_j^+|{\cal H}_j])^2$.  Therefore, (\ref{2immfam}) also follows from (\ref{Ljexp}).

It remains to prove the first inequality in (\ref{bpineq}).  The argument is similar to that for the third inequality, but we will need a lower bound on the expectation.  For sufficiently large $N$,
\begin{align}\label{ELj-}
E[L_j^-|{\cal H}_j] &= \int_0^{\xi_j - \tau_j} \phi_j^-(u) \cdot \frac{\lambda_j^- - \nu_j^-}{\lambda_j^- - \nu_j^- e^{-(\lambda_j^- - \nu_j^-)(\tau_j' - \tau_j - u)}} \: du \nonumber \\
&\geq \int_0^{\xi_j - \tau_j} \phi_j^-(u) \cdot \frac{\lambda_j^- - \nu_j^-}{\lambda_j^-} \: du \nonumber \\
&= \int_0^{\xi_j - \tau_j} (1 - \delta) s e^{s(q_j - C_4)u} \cdot \frac{s(q_j - C_4) - \mu}{1 + s(q_j - C_4)} \: du \nonumber \\
&= \frac{(1 - \delta) s^2 (q_j - C_4 - \mu/s)}{1 + s(q_j - C_4)} \bigg( \frac{e^{s(q_j - C_4)(\xi_j - \tau_j)} - 1}{s(q_j - C_4)} \bigg) \nonumber \\
&\geq \frac{(1 - (3/2) \delta) e^b}{q_j}.
\end{align}
Because the conditional distribution of $L_j^-$ given ${\cal H}_j$ is Poisson, we have $$P(X_j^-(\tau_j' - \tau_j) > 0|{\cal H}_j) = P(L_j^- > 0|{\cal H}_j) = 1 - e^{-E[L_j^-|{\cal H}_j]} \geq E[L_j^-|{\cal H}_j] - (E[L_j^-|{\cal H}_j])^2.$$  The first inequality in (\ref{bpineq}) follows from this result and (\ref{ELj-}).
\end{proof}

\subsection{The size of a surviving family}

The lemma below bounds the probability that some individual will acquire a $j$th mutation before time $\xi_j$ and have at least $x e^{sq_j(\tau_j' - \tau_j)}$ descendants alive at time $\tau_j'$.  Recall from (\ref{prop22}) and part 1 of Proposition \ref{newGq} that $e^{sq_j(\tau_j' - \tau_j)}$ is approximately the number of type $j$ individuals that we would expect there to be in the population in the absence of such an early type $j$ mutation.  This result is the precise version of (\ref{earlymut}), which is the key to understanding why the Bolthausen-Sznitman coalescent describes the genealogy of the population.

\begin{Lemma}\label{pmlem2}
Fix $j \in I$, and recall the definition of ${\cal H}_j$ from subsection \ref{bpcouplesec}.
For sufficiently large $N$, on $\{\tau_j < \zeta\}$, we have for all $x \in [\delta/2, 2/\delta]$,
\begin{equation}\label{bptail1}
P(X_j^-(\tau_j' - \tau_j) > x e^{sq_j(\tau_j' - \tau_j)}|{\cal H}_j) \geq \frac{1 - 7 \delta}{q_j x}
\end{equation}
and for all $x \in [e^{-b}, 2/\delta]$,
\begin{equation}\label{bptail2}
P(X_j^+(\tau_j' - \tau_j) > x e^{sq_j(\tau_j' - \tau_j)}|{\cal H}_j) \leq \frac{1 + 7 \delta}{q_j x}.
\end{equation}
\end{Lemma}

\begin{proof}
Throughout the proof, we work on the event $\{\tau_j < \zeta\}$.
We first prove (\ref{bptail2}).  Suppose $x \in [e^{-b}, 2/\delta]$.  If $X_j^+(\tau_j' - \tau_j) > x e^{sq_j(\tau_j' - \tau_j)}$, then either two immigrants in the population have descendants alive at time $\tau_j' - \tau_j$, an event whose probability has already been bounded above in (\ref{2immfam}), or else for some $u \in (0, \xi_j - \tau_j]$, an immigrant arrives at time $u$ and has more than $x e^{sq_j(\tau_j' - \tau_j)}$ descendants at time $\tau_j' - \tau_j$.  Note that
\begin{equation}\label{Weq0}
|(\lambda_j^+ - \nu_j^+) - sq_j|(\tau_j' - \tau_j) = s(C_4 + 1) \cdot \frac{3}{sq_j} \log \bigg( \frac{1}{sq_j} \bigg) \rightarrow 0
\end{equation}
as $N \rightarrow \infty$ by the reasoning in (\ref{Aj8}).  Therefore, for sufficiently large $N$, we have
\begin{equation}\label{Weq1}
x e^{sq_j(\tau_j' - \tau_j)} \geq (1 - \delta) x e^{(\lambda_j^+ - \nu_j^+)(\tau_j' - \tau_j)}.
\end{equation}
Suppose an immigrant arrives at time $u$, and let $X_{j,u}^+(t)$ be the number of descendants of this immigrant in the population at time $t$.  For $t \geq 0$, let
\begin{equation}\label{Weq2}
W_u^+(t) = e^{-(\lambda_j^+ - \nu_j^+)t} X_{j,u}^+(t + u),
\end{equation}
and let $W^+ = \lim_{t \rightarrow \infty} W_u^+(t)$, which exists by (\ref{Wlim}).  Equations (\ref{Weq1}) and (\ref{Weq2}) imply that for the immigrant to have more than $xe^{sq_j(\tau_j' - \tau_j)}$ descendants in the population at time $\tau_j' - \tau_j$, if $N$ is sufficiently large we must have
\begin{equation}\label{Weq3}
W_u^+(\tau_j' -\tau_j - u) \geq (1 - \delta) x e^{(\lambda_j^+ - \nu_j^+)u}.
\end{equation}
To estimate the probability that this occurs, observe that by Lemma \ref{bplem2} and (\ref{survivalprob})
\begin{align*}
P\big(|W^+ - W_u^+(\tau_j' - \tau_j - u)| > \delta x e^{(\lambda_j^+ - \nu_j^+)u} \big) &\leq \frac{2 \lambda_j^+ e^{-(\lambda_j^+ - \nu_j^+)(\tau_j' - \tau_j - u)}}{\delta^2 x^2 e^{2(\lambda_j^+ - \nu_j^+)u} (\lambda_j^+ - \nu_j^+)} \\
&\leq \frac{2(1 + s(q_j + C_4)) e^{-(\lambda_j^+ - \nu_j^+)(\tau_j' - \tau_j)}}{\delta^2 x^2 s(q_j + C_4 + 1)}.
\end{align*}
Since $e^{-(\lambda_j^+ - \nu_j^+)(\tau_j' - \tau_j)} \leq e^{-sq_j(\tau_j' - \tau_j)} = (sq_j)^3$, it follows that for sufficiently large $N$,
\begin{equation}\label{Werror}
P\big(|W^+ - W^+_u(\tau_j' - \tau_j - u)| > \delta x e^{(\lambda_j^+ - \nu_j^+)u} \big) \leq \frac{3 (sq_j)^2}{\delta^2 x^2}.
\end{equation}
Note that $\lambda_j^+ - \nu_j^+ \geq sq_j$, and $(1 - \delta/2)sq_j \leq (\lambda_j^+ - \nu_j^+)/\lambda_j^+ \leq (1 + \delta)sq_j$ for sufficiently large $N$.  Therefore, by (\ref{survivalprob}) and (\ref{Wexp}), for sufficiently large $N$,
\begin{align}\label{Wtailprob}
P \big( W^+ > (1 - 2 \delta) x e^{(\lambda_j^+ - \nu_j^+)u} \big) &= \bigg( \frac{\lambda_j^+ - \nu_j^+}{\lambda_j^+} \bigg) e^{-(1 - 2 \delta) x e^{(\lambda_j^+ - \nu_j^+)u} (\lambda_j^+ - \nu_j^+)/\lambda_j^+} \nonumber \\
&\leq (1 + \delta) sq_j e^{-(1 - 3 \delta) sq_j x e^{sq_j u}}.
\end{align}
The probability of the event in (\ref{Weq3}) is bounded above by the sum of the expressions in (\ref{Werror}) and (\ref{Wtailprob}).  Thus, combining this result with (\ref{2immfam}), we have
\begin{align}\label{Weq4}
&P(X_j^+(\tau_j' - \tau_j) > x e^{sq_j(\tau_j' - \tau_j)}|{\cal H}_j) \nonumber \\
&\hspace{.3in}\leq \frac{2e^{2b}}{q_j^2} + \int_0^{\xi_j - \tau_j} (1 + \delta) s e^{s(q_j + C_4)u} \bigg((1 + \delta) s q_j e^{-(1 - 3 \delta) sq_j x e^{sq_j u}} + \frac{3(sq_j)^2}{\delta^2 x^2} \bigg) \: du.
\end{align}
Using that $e^{sq_j(\xi_j - \tau_j)} = e^b/(sq_j)$ and that $s(\xi_j - \tau_j) \rightarrow 0$ as $N \rightarrow \infty$ by (\ref{Aj8}), we have, for sufficiently large $N$,
\begin{equation}\label{Weq5}
\int_0^{\xi_j - \tau_j} (1 + \delta) s e^{s(q_j + C_4)u} \cdot \frac{3(s q_j)^2}{\delta^2 x^2} \: du \leq \frac{4 e^b s}{\delta^2 x^2}.
\end{equation}
Also, making the substitution $y = (1 - 3 \delta) sq_j x e^{sq_j u}$, so that $dy/du = sq_j y$, and using again that $s(\xi_j - \tau_j) \rightarrow 0$ as $N \rightarrow \infty$, for sufficiently large $N$ we have
\begin{align}\label{Weq6}
&\int_0^{\xi_j - \tau_j} (1 + \delta) s e^{s(q_j + C_4)u} \cdot (1 + \delta) s q_j e^{-(1 - 3 \delta) sq_j x e^{sq_j u}} \: du \nonumber \\
&\hspace{.5in}\leq (1 + \delta)^2 s^2 q_j e^{C_4 s (\xi_j - \tau_j)} \int_0^{\xi_j - \tau_j} e^{sq_j u} e^{-(1 - 3 \delta) sq_j x e^{sq_j u}} \: du \nonumber \\
&\hspace{.5in}= (1 + \delta)^2 s^2 q_j e^{C_4 s (\xi_j - \tau_j)} \int_{(1 - 3 \delta) s q_j x}^{(1 - 3 \delta) sq_j x e^{sq_j(\xi_j - \tau_j)}} \frac{e^{-y}}{(1 - 3 \delta) s^2 q_j^2 x} \: dy \nonumber \\
&\hspace{.5in}\leq \frac{1 + 6 \delta}{q_j x}.
\end{align}
From (\ref{Weq4}), (\ref{Weq5}), and (\ref{Weq6}), we get
\begin{equation}\label{newWeq}
P(X_j^+(\tau_j' - \tau_j) > x e^{sq_j(\tau_j' - \tau_j)}|{\cal H}_j) \leq \frac{1}{q_j x} \bigg(1 + 6 \delta + \frac{4 e^b s q_j}{\delta^2 x} + \frac{2 e^{2b} x}{q_j} \bigg).
\end{equation}
Recall that $(1 - 2 \delta) k_N \leq q_j \leq (e + 2 \delta)k_N$ on $\{\tau_j < \zeta\}$ by part 3 of Proposition \ref{newGq}.  Since $k_N \rightarrow \infty$ and $sk_N \rightarrow 0$ as $N \rightarrow \infty$ by assumptions A1 and A3 respectively, the upper bound (\ref{bptail2}) follows from (\ref{newWeq}).

Next, we will suppose $x \in [\delta/2, 2/\delta]$ and show (\ref{bptail1}) by similar arguments.  We consider only the individuals colored red in the construction given above.  Suppose a red immigrant arrives at time $u$.  Then let $X_{j,u}^-(t)$ denote the number of red descendants of this immigrant at time $t$, and for $t \geq 0$, let $$W_u^-(t) = e^{-(\lambda_j^- - \nu_j^-)t} X_{j,u}^-(t + u).$$  Let $W^- = \lim_{t \rightarrow \infty} W_u^-(t)$.  Because $|(\lambda_j^- - \nu_j^-) - sq_j| \rightarrow 0$ as $N \rightarrow \infty$ by the reasoning in (\ref{Weq0}), the reasoning that led to (\ref{Weq3}) implies that if
\begin{equation}\label{Weq8}
W_u^-(\tau_j' - \tau_j - u) \geq (1 + \delta) x e^{(\lambda_j^- - \nu_j^-)u}
\end{equation}
and $N$ is large enough, then we must have $X_j^-(\tau_j' - \tau_j) > x e^{sq_j(\tau_j' - \tau_j)}$.  Because $s(\tau_j' - \tau_j) \rightarrow 0$ as $N \rightarrow \infty$ by the reasoning in (\ref{Aj8}), we have $$e^{-(\lambda_j^- - \nu_j^-)(\tau_j' - \tau_j)} \leq (1 + \delta) (sq_j)^3$$ for sufficiently large $N$.  Therefore, by the reasoning leading to (\ref{Werror}), for sufficiently large $N$ we have
\begin{equation}\label{Werror2}
P\big(|W^- - W^-_u(\tau_j' - \tau_j - u)| > \delta x e^{(\lambda_j^- - \nu_j^-)u} \big) \leq \frac{3 (sq_j)^2}{\delta^2 x^2}.
\end{equation}
Note that $\lambda_j^- - \nu_j^- \leq sq_j$, and $(1 - \delta)s q_j \leq (\lambda_j^- - \nu_j^-)/\lambda_j^- \leq (1 + \delta/2) sq_j$ for sufficiently large $N$.  Therefore, by (\ref{Wexp})
\begin{align}\label{Weq7}
P\big( W^- > (1 + 2 \delta) x e^{(\lambda_j^- - \nu_j^-)u} \big) &= \bigg( \frac{\lambda_j^- - \nu_j^-}{\lambda_j^-} \bigg) e^{-(1 + 2 \delta) x e^{(\lambda_j^- - \nu_j^-) u} (\lambda_j^- - \nu_j^-)/\lambda_j^-} \nonumber \\
&\geq (1 - \delta) e^{-(1 + 3 \delta) sq_j x e^{sq_j u}}. 
\end{align}
By using (\ref{Werror2}) and (\ref{Weq7}) to bound from below the probability in (\ref{Weq8}), we get that for sufficiently large $N$,
\begin{align}\label{Weq9}
&P(X_j^-(\tau_j' - \tau_j) > x e^{sq_j(\tau_j' - \tau_j)}|{\cal H}_j) \nonumber \\
&\hspace{.5in}\geq \int_0^{\xi_j - \tau_j} (1 - \delta) s e^{s(q_j - C_4)} \bigg( (1 - \delta) e^{-(1 + 3 \delta) sq_j x e^{sq_j u}} - \frac{3(sq_j)^2}{\delta^2 x^2} \bigg) \: du.
\end{align}
Following the reasoning in (\ref{Weq6}), this time using the substitution $y = (1 + 3 \delta) sq_j x e^{sq_j u}$, we get
\begin{align}\label{Weq10}
&\int_0^{\xi_j - \tau_j} (1 - \delta) s e^{s(q_j - C_4)} \cdot (1 - \delta) e^{-(1 + 3 \delta) sq_j x e^{sq_j u}} \: du \nonumber \\
&\hspace{1.2in} \geq \frac{1 - 6 \delta}{q_j x} \int_{(1 + 3 \delta) sq_j x}^{(1 + 3 \delta) sq_j x e^{sq_j (\xi_j - \tau_j)}} e^{-y} \: dy  \nonumber \\
&\hspace{1.2in} = \frac{1 - 6 \delta}{q_j x} \big( e^{-(1 + 3 \delta) sq_j x} - e^{-(1 + 3 \delta) e^b x} \big) \nonumber \\
&\hspace{1.2in} \geq \frac{1 - 6 \delta}{q_j x} \big(1 - (1 + 3 \delta) sq_j x - e^{-e^b x} \big).
\end{align}
On $\{\tau_j < \zeta\}$, by part 3 of Proposition \ref{newGq}, we have $sq_j x \leq 2 (e + 2\delta) s k_N/\delta \rightarrow 0$ as $N \rightarrow \infty$. Also, using the definition of $b$ from (\ref{bdef}), we have $e^{-e^b x} \leq e^{-12000T/(\delta \eps)}$.  Therefore, using (\ref{Weq10}) to bound the first term in (\ref{Weq9}), and using the reasoning of (\ref{Weq5}) to bound the second term, we obtain (\ref{bptail1}).
\end{proof}

In view of (\ref{bpcoupling}), Lemmas \ref{pmlem1} and \ref{pmlem2} show that the number of early type $j$ individuals is well-approximated up to time $\tau_j'$ by a continuous-time branching process.  The result below tells us that the number of early type $j$ individuals at time $\tau_{j+1}$ is usually determined, to within a small error, by the number of such individuals at time $\tau_j'$.

\begin{Lemma}\label{bptail3}
For $j \in I$, define the event
\begin{equation}\label{Ajprime}
A_j = \big\{ \big| e^{-sq_j(\tau_j' - \tau_j)} X'_j(\tau_j') - e^{-\int_{\tau_j}^{\tau_{j+1}} G_j(v) \: dv} X_j'(\tau_{j+1}) \big| > e^{-b} \big\}.
\end{equation}
Then
$$\lim_{N \rightarrow \infty} P \bigg( \Lambda \cap \bigcup_{j\in I} A_j \bigg) = 0.$$
\end{Lemma}

\begin{proof}
Let $S$ be the set of individuals at time $\tau_j'$ descended from individuals that acquired their $j$th mutation during the time interval $(\tau_j, \xi_j]$, which means there are $X_j'(\tau_j')$ individuals in the set $S$.  Then, using the notation of Corollary \ref{ZmartCor3} with $\tau_j'$ in place of $\kappa$, we get that for $t \geq \tau_j'$,
\begin{equation}\label{vtail1}
e^{-\int_{\tau_j'}^{t \wedge \tau_{j+1} \wedge \zeta} G_j(v) \: dv} X_j'(t \wedge \tau_{j+1} \wedge \zeta) = X_j'(\tau_j') + Z_j^{S}(t),
\end{equation}
where $(Z_j^{S}(\tau_j' + t), t \geq 0)$ is a mean zero martingale.  Therefore, on $\{\tau_{j+1} < \zeta\}$, we have
\begin{align}\label{vtail3}
e^{-sq_j (\tau_j' - \tau_j)} X_j'(\tau_j') &= e^{-sq_j(\tau_j' - \tau_j) - \int_{\tau_j'}^{\tau_{j+1}} G_j(v) \: dv} X_j'(\tau_{j+1}) - e^{-sq_j(\tau_j' - \tau_j)} Z_j^{S}(\tau_{j+1}) \nonumber \\
&= e^{\int_{\tau_j}^{\tau_j'} (G_j(v) - sq_j) \: dv} e^{-\int_{\tau_j}^{\tau_{j+1}} G_j(v) \: dv} X_j'(\tau_{j+1}) - e^{-sq_j(\tau_j' - \tau_j)} Z_j^{S}(\tau_{j+1}).
\end{align}
By (\ref{prop21}), on $\{\tau_{j+1} < \zeta\}$, we have $e^{-\int_{\tau_j}^{\tau_{j+1}} G_j(v) \: dv} X_j'(\tau_{j+1}) \leq C_1$.  Also, by part 1 of Proposition \ref{newGq}, on $\{\tau_{j+1} < \zeta\}$, we have
\begin{equation}\label{vtail4}
\int_{\tau_j}^{\tau_j'} |G_j(v) - sq_j| \: dv \leq C_3 s (\tau_j' - \tau_j),
\end{equation}
which tends to zero as $N \rightarrow \infty$ by the argument in (\ref{Aj8}).  Thus, (\ref{vtail3}) implies that for sufficiently large $N$, on $\{\tau_{j+1} < \zeta\}$, we have
\begin{equation}\label{vtail8}
\big| e^{-sq_j(\tau_j' - \tau_j)} X'_j(\tau_j') - e^{-\int_{\tau_j}^{\tau_{j+1}} G_j(v) \: dv} X_j'(\tau_{j+1}) \big| \leq \frac{e^{-b}}{2} + e^{-sq_j(\tau_j' - \tau_j)}|Z_j^{S}(\tau_{j+1})|.
\end{equation}

It remains to bound $|Z_j^{S}(\tau_{j+1})|$.  By Corollary \ref{ZmartCor3} and the argument leading to (\ref{BD3}),
\begin{equation}\label{vtail2}
\Var(Z_j^{S}(\tau_j' + t)|{\cal F}_{\tau_j'}) \leq 3 E \bigg[ \int_{\tau_j'}^{(\tau_j' + t) \wedge \tau_{j+1} \wedge \zeta} e^{-2 \int_{\tau_j'}^u G_j(v) \: dv} X_j'(u) \: du \bigg| {\cal F}_{\tau_j'} \bigg].
\end{equation}
Because $G_j(v) \geq s(q_j - C_3)$ for $v \in [\tau_j, \tau_{j+1} \wedge \zeta)$ by part 1 of Proposition \ref{newGq}, it follows from equations (\ref{vtail1}) and (\ref{vtail2}), Fubini's Theorem, and the fact that $(Z_j^{S}(\tau_j' + t), t \geq 0)$ is a mean zero martingale that for sufficiently large $N$,
\begin{align*}
\Var(Z_j^{S}(\tau_j' + t)|{\cal F}_{\tau_j'}) &\leq 3 E \bigg[ \int_{\tau_j'}^{(\tau_j' + t) \wedge \tau_{j+1} \wedge \zeta} e^{-\int_{\tau_j'}^u G_j(v) \: dv} (X'_j(\tau_j') + Z_j^{S}(t)) \: du \bigg| {\cal F}_{\tau_j'} \bigg] \\
&\leq 3 E \bigg[ \int_{\tau_j'}^{(\tau_j' + t) \wedge \tau_{j+1} \wedge \zeta} e^{-s(q_j - C_3)(u - \tau_j')}(X'_j(\tau_j') + Z_j^{S}(t)) \: du \bigg| {\cal F}_{\tau_j'} \bigg] \\
&\leq 3 X'_j(\tau_j') \int_{\tau_j'}^{\infty} e^{-s(q_j - C_3)(u - \tau_j')} \: du \\
&\leq \frac{4 X'_j(\tau_j')}{s q_j}.
\end{align*}
Therefore, by the $L^2$ Maximum Inequality for martingales,
\begin{equation}\label{vtail5}
P \bigg( \big|Z_j^{S}(\tau_{j+1})\big| > \frac{e^{-b}}{2} e^{sq_j (\tau_j' - \tau_j)} \bigg| {\cal F}_{\tau_j'} \bigg) \leq \frac{C X'_j(\tau_j')}{sq_j e^{-2b}} e^{-2sq_j(\tau_j' - \tau_j)} = C X_j'(\tau_j') (sq_j)^5 e^{2b}.
\end{equation}
On $\{\kappa_j > \tau_j' - \tau_j\}$, we have $X_j'(\tau_j') \leq X_j^+(\tau_j' - \tau_j)$ by (\ref{bpcoupling}).  Let ${\cal F}_{\tau_j'}^*$ be the $\sigma$-field generated by ${\cal F}_{\tau_j'}$ and the event $\{\kappa_j > \tau_j' - \tau_j\}$.  Since the additional Poisson processes $N_{b,i,j}$, $N_{d,i,j}$, and $N_{m,j}$ and random variables $\beta_{\ell,j}$ are independent of the population process $({\bf X}(t), t \geq 0)$, we have on $\{\kappa_j > \tau_j' - \tau_j\}$,
\begin{align*}
P \bigg( |Z_j^S(\tau_{j+1})| > \frac{e^{-b}}{2} e^{sq_j(\tau_j' - \tau_j)} \bigg| {\cal F}_{\tau_j'}^* \bigg) \leq C X_j^+(\tau_j' - \tau_j) (sk_N)^5 e^{2b}.
\end{align*}
Therefore, taking conditional expectations of both sides of (\ref{vtail5}) with respect to ${\cal F}_{\tau_j}$ and then using 
Lemma \ref{expbp} and part 3 of Proposition \ref{newGq}, we get
\begin{align}\label{vtail9}
P \bigg( \{\kappa_j > \tau_j' - \tau_j\} \cap \bigg\{ \big|Z_j^{S}(\tau_{j+1})\big| > \frac{e^{-b}}{2} e^{sq_j (\tau_j' - \tau_j)} \bigg\} \bigg| {\cal F}_{\tau_j} \bigg) &\leq C E[X_j^+(\tau_j' - \tau_j)|{\cal F}_{\tau_j}] (sk_N)^5 e^{2b} \nonumber \\
&\leq \frac{C (sk_N)^2 e^{2b}}{k_N} \log \bigg( \frac{1}{sk_N} \bigg).
\end{align}
Using Boole's Inequality and summing over $j \in I$, we now deduce from equations (\ref{vtail8}) and (\ref{vtail9}) and Lemmas \ref{Ilem} and \ref{kappabound} that
$$P \bigg( \Lambda \cap \bigcup_{j \in I} A_j \bigg) \leq 3 T k_N \cdot \frac{C (sk_N)^2 e^{2b}}{k_N} \log \bigg( \frac{1}{sk_N} \bigg),$$
which tends to zero as $N \rightarrow \infty$ by assumption A3.
\end{proof}

\subsection{The fraction of individuals descended from an early mutation}

To determine the genealogy of the population, it will be important to consider the fraction of type $j$ individuals in the population descended from an early type $j$ mutation, as this is an estimate of the fraction of lineages that will coalesce near the time of this mutation.  To this end, we let
\begin{equation}\label{Yjdef}
Y_j = \frac{X_j'(\tau_{j+1})}{\lceil s/\mu \rceil},
\end{equation}
which is the fraction of type $j$ individuals at time $\tau_{j+1}$ that are descended from a type $j$ mutation that occurred between times $\tau_j$ and $\xi_j$.  Also, define
$$Y_j^- = \frac{(e^{-sq_j(\tau_j' - \tau_j)} X_j^-(\tau_j' - \tau_j) - e^{-b}) \vee 0}{((e^{-sq_j(\tau_j' - \tau_j)} X_j^-(\tau_j' - \tau_j) - e^{-b}) \vee 0) + 1 + 4 \delta}$$ and
\begin{equation}\label{Yjplus}
Y_j^+ = \frac{e^{-sq_j(\tau_j' - \tau_j)} X_j^+(\tau_j' - \tau_j) + e^{-b}}{e^{-sq_j(\tau_j' - \tau_j)} X_j^+(\tau_j' - \tau_j) + e^{-b} + 1 - 4 \delta}.
\end{equation}

\begin{Lemma}\label{Ylem}
Suppose $j \in I$.  For sufficiently large $N$, on $\{\tau_j < \zeta\}$, we have, for all $y \in [\delta, 1 - \delta]$,
\begin{equation}\label{Ylem2}
\frac{(1 - y)(1 - 13 \delta)}{q_j y} \leq P(Y_j^- \geq y|{\cal H}_j) \leq P(Y_j^+ \geq y|{\cal H}_j) \leq \frac{(1 - y)(1 + 13 \delta)}{q_j y}.
\end{equation}
Also, on $A_j^c \cap \{\tau_{j+1} < \zeta\} \cap \{\kappa_j > \tau_j' - \tau_j\}$, we have
\begin{equation}\label{Ylem1}
Y_j^- \leq Y_j \leq Y_j^+.
\end{equation}
\end{Lemma}

\begin{proof}
We first prove (\ref{Ylem2}).  Suppose $y \in [\delta, 1 - \delta]$.  The middle inequality in (\ref{Ylem2}) is immediate.  To prove the third inequality in (\ref{Ylem2}), note that $Y_j^+ \geq y$ if and only if
\begin{equation}\label{ym3}
e^{-sq_j(\tau_j' - \tau_j)} X_j^+(\tau_j' - \tau_j) \geq \frac{(1 + e^{-b} - 4 \delta)y - e^{-b}}{1 - y}.
\end{equation}
Since $e^{-b}/y \leq \delta$ by (\ref{bdef}), we see that (\ref{ym3}) implies $e^{-sq_j(\tau_j' - \tau_j)} X_j^+(\tau_j' - \tau_j) \geq (1 - 5 \delta)y/(1-y)$.  Thus, by Lemma \ref{pmlem2}, for sufficiently large $N$, on the event $\{\tau_j < \zeta\}$, we have for all $y \in [\delta, 1 - \delta]$,
$$P(Y_j^+ \geq y|{\cal H}_j) \leq P \bigg( e^{-sq_j(\tau_j' - \tau_j)} X_j^+(\tau_j' - \tau_j) \geq \frac{(1 - 5 \delta)y}{1-y} \bigg| {\cal H}_j \bigg) \leq \frac{(1 + 7 \delta)(1 - y)}{(1 - 5 \delta) q_j y},$$ which leads to the third inequality in (\ref{Ylem2}).  Likewise, note that $Y_j^- \geq y$ if and only if $$e^{-sq_j(\tau_j' - \tau_j)} X_j^-(\tau_j' - \tau_j) \geq \frac{(1 - e^{-b} + 4 \delta)y + e^{-b}}{1 - y},$$ which, since $e^{-b}/y \leq \delta$, will always hold if $e^{-sq_j(\tau_j' - \tau_j)} X_j^+(\tau_j' - \tau_j) \geq (1 + 5 \delta)y/(1-y)$.  Therefore, by Lemma \ref{pmlem2},
$$P(Y_j^- \geq y|{\cal H}_j) \geq P \bigg( e^{-sq_j(\tau_j' - \tau_j)} X_j^+(\tau_j' - \tau_j) \geq \frac{(1 + 5 \delta)y}{1-y} \bigg| {\cal H}_j \bigg) \geq \frac{(1 - 7 \delta)(1 - y)}{(1 + 5 \delta) q_j y},$$
which implies the first inequality in (\ref{Ylem2}).  It remains to prove (\ref{Ylem1}).

The last statement of part 1 of Proposition \ref{mainjprop}, combined with (\ref{tauspacing}), implies that on the event $\{\tau_{j+1} < \zeta\}$, no individual that gets a $j$th mutation at or before time $\tau_j$ has a descendant alive at time $\tau_{j+1}$.  In particular, we have $X_j'(\tau_{j+1}) = X_{j,1}(\tau_{j+1})$.  Therefore, using also that $X_{j,1}(\tau_{j+1}) + X_{j,2}(\tau_{j+1}) = X_j(\tau_{j+1}) = \lceil s/\mu \rceil$, we get, on $\{\tau_{j+1} < \zeta\}$,
\begin{equation}\label{ym1}
Y_j = \frac{X_{j,1}(\tau_{j+1})}{X_{j,1}(\tau_{j+1}) + X_{j,2}(\tau_{j+2})} = \frac{e^{-\int_{\tau_j}^{\tau_{j+1}} G_j(v) \: dv} X_{j,1}(\tau_{j+1})}{e^{-\int_{\tau_j}^{\tau_{j+1}} G_j(v) \: dv} X_{j,1}(\tau_{j+1}) + e^{-\int_{\tau_j}^{\tau_{j+1}} G_j(v) \: dv} X_{j,2}(\tau_{j+2})}.
\end{equation}
By (\ref{prop22}), on $\{\tau_{j+1} < \zeta\}$,
\begin{equation}\label{ym2}
1 - 4 \delta \leq e^{-\int_{\tau_j}^{\tau_{j+1}} G_j(v) \: dv} X_{j,2}(\tau_{j+1}) \leq 1 + 4 \delta.
\end{equation}
Combining (\ref{ym1}), (\ref{ym2}), and the definition of $A_j$, we get that on $A_j^c \cap \{\tau_{j+1} < \zeta\}$,
$$\frac{e^{-sq_j(\tau_j' - \tau_j)} X'_j(\tau_j') - e^{-b}}{e^{-sq_j(\tau_j' - \tau_j)} X'_j(\tau_j') - e^{-b} + 1 + 4 \delta} \leq Y_j \leq \frac{e^{-sq_j(\tau_j' - \tau_j)} X'_j(\tau_j') + e^{-b}}{e^{-sq_j(\tau_j' - \tau_j)} X'_j(\tau_j') + e^{-b} + 1 - 4 \delta}.$$
Combining this observation with (\ref{bpcoupling}) and noting that $Y_j \geq 0$, we conclude that (\ref{Ylem1}) holds on $A_j^c \cap \{\tau_{j+1} < \zeta\} \cap \{\kappa_j > \tau_j' - \tau_j\}$.
\end{proof}

\section{Coupling with the Bolthausen-Sznitman coalescent}

In this section, we prove Theorem \ref{boszthm} by establishing a coupling between the coalescent process $(\Pi_N(u), 0 \leq u \leq t+1)$ and the Bolthausen-Sznitman coalescent.  Our strategy will involve examining the process at the times $\tau_j$.  A very similar idea was used in \cite{dwf13} by Desai, Walczak, and Fisher. 

\subsection{No coalescence between times $\tau_L$ and $a_N T$} 

Recall from Remark \ref{LL9Rmk} that with probability tending to one as $N \rightarrow \infty$, no lineages coalesce as they are traced back from time $a_N T$ to time $\tau_{L+10}$.  The result below shows that the lineages are also unlikely to coalesce as they are traced back further from time $\tau_{L+10}$ to time $\tau_L$, which implies the statement (\ref{part1}) from Theorem \ref{boszthm}.  As with Lemmas \ref{noearlycoal} and \ref{coalj1}, it is sufficient to state the result for the first two lineages.

\begin{Lemma}\label{lemp1}
We have
\begin{equation}\label{tauL1}
\limsup_{N \rightarrow \infty} P \big( \Lambda \cap \{T_{1,2} \geq \tau_L\} \big) \leq C T e^{-b}.
\end{equation}
In particular, the statement (\ref{part1}) holds.  
\end{Lemma}

\begin{proof}
Let $\ell_1 = U_1(a_N T)$ and $\ell_2 = U_2(a_N T)$.  Without loss of generality, suppose $\ell_1 \leq \ell_2$.  We know from the argument in Remark \ref{LL9Rmk} that
$$\lim_{N \rightarrow \infty} P \big( \Lambda \cap \{T_{1,2} \geq \tau_{L+10}\} \big) = 0,$$ so we only need to follow these two lineages between times $\tau_L$ and $\tau_{L+10}$.   By Lemmas \ref{noearlymut} and \ref{taujmut}, we know that, outside of an event $A$ such that $\lim_{N \rightarrow \infty} P(\Lambda \cap A) = 0$, for $i \in \{1, 2\}$ we have $U_i(\tau_{j+1}) = j$ for $j \in \{L - 1, L, \dots, \ell_i\}$ and $U_i(\tau_{j+1}) = \ell_i$ for $j \in \{\ell_i, \dots, L + 9\}$.  When this occurs, there are only three ways that these lineages could coalesce between times $\tau_L$ and $\tau_{L+10}$, in view of the fact that only lineages of the same type can coalesce:
\begin{enumerate}
\item We have $\ell_1 = \ell_2$ and $T_{1,2} \geq \tau_{\ell_1+1}$.

\item We have $\ell_1 < \ell_2$ and $\tau_{\ell_1+1} < T_{1,2} < V_{2,\ell_1+1} < \tau_{\ell_1+2}$.  That is, as we trace back the ancestral lines, the second lineage gets traced back to a type $\ell_1$ individual, then coalesces with the first lineage between times $\tau_{\ell_1+1}$ and $\tau_{\ell_1+2}$.

\item For some $j \in \{L - 1, L, \dots, \ell_1\}$, two type $j$ lineages at time $\tau_{j+1}$ are descended from the same type $j-1$ lineage at time $\tau_j$.
\end{enumerate}

Lemma \ref{noearlycoal} bounds the probability of the first possibility above, while Lemma \ref{coalj2} bounds the probability of the second possibility.  It remains only to consider the third possibility, in which the lineages coalesce between times $\tau_j$ and $\tau_{j+1}$ for $j \in \{L-1, L, \dots, \ell_i\}$.  As noted in the discussion in subsection \ref{coaljj}, Lemmas \ref{coalj1} and \ref{coalj2} establish that the probability that such a coalescence event occurs without the ancestor acquiring an early type $j$ mutation is bounded above by $C T e^{-b}$.  Also, because the result of Lemma \ref{pmlem1} holds even when $j$ is random provided that $\tau_j$ is a stopping time, we have $$P(\Lambda \cap \{X_j^+(\tau_j' - \tau_j) > 0 \mbox{ for some }j \in \{L-1, L, \dots L+9\}) \leq \frac{C e^b}{k_N},$$ where we have used also part 3 of Proposition \ref{newGq}.  In view of (\ref{bpcoupling}) and Lemma \ref{kappabound}, it follows that the probability that, for some $j \in \{L-1, L, \dots, \ell_i\}$, two type $j$ lineages at time $\tau_{j+1}$ are descended from an early type $j$ mutation tends to zero as $N \rightarrow \infty$.  The result (\ref{tauL1}) now follows from the bounds collected in this paragraph.

Finally, since $\tau_L < a_N(T - 1)$ on $\Lambda$ by (\ref{Ldef}) and (\ref{tauspacing}), the statement (\ref{part1}) follows from (\ref{tauL1}), (\ref{bdef}), and the fact that $\eps > 0$ and $\delta > 0$ are arbitrary.
\end{proof}

\subsection{Representing the early type $j$ mutations by a point process}

Fix $j \in I$.  Recall from the discussion before Lemma \ref{coalj1} that the individuals sampled at time $a_N T$ are typically descended from type $j$ individuals at time $\tau_{j+1}$, and these lineages will typically coalesce only if they are traced back to one individual that acquires its $j$th mutation before time $\xi_j$.  We construct in this subsection a point process that encodes these coalescence events.

Let $\Lambda_j$ be the event that $\Lambda$ occurs and that $U_i(\tau_j) = j+1$ for all $i \in \{1, \dots, n\}$.  Suppose we condition on the event $\Lambda_j$, the random variables $Y_{\ell} = X_{\ell}'(\tau_{\ell+1})/\lceil s/\mu \rceil$ and $\tau_{\ell}$ for $\ell \in I$, and the partitions $\Pi_N(T - \tau_{\ell}/a_N)$ for $\ell \in I$ with $\ell \geq j+1$.  
Denote the blocks of $\Pi_N(T - \tau_{\ell}/a_N)$ by $B_{\ell, 1}, \dots, B_{\ell, n_{\ell}}$, where we rank the blocks in order by their smallest element.  By the definition of $\Lambda_j$, the $n_{j+1}$ individuals in the population at time $\tau_{j+1}$ that are ancestors of individuals in the sample are all among the $\lceil s/\mu \rceil$ type $j$ individuals in the population at time $\tau_{j+1}$.  However, by the symmetry in the process, all $\lceil s/\mu \rceil (\lceil s/\mu \rceil - 1) \dots (\lceil s/\mu \rceil - n_{j+1} + 1)$ possible choices of $n_{j+1}$ individuals out of these $\lceil s/\mu \rceil$ are equally likely to be the ancestors of the individuals in the sample corresponding to the integers in the blocks $B_{j+1,1}, \dots, B_{j+1, n_{j+1}}$ respectively.  Also, $X_j'(\tau_{j+1})$ of the $\lceil s/\mu \rceil$ type $j$ individuals at time $\tau_{j+1}$ are descended from an individual that got an early type $j$ mutation between times $\tau_j$ and $\xi_j$.  We call these type $j$ individuals good.

We now construct some uniformly distributed random variables $Z_{i,j}$ for $i \in \{1, \dots, n\}$ and $j \in I$.  Begin by defining random variables $Z_{i,j}^*$ for $i \in \{1, \dots, n\}$ and $j \in I$ which are uniformly distributed on $[0, 1]$ and independent of the population process $({\bf X}(t), t \geq 0)$ and of one another.  If $j \geq L+1$, then let $Z_{i,j} = Z_{i,j}^*$.  Likewise, if either $\Lambda_j$ does not occur or $n_{j+1} < i \leq n$, then let $Z_{i,j} = Z_{i,j}^*$.  Now suppose $\Lambda_j$ occurs.  For $i \in \{1, \dots, n_{j+1}\}$, we call the $(i,j)$ ancestor the individual at time $\tau_{j+1}$ that is the ancestor of the individuals in the sample whose label is in the block $B_{j+1, i}$.  Let $K_0 = 0$, and for $i \in \{1, \dots, n_{j+1}-1\}$, let $K_i$ be the number of integers $h \in \{1, \dots, i\}$ such that the $(h, j)$ ancestor is good.  Then, conditioning on $K_{i-1}$ in addition to the event $\Lambda_j$, the random variables $Y_{\ell}$ and $\tau_{\ell}$ for $\ell \in I$, and the partitions $\Pi_N(T - \tau_{\ell}/a_N)$ for $\ell \in I$ with $\ell \geq j+1$, the probability that the $(i,j)$ ancestor is good is $$P_{i,j} = \frac{X_j'(\tau_{j+1}) - K_{i-1}}{\lceil s/\mu \rceil - (i-1)}.$$  Let $Z_{i,j} = Z_{i,j}^* P_{i,j}$ if the $(i,j)$ ancestor is good, and let $Z_{i,j} = P_{i,j} + Z_{i,j}^* (1 - P_{i,j})$ otherwise.  Note that $Z_{i,j}$ has a uniform distribution on $[0, 1]$, and the $(i,j)$ ancestor is good if and only if $Z_{i,j} \leq P_{i,j}$.  Also, the random variables $Z_{i,j}$ are jointly independent of the random variables $Y_{\ell}$ and the stopping times $\tau_{\ell}$ for $\ell \in I$.

Let $\Phi_N$ be the point process on $[0, t + 1] \times [0, 1]^{n+1}$ consisting of all of the points $$\bigg( T - \frac{\tau_j}{a_N}, Y_j, Z_{1,j}, \dots, Z_{n,j} \bigg)$$ such that $j \in I$, $j \leq L$, and $Y_j > 0$.
We use the point process $\Phi_N$ to construct a coalescent process $(\Pi^*_N(u), 0 \leq u \leq t+1)$ as follows.  Let $\Pi^*_N(0) = \{\{1\}, \dots, \{n\}\}$.  For $u \in (0, t+1]$, suppose $(u, y, z_1, \dots, z_n)$ is a point of $\Phi_N$ and $\Pi_N^*(u-) = \pi$, where $\pi$ is a partition of $\{1, \dots, n\}$ whose blocks, ordered by their smallest elements, are $B_1, \dots, B_{\ell}$.  Then $\Pi_N^*(u)$ is obtained from $\Pi_N^*(u-)$ by merging together all of the blocks $B_i$ for which $z_i \leq y$.  The result below relates the coalescent processes $(\Pi_N(u), 0 \leq t \leq t+1)$ and $(\Pi_N^*(u), 0 \leq u \leq t +1)$.

\begin{Lemma}\label{Pistar}
We have
\begin{equation}\label{Pistareq}
\liminf_{N \rightarrow \infty} P \bigg( \bigcap_{j \in I} \bigg\{ \Pi_N\bigg(T - \frac{\tau_j}{a_N} \bigg) = \Pi_N^* \bigg(T - \frac{\tau_j}{a_N} \bigg) \bigg\} \bigg) \geq 1 - C n^2 \eps.
\end{equation}
\end{Lemma}

\begin{proof}
We claim that the event in (\ref{Pistareq}) could fail to hold in the following ways:
\begin{enumerate}
\item Either $\Pi_N(T - \tau_L/a_N) \neq \{\{1\}, \dots, \{n\}\}$ or $\Pi_N^*(T - \tau_L/a_N) \neq \{\{1\}, \dots, \{n\}\}$.

\item The event $\Lambda_j$ could fail to hold for some $j \in I$ with $j \leq L$.

\item For some $j \in I$, either the event $A_j'$ defined in the statement of Lemma \ref{coalj1} or the event $A_j^*$ defined in the statement of Lemma \ref{coalj2} occurs.

\item For some $j \in I$, two or more individuals at time $\tau_j$ have descendants that got a $j$th mutation before time $\xi_j$ and then have type $j$ descendants in the population at time $\tau_{j+1}$.

\item For some $j \in I$ with $j \leq L$ and $Y_j > 0$, and some $i \in \{1, \dots, n\}$, the random variable $Z_{i,j}$ is between $P_{i,j}$ and $Y_j$.
\end{enumerate}
To see that these are the only possibilities, recall from the discussion at the beginning of subsection \ref{coaljj} that if $\Lambda_{\ell}$ occurs for all ${\ell} \in I$ with $\ell \leq L$, then unless $A_j'$ or $A_j^*$ occurs, the only way that lineages can coalesce between times $\tau_j$ and $\tau_{j+1}$ is for two or more lineages at time $\tau_{j+1}$ to be traced back to one individual that acquires its $j$th mutation before time $\xi_j$.  Unless the fourth event listed above occurs, the only way this can happen is for a group of lineages at time $\tau_{j+1}$ to get traced back to the same individual that acquires its $j$th mutation before time $\xi_j$.  In this case, suppose $\Pi_N(T - \tau_{j+1}/a_N) = \Pi_N^*(T - \tau_{j+1}/a_N) = \pi_{j+1}$, and $B_{j+1, 1}, \dots, B_{j+1, n_{j+1}}$ are the blocks of $\pi_{j+1}$, ranked in order by their smallest elements.  By the construction described at the beginning of this subsection, we obtain $\Pi_N(T - \tau_j/a_N)$ by merging the blocks $B_{j+1,i}$ for which $Z_{i,j} \leq P_{i,j}$.  We obtain $\Pi_N^*(T - \tau_j/a_N)$ by merging the blocks $B_{j+1,i}$ for which $Z_{i,j} \leq Y_{i,j}$.  Therefore, we can only have $\Pi_N(T - \tau_j/a_N) \neq \Pi_N^*(T - \tau_j/a_N)$ if the fifth event listed above occurs.

We thus need to bound the probabilities of the five events listed above.  Recall that $P(\Lambda^c) < 2 \eps$ by (\ref{problam}).
By construction, $(T - \tau_j/a_N, Y_j, Z_{1,j}, \dots, Z_{n,j})$ will only be a point of $\Phi_N$ if $j \leq L$, and $\tau_L < a_N(T - 1)$ on $\Lambda$ by (\ref{Ldef}) and (\ref{tauspacing}).  It follows that $\Pi_N^*(T - \tau_L/a_N) = \{\{1\}, \dots, \{n\}\}$ on $\Lambda$.  Also, by Lemma \ref{lemp1}, the probability that $\Lambda$ occurs and $\Pi_N (T - \tau_L/a_N) \neq \{\{1\}, \dots, \{n\}\}$ is at most $C n^2 T e^{-b} \leq C n^2 \eps$ in view of (\ref{bdef}).  By Lemma \ref{taujmut}, the probability that $\Lambda$ occurs and the second event above occurs tends to zero as $N \rightarrow \infty$.  Lemmas \ref{coalj1} and \ref{coalj2} show that the probability that $\Lambda$ occurs and the third event above occurs is at most $C n^2 T e^{-b} \leq C n^2 \eps$.  The probability that $\Lambda$ occurs and the fourth event above occurs tends to zero as $N \rightarrow \infty$ by (\ref{2immfam}) along with (\ref{bpcoupling}), Lemma \ref{kappabound}, and part 3 of Proposition \ref{newGq}.

It remains to bound the probability of the fifth event above.  For sufficiently large $N$,
$$|P_{i,j} - Y_j| = \bigg| \frac{(i-1)X_j'(\tau_{j+1}) - K_{i-1} \lceil s/\mu \rceil}{\lceil s/\mu \rceil (\lceil s/\mu \rceil - (i-1))} \bigg| \leq \frac{n \lceil s/\mu \rceil}{\lceil s/\mu \rceil (\lceil s/\mu \rceil - (i-1))} \leq \frac{2n \mu}{s}.$$  Because $Z_{i,j}$ has a uniform distribution on $[0, 1]$ and is independent of $Y_j$, the probability that $Z_{i,j}$ is between $P_{i,j}$ and $Y_j$ is at most $2n \mu/s$.  Therefore, using Lemma \ref{Ilem}, the probability that this occurs for some $\ell \in \{1, \dots, n\}$ and $j \in I$ is at most $6n^2 T k_N \mu/s$, which tends to zero as $N \rightarrow \infty$ by (\ref{musN}) and assumption A2.  The lemma follows.
\end{proof}

\subsection{A Poisson point process derived from $\Phi_N$}\label{Poissec}

In this subsection, we modify the point process $\Phi_N$ to obtain a Poisson point process $\Phi$ from which we can construct a Bolthausen-Sznitman coalescent via the technique outlined in subsection \ref{boszsec}.  The random variables $Z_{j,1}, \dots, Z_{j,n}$ are already independent and uniformly distributed on $[0, 1]$, and they will remain unchanged.  However, we will define new random variables $Y_j^*$ that are coupled with the original random variables $Y_j$ as well as new times $T_j^*$.

For $j \in I$, let $Z_j$ be a random variable having the uniform distribution on $[0, 1]$ that is independent of the population process.  Recall the definition of the $\sigma$-field ${\cal H}_j$ from subsection \ref{bpcouplesec}.  Define the random function $$H_j(y,z) = P(Y_j^+ < y|{\cal H}_j) + z P(Y_j^+ = y|{\cal H}_j), \hspace{.1in}\mbox{ for all }y, z \in [0,1].$$  Also, let $F_j(y) = P(Y_j^+ \leq y|{\cal H}_j) = H_j(y,1)$, and for $x \in [0, 1]$, let $F_j^{-1}(x) = \sup\{y: F_j(y) \leq x\}$.  Then it is easy to see that almost surely
\begin{equation}\label{Yjcons1}
Y_j^+ = F_j^{-1}(H_j(Y_j^+, Z_j)).
\end{equation}
Note that if $0 < x < 1$, then there is a random integer $K(x)$ such that $$P \bigg( Y_j^+ \leq \frac{K(x)}{\lceil s/\mu \rceil} \bigg| {\cal H}_j \bigg) \leq x < P \bigg( Y_j^+ \leq \frac{K(x) + 1}{\lceil s/\mu \rceil} \bigg| {\cal H}_j \bigg).$$  Then
\begin{align}
P(H_j(Y_j^+, Z_j) \leq x|{\cal H}_j) &= P\bigg(Y_j^+ \leq \frac{K(x)}{\lceil s/\mu \rceil} \bigg| {\cal H}_j \bigg) + P \bigg(Y_j^+ = \frac{K(x) + 1}{\lceil s/\mu \rceil} \bigg| {\cal H}_j \bigg) \nonumber \\
&\hspace{.8in}\times P \bigg(Z_j \leq \frac{x - P(Y_j^+ \leq K(x)/\lceil s/\mu \rceil |{\cal H}_j)}{P(Y_j^+ = (K(x) + 1)/\lceil s/\mu \rceil | {\cal H}_j)}\bigg) \nonumber \\
&= x. \nonumber
\end{align}
Therefore, the conditional distribution of $H_j(Y_j^+, Z_j)$ given ${\cal H}_j$ is uniform on $[0, 1]$.  For $x \geq 0$, let
\begin{displaymath}
K_j(x) = \left\{
\begin{array}{ll} e^{-(\tau_{j+1}^* - \tau_j^*)(1-x)/a_N x} & \mbox{ if }\eps \leq x \leq 1 \\
e^{-(\tau_{j+1}^* - \tau_j^*)(1 - \eps)/a_N \eps} & \mbox{ if }0 \leq x < \eps \\
0 & \mbox{ if }x < 0
\end{array} \right.
\end{displaymath}
For $x \in [0, 1]$, let $K_j^{-1}(x) = \sup\{y: K_j(y) \leq x\}$.  Also, let
\begin{equation}\label{Yjcons2}
Y_j^* = K_j^{-1}(H_j(Y_j^+, Z_j)).
\end{equation}
Then for all $x \geq 0$, we have
\begin{equation}\label{Yjdist}
P(Y_j^* \leq x|{\cal H}_j) = K_j(x).
\end{equation}
Note that $Y_j^*$ never takes a value between $0$ and $\eps$, so if $Y_j^* > 0$, then $Y_j^* \geq \eps$.

We now continue with the construction of $\Phi$.  For all $j \in I$, independently of the population process $({\bf X}(t), t \geq 0)$ and of all other auxiliary random variables introduced up to this point, let $T_j^*$ be uniformly distributed on $[T - \tau_{j+1}^*/a_N, T - \tau_j^*/a_N]$, and let $\Phi_j'$ be a Poisson point process on $[T - \tau_{j+1}^*/a_N, T - \tau_j^*/a_N] \times [0, 1]^{n+1}$ with intensity $$du \times x^{-2} \: dx \times dz_1 \times \dots \times dz_n.$$  For all $j$ such that $T_j^* \in [1, t+1]$ and $Y_j^* > 0$, the point process $\Phi$ will include the point $(T_j^*, Y_j^*, Z_{j,1}, \dots, Z_{j,n})$.  Also, for all $j$ such that $Y_j^* > 0$, the point process $\Phi$ will include all points of $\Phi_j'$ whose first coordinate is in $[1, t+1]$ and whose second coordinate is in the interval $(\eps, Y_j^*)$.  Finally, $\Phi$ will include all points of $\Phi_j'$ whose first coordinate is in $[1, t+1]$ and whose second coordinate is less than $\eps$. 

\begin{Lemma}
The point process $\Phi$ defined above is a Poisson point process on $[1, t+1] \times [0, 1]^{n+1}$ with intensity
\begin{equation}\label{inten}
du \times x^{-2} \: dx \times dz_1 \times \dots \times dz_n.
\end{equation}
\end{Lemma}

\begin{proof}
We separately consider, for each $j$, the restriction of $\Phi$ to points whose first coordinate is in the interval $[T - \tau_{j+1}^*/a_N, T - \tau_j^*/a_N]$.  Note that for a Poisson point process with intensity (\ref{inten}), the expected number of points in the region $[T - \tau_{j+1}^*/a_N, T - \tau_j^*/a_N] \times [x, 1] \times [0,1]^n$ is $$\bigg( \frac{\tau_{j+1}^* - \tau_j^*}{a_N} \bigg) \int_x^1 y^{-2} \: dy = \frac{(\tau_{j+1}^* - \tau_j^*)(1-x)}{a_N x}.$$  Therefore, from (\ref{Yjdist}), we see that if $x \geq \eps$, then $P(Y_j^* \geq x|{\cal H}_j)$ is the probability that there are no points in this region.  Using also that $T_j^*$ is uniformly distributed on $[T - \tau_{j+1}^*/a_N, T - \tau_j^*/a_N]$ and that the random variables $Z_{j,1}, \dots, Z_{j,n}$ are uniformly distributed on $[0,1]^n$, it follows that $$(T_j^*, Y_j^*, Z_{j,1}, \dots, Z_{j,n})$$ has the same distribution as the point whose second coordinate is the largest among points of a Poisson process with intensity (\ref{inten}) restricted to $[T - \tau_{j+1}^*/a_N, T - \tau_j^*/a_N] \times [\eps, 1] \times [0,1]^n$.  Furthermore, conditional on the event that such a Poisson process 
has a point whose second coordinate is $y$ and no point whose second coordinate is larger than $y$, the distribution of the restriction of the Poisson process to $[T - \tau_{j+1}^*/a_N, T - \tau_j^*/a_N] \times [\eps, y) \times [0,1]^n$ is that of a Poisson process with intensity (\ref{inten}).  It thus follows from the construction of $\Phi$ that the restriction of $\Phi$ to $[T - \tau_{j+1}^*/a_N, T - \tau_j^*/a_N]$ has intensity given by (\ref{inten}).

Finally, because of the conditioning on ${\cal H}_j$ in (\ref{Yjdist}), the random variables $Y_j^*$ for $j \in I$ are independent.  Because the Poisson processes $\Phi_j'$ are independent, it follows that the restrictions of $\Phi$ to the intervals $[T - \tau_{j+1}^*/a_N, T - \tau_j^*/a_N]$ are independent.  The lemma now follows from the superposition theorem for Poisson processes.
\end{proof}

We now use the Poisson point process $\Phi$ to construct a coalescent process $(\Pi(u), 0 \leq u \leq t+1)$.  Let $\Pi(u) = \{\{1\}, \dots, \{n\}\}$ for $u \in [0, 1]$.  For $u \in (1, t+1]$, suppose $(u, y, z_1, \dots, z_n)$ is a point of $\Phi$ and $\Pi(u-) = \pi$, where $\pi$ is a partition of $\{1, \dots, n\}$ into the blocks $B_1, \dots, B_{\ell}$, ordered by their smallest element.  Then $\Pi(u)$ is obtained from $\Pi(u-)$ by merging together all of the blocks $B_i$ for which $z_i \leq y$.  As discussed in subsection \ref{boszsec}, this construction is well-defined, and the process $(\Pi(1 + u), 0 \leq u \leq t)$ obeys the law of the Bolthausen-Sznitman coalescent.  

\subsection{Comparing $Y_j$ and $Y_j^*$}

The goal in this subsection is to prove two lemmas that establish that, with high probability, the random variables $Y_j$ and $Y_j^*$ are close.  Lemma \ref{YYstar} bounds the probability that either $Y_j$ or $Y_j^*$ is greater than $\eps$, but the other is not.  Lemma \ref{YdiffY} bounds the probability that the difference between $Y_j$ and $Y_j^*$ is more than $\eps^2$.  We will need a couple of preliminary estimates.

\begin{Lemma}\label{Ajpp}
For $j \in I$, let $$A_j'' = \bigg\{ \bigg| \frac{q_j}{k_N} - q \bigg( \frac{\tau_j}{a_N} \bigg) \bigg| > \delta \bigg\}.$$  Then $$\lim_{N \rightarrow \infty} P \bigg( \Lambda \cap \bigcup_{j \in I} A_j'' \bigg) = 0.$$
\end{Lemma}

\begin{proof}
Lemma \ref{Ilem} and part 1 of Proposition \ref{mainjprop} imply that on $\Lambda$, the fittest individual in the population at time $\tau_j$ must have either $j$ or $j-1$ mutations.  It therefore follows from (\ref{qjdef}) and (\ref{Qdef}), along with the fact that $\tau_j > a_N + 2a_N/k_N$ for all $j \in I$ by Lemma \ref{Ilem}, that $Q(\tau_j)$ must either equal $q_j$ or $q_j - 1$ on $\Lambda$ for all $j \in I$.  Let $S = [1 + (T - (t+2))/2, T]$, which is a compact subset of $(1, \infty)$.  It follows from Proposition \ref{Qthm} that $$\sup_{t \in S} \bigg| \frac{Q(a_N t)}{k_N} - q(t) \bigg| \rightarrow_p 0,$$ where $\rightarrow_p$ denotes convergence in probability as $N \rightarrow \infty$.  By (\ref{jp2}) and Lemma \ref{Ilem}, on $\Lambda$ we have $\tau_j/a_N \in S$ for all $j \in I$.  Therefore,
$$\sup_{j \in I} \bigg| \frac{q_j}{k_N} - q \bigg( \frac{\tau_j}{a_N} \bigg) \bigg| \1_{\Lambda} \rightarrow_p 0,$$
which implies the lemma.
\end{proof}

\begin{Lemma}\label{Kjlem}
There is a positive constant $C$ such that if $\eps \leq y \leq 1$ and $j \in I$, then on the event
$\{\tau_j < \zeta\} \cap (A_j'')^c \in {\cal H}_j$,
we have for sufficiently large $N$,
$$\frac{(1 - y)(1 - C \delta T)}{q_j y} \leq P(Y_j^* \geq y|{\cal H}_j) \leq \frac{(1 - y)(1 + C \delta T)}{q_j y}.$$
\end{Lemma}

\begin{proof}
By (\ref{tjstardef}),
\begin{equation}\label{K0}
\bigg| \frac{\tau_{j+1}^* - \tau_j^*}{a_N} - \frac{1}{q_j} \bigg| = \bigg| \frac{1}{k_N q(\tau_j^*/a_N)} - \frac{1}{q_j} \bigg| = \frac{1}{k_N} \bigg| \frac{1}{q(\tau_j^*/a_N)} - \frac{k_N}{q_j} \bigg|.
\end{equation}
Also, by (\ref{qlip}) and (\ref{tslem2}), we have on $\{\tau_j < \zeta\} \cap (A_j'')^c$,
\begin{equation}\label{K1}
\bigg| \frac{q_j}{k_N} - q \bigg( \frac{\tau_j^*}{a_N} \bigg) \bigg| \leq \delta + \bigg| q \bigg( \frac{\tau_j}{a_N} \bigg) - q \bigg( \frac{\tau_j^*}{a_N} \bigg) \bigg| \leq \delta + 10e \delta T.
\end{equation}
Therefore, using (\ref{K0}) and (\ref{K1}) along with the facts that $q(\tau_j^*/a_N) \geq 1$ by Proposition \ref{Qlem} and that $q_j/k_N \geq 1 - 2 \delta$ on $\{\tau_j < \zeta\}$ by part 3 of Proposition \ref{newGq}, we get that on $\{\tau_j < \zeta\} \cap (A_j'')^c$, $$\bigg| \frac{\tau_{j+1}^* - \tau_j^*}{a_N} - \frac{1}{q_j} \bigg| \leq \frac{C \delta T}{k_N}.$$  Because $|(1 - e^{-x}) - x| \leq x^2/2$ for $x \geq 0$ and (\ref{tjstardef}) holds, it follows that when $\eps \leq y \leq 1$, we have for sufficiently large $N$, on $\{\tau_j < \zeta\} \cap (A_j'')^c$,
$$\bigg|(1 - K_j(y)) - \frac{(1 - y)}{q_j y} \bigg| \leq \frac{1}{2} \bigg( \frac{(\tau_{j+1}^* - \tau_j^*) (1 - y)}{a_N y} \bigg)^2 + \frac{1 - y}{y} \bigg| \frac{\tau_{j+1}^* - \tau_j^*}{a_N} - \frac{1}{q_j} \bigg| \leq \frac{1 - y}{y} \cdot \frac{C \delta T}{k_N}.$$  Because $q_j \leq (e + 2 \delta) k_N$ on $\{\tau_j < \zeta\}$ by part 3 of Proposition \ref{newGq}, the result follows.
\end{proof}

\begin{Lemma}\label{YYstar}
Letting $\triangle$ denote the symmetric difference between two events, for sufficiently large $N$ we have
$$P\bigg( \Lambda \cap \bigcup_{j \in I} \big(\{Y_j \geq \eps\} \triangle \{Y_j^* \geq \eps\} \big) \bigg) \leq \frac{C \delta T^2}{\eps}.$$
\end{Lemma}

\begin{proof}
By Lemmas \ref{Ylem} and \ref{Kjlem} and part 3 of Proposition \ref{newGq}, $$\big| P(Y_j^+ \geq \eps|{\cal H}_j) - P \big(Y_j^* \geq \eps|{\cal H}_j)\big| \leq \frac{C \delta T}{\eps k_N}$$ for sufficiently large $N$ on $\{\tau_j < \zeta\} \cap (A_j'')^c$, and the same result holds with $Y_j^-$ in place of $Y_j^+$.  Because $Y_j^- \leq Y_j^+$, and the random variables $Y_j^+$ and $Y_j^*$ are monotone functions of the same uniformly distributed random variable by (\ref{Yjcons1}) and (\ref{Yjcons2}), it follows that $$P(\{Y_j^+ \geq \eps\} \triangle \{Y_j^* \geq \eps\}|{\cal H}_j) \leq \frac{C \delta T}{\eps k_N}$$ on $\{\tau_j < \zeta\} \cap (A_j'')^c$, and the same result holds with $Y_j^-$ in place of $Y_j^+$.  Let $\Psi_j = A_j^c \cap \{\tau_{j+1} < \zeta\} \cap \{\kappa_j > \tau_j' - \tau_j\}$.  By (\ref{Ylem1}), we have $$(\{Y_j^- \geq \eps\} \cap \Psi_j) \subset (\{Y_j \geq \eps\} \cap \Psi_j) \subset (\{Y_j^+ \geq \eps\} \cap \Psi_j).$$  It follows that on $\{\tau_j < \zeta\} \cap (A_j'')^c$, we have $$P \big( (\{Y_j \geq \eps\} \triangle \{Y_j^* \geq \eps\}) \cap \Psi_j \big| {\cal H}_j \big) \leq \frac{C \delta T}{\eps k_N}.$$  The result follows by taking expectations, summing over $j \in I$, and using Lemmas \ref{kappabound}, \ref{bptail3}, and \ref{Ajpp}, along with the fact that the cardinality of $I$ is at most $3T k_N$ by Lemma \ref{Ilem}.
\end{proof}

\begin{Lemma}\label{YdiffY}
There is a positive constant $C^*$, not depending on $\eps$, $\delta$, or $T$, such that
for sufficiently large $N$, we have
$$P\bigg( \Lambda \cap \bigcup_{j \in I} \big( \{|Y_j - Y_j^*| > C^* \eps^2\} \cap \{Y_j \geq \eps\} \cap \{Y_j^* \geq \eps\} \big) \bigg) \leq \frac{C \delta T \log(1/\eps)}{\eps^2}.$$
\end{Lemma}

\begin{proof}
We first compare $Y_j^*$ to $Y_j^+$.  In view of (\ref{Yjcons1}) and (\ref{Yjcons2}), we need to compare the functions $F_j^{-1}$ and $K_j^{-1}$.  Suppose $z \in (0, 1)$.  If $F_j^{-1}(1 - z) \in [\delta, 1 - \delta]$, then (\ref{Ylem2}) implies that on $\{\tau_j < \zeta\}$, we have $$\frac{1 - 13 \delta}{q_j z + 1 - 13 \delta} \leq F_j^{-1}(1 - z) \leq \frac{1 + 13 \delta}{q_j z + 1 + 13 \delta}.$$  Likewise, Lemma \ref{Kjlem} implies that if $K_j^{-1}(1 - z) \geq \eps$, then on $\{\tau_j < \zeta\} \cap (A_j'')^c$, we have
$$\frac{1 - C \delta}{q_j z + 1 - C \delta} \leq K_j^{-1}(1 - z) \leq \frac{1 + C \delta}{q_j z + 1 + C \delta}.$$  It follows that on the event $\{\tau_j < \zeta\} \cap (A_j'')^c$, if $F_j^{-1}(1 - z) \in [\delta, 1 - \delta]$ and $K_j^{-1}(1-z) \in [\eps, 1]$, then
\begin{equation}\label{FKinv}
|F_j^{-1}(1 - z) - K_j^{-1}(1 - z)| \leq C \delta.
\end{equation}
Because $F_j^{-1}$ and $K_j^{-1}$ are increasing functions taking their values in $[0, 1]$, and $\delta < \eps$ by (\ref{deldef}), we see that (\ref{FKinv}) holds on $\{\tau_j < \zeta\} \cap (A_j'')^c$ as long as
$F_j^{-1}(1 - z) \in [\eps, 1]$ and $K_j^{-1}(1-z) \in [\eps, 1]$.
Since $\delta < \eps^2$ by (\ref{deldef}), it follows that there is a positive constant $C^*$ such that on $\{\tau_j < \zeta\} \cap (A_j'')^c$, we have
\begin{equation}\label{YY1}
|Y_j^+ - Y_j^*| \1_{\{Y_j^* \geq \eps\}} \1_{\{Y_j^+ \geq \eps\}} \leq (C^* - 1) \eps^2.
\end{equation}

It remains to control the difference between $Y_j^+$ and $Y_j$.  By (\ref{Ylem2}), on $\{\tau_j < \zeta\}$,
\begin{align*}
E[Y_j^+ \1_{\{Y_j^+ \geq \eps\}} - Y_j^- \1_{\{Y_j^- \geq \eps\}}|{\cal H}_j] &= \int_0^1 \big( P(Y_j^+ \1_{\{Y_j^+ \geq \eps\}} \geq y|{\cal H}_j) - P(Y_j^- \1_{\{Y_j^- \geq \eps\}} \geq y|{\cal H}_j) \big) \: dy \\
&= \int_0^{\eps} \big(P(Y_j^+ \geq \eps|{\cal H}_j) - P(Y_j^- \geq \eps|{\cal H}_j) \big) \: dy \\
&\hspace{.5in}+ \int_{\eps}^1 \big(P(Y_j^+ \geq y|{\cal H}_j) - P(Y_j^- \geq y|{\cal H}_j) \big) \: dy \\
&\leq \eps \cdot \frac{(1 - \eps) C \delta}{q_j \eps} + \int_{\eps}^{1 - \delta} \frac{C \delta (1 - y)}{q_j y} \: dy + \delta \cdot \frac{\delta (1 + C \delta)}{q_j(1 - \delta)} \\
&\leq \frac{C \delta \log(1/\eps)}{q_j}.
\end{align*}
Let $\Psi_j = A_j^c \cap \{\tau_{j+1} < \zeta\} \cap \{\kappa_j > \tau_j' - \tau_j\}$.
Because $Y_j^- \leq Y_j \leq Y_j^+$ on $\Psi_j$, by (\ref{Ylem1}),
$$E\big[(Y_j^+ \1_{\{Y_j^+ \geq \eps\}} - Y_j \1_{\{Y_j \geq \eps\}}) \1_{\Psi_j}|{\cal H}_j\big] \leq \frac{C \delta \log(1/\eps)}{q_j}.$$
Now Markov's Inequality implies that
$$P\big(\{ |Y_j \1_{\{Y_j \geq \eps\}} - Y_j^+ \1_{\{Y_j^+ \geq \eps\}}| > \eps^2\} \cap \Psi_j \big|{\cal H}_j\big) \leq \frac{C \delta \log(1/\eps)}{q_j \eps^2}.$$
Combining this result with (\ref{YY1}) and part 3 of Lemma \ref{newGq} gives, for sufficiently large $N$, $$P \big( \{|Y_j - Y_j^*| > C^* \eps^2 \big\} \cap \{Y_j \geq \eps\} \cap \{Y_j^* \geq \eps\} \cap \Psi_j \cap (A_j'')^c \cap \Lambda \big) \leq \frac{C \delta \log(1/\eps)}{k_N \eps^2}.$$
The result follows by summing over $j$ and using Lemmas \ref{kappabound}, \ref{bptail3}, and \ref{Ajpp}.
\end{proof}

\subsection{Small coalescence events}

Lemma \ref{smallco} below shows that it is unlikely that lineages will coalesce between times $\tau_j$ and $\tau_{j+1}$ if $Y_j \leq \eps$.

\begin{Lemma}\label{smallco}
For sufficiently large $N$, we have $$P \bigg( \Lambda \cap \bigcup_{j \in I} \bigg( \bigg\{\Pi_N^* \bigg(T - \frac{\tau_j}{a_N} \bigg) \neq \Pi_N^* \bigg(T - \frac{\tau_{j+1}}{a_N}\bigg) \bigg\} \cap \{Y_j \leq \eps\} \bigg) \bigg) \leq C T n^2 \eps.$$
\end{Lemma}

\begin{proof}
Suppose $j \in I$.  Let $\Psi_j = A_j^c \cap \{\tau_{j+1} < \zeta \} \cap \{\kappa_j > \tau_j' - \tau_j\}$, where $A_j$ is the event defined in Lemma \ref{bptail3} and $\kappa_j$ is defined in (\ref{kappajdef}).  Define the $\sigma$-field ${\cal H}_j$ as in subsection \ref{bpcouplesec}.  Let ${\cal G}_j$ be the $\sigma$-field generated by the $\sigma$-field ${\cal H}_j$, the random variable $Y_j$ defined in (\ref{Yjdef}), and the event $\Psi_j$.  Conditional on ${\cal G}_j$, the probability that at least two of the random variables $Z_{1,j}, \dots, Z_{n,j}$ are less than or equal to $Y_j$ is at most $\binom{n}{2} Y_j^2$.  Therefore, on $\{\tau_j < \zeta\}$, we have
$$P \bigg(\bigg\{\Pi_N^*\bigg(T - \frac{\tau_j}{a_N} \bigg) \neq \Pi_N^* \bigg(T - \frac{\tau_{j+1}}{a_N}\bigg) \bigg\} \cap \{Y_j \leq \eps\} \cap \Psi_j \bigg| {\cal G}_j \bigg) \leq \binom{n}{2} Y_j^2 \1_{\{Y_j \leq \eps\}} \1_{\Psi_j}.$$
Now take conditional expectations of both sides with respect to ${\cal H}_j$ to get that on $\{\tau_j < \zeta\}$,
\begin{align}\label{sc0}
&P \bigg(\bigg\{\Pi_N^* \bigg(T - \frac{\tau_j}{a_N} \bigg) \neq \Pi_N^* \bigg(T - \frac{\tau_{j+1}}{a_N}\bigg) \bigg\} \cap \{Y_j \leq \eps\} \cap \Psi_j \bigg| {\cal H}_j \bigg) \leq \binom{n}{2} E \big[ Y_j^2 \1_{\{Y_j \leq \eps\}} \1_{\Psi_j} \big| {\cal H}_j \big].
\end{align}
Recall that for any nonnegative random variable $X$, we have $E[X^2] = \int_0^{\infty} 2x P(X \geq x) \: dx$.    Therefore, on $\{\tau_j < \zeta\}$,
\begin{align}\label{sc1}
E \big[ Y_j^2 \1_{\{Y_j \leq \eps\}} \1_{\Psi_j} \big| {\cal H}_j \big] &= \int_0^{\infty} 2x P \big( Y_j \1_{\{Y_j \leq \eps\}} \1_{\Psi_j}  > x \big| {\cal H}_j \big) \: dx \nonumber \\
&\leq \int_0^{\eps} 2x P\big(Y_j \1_{\Psi_j} > x \big| {\cal H}_j \big) \: dx.
\end{align}
Recall from (\ref{Ylem1}) that $Y_j \leq Y_j^+$ on $\Psi_j$.  Also, from (\ref{bpcoupling}), we see that on
$\Psi_j$, if $Y_j > 0$ then $X_j^+(\tau_j' - \tau_j) > 0$, and on $\{\tau_j < \zeta\}$, we have $q_j \geq (1 - 2 \delta) k_N$ by part 3 of Proposition \ref{newGq}. Therefore, by Lemma \ref{pmlem1}, 
\begin{equation}\label{sc2}
P(Y_j \1_{\Psi_j} > 0 | {\cal H}_j) \leq \frac{C e^b}{k_N}. 
\end{equation}
Also, on $\Psi_j$, if $Y_j > x$ and $3e^{-b} \leq x \leq \eps$, it follows from (\ref{Yjplus}) that if $\eps$ is sufficiently small, then
$$e^{-sq_j(\tau_j' - \tau_j)} X_j^+(\tau_j' - \tau_j) \geq \frac{(e^{-b} + 1 - 4 \delta)x - e^{-b}}{1 - x} \geq \frac{x}{2}.$$
Therefore, Lemma \ref{pmlem2} implies that if $\eps$ is sufficiently small and $N$ is sufficiently large, and if $3 e^{-b} \leq x \leq \eps$, then
\begin{equation}\label{sc3}
P(Y_j \1_{\Psi_j} > x | {\cal H}_j) \leq \frac{C}{k_N x}.
\end{equation}
Dividing the integral on the right-hand side of (\ref{sc1}) into two pieces and using (\ref{sc2}) to estimate the first piece and (\ref{sc3}) to estimate the second piece, we get
\begin{align} \label{sc4}
E \big[ Y_j^2 \1_{\{Y_j \leq \eps\}} \1_{\Psi_j} \big| {\cal H}_j \big] &\leq \int_0^{3 e^{-b}} 2x \cdot \frac{C e^b}{k_N} \: dx + \int_{3e^{-b}}^{\eps} 2x \cdot \frac{C}{k_N x} \: dx \nonumber \\
&\leq \frac{C e^{-b}}{k_N} + \frac{C \eps}{k_N} \nonumber \\
&\leq \frac{C \eps}{k_N}.
\end{align}
Using (\ref{sc4}) to bound the right-hand side of (\ref{sc0}) and then taking expectations, we get
\begin{equation}\label{sc5}
P \bigg(\bigg\{\Pi_N^* \bigg(T - \frac{\tau_j}{a_N} \bigg) \neq \Pi_N^* \bigg(T - \frac{\tau_{j+1}}{a_N}\bigg) \bigg\} \cap \{Y_j \leq \eps\} \cap \Psi_j \bigg) \leq \frac{C n^2 \eps}{k_N}.
\end{equation}
The result now follows by summing over $j$ and using Lemmas \ref{kappabound} and \ref{bptail3}.
\end{proof}

\subsection{Completion of the coupling argument}

Fix a positive integer $d$ and times $0 = t_0 < t_1 < \dots < t_d = t$.  Recall that equation (\ref{part1}) was established as part of Lemma \ref{lemp1}.  Therefore, to prove Theorem \ref{boszthm}, we need to show that the joint distribution of $(\Pi_N(1 + t_0), \dots, \Pi_N(1 + t_d))$ converges as $N \rightarrow \infty$ to the joint distribution of $(\Pi(1 + t_0), \dots, \Pi(1 + t_d))$, where $(\Pi(u), 0 \leq u \leq t+1)$ is the coalescent process derived from the Poisson point process $\Phi$ at the end of subsection \ref{Poissec}.

\begin{proof}[Proof of Theorem \ref{boszthm}]
The key to the proof will be to show that with high probability, we have
\begin{equation}\label{Picompare}
\Pi\bigg(T - \frac{\tau_j^*}{a_N} \bigg) = \Pi_N^*\bigg(T - \frac{\tau_j}{a_N} \bigg)\hspace{.1in}\mbox{ for all } j \in I \mbox{ with }j \leq L.
\end{equation}
Recall that the coalescent process $\Pi_N^*$ was constructed from the point process $\Phi_N$ in the same way that $\Pi$ was constructed from $\Phi$.  Therefore, we simply need to compare the two constructions.  If (\ref{Picompare}) fails to hold, then one of the following must occur:
\begin{enumerate}
\item Either $\Pi_N^*(T - \tau_L/a_N) \neq \{\{1\}, \dots, \{n\}\}$ or $\Pi(T - \tau_L^*/a_N) \neq \{\{1\}, \dots, \{n\}\}$.

\item For some $j \in I$, we have either $Y_j \geq \eps$ and $Y_j^* < \eps$, or $Y_j < \eps$ and $Y_j^* \geq \eps$.

\item For some $j \in I$, we have $\Pi_N^*(T - \tau_j/a_N) \neq \Pi_N^*(T - \tau_{j+1}/a_N)$ and $Y_j < \eps$.

\item For some $u \in [1, t+1]$, we have $\Pi(u) \neq \Pi(u-)$ but $u$ does not equal $T_j^*$ for any $j$.

\item For some $j \in I$ with $j \leq L$, we have $Y_j \geq \eps$, $Y_j^* \geq \eps$, and $\Pi_N^*((T - \tau_j/a_N)-) = \Pi(T_j^*-)$, but $\Pi_N^*(T - \tau_j/a_N) \neq \Pi(T_j^*)$.
\end{enumerate}

We now bound the probabilities of these five events.  As for the first event, note that (\ref{tauL1}) and (\ref{problam}) imply that $P(\Pi_N(T - \tau_L/a_N) \neq \{\{1\}, \dots, \{n\}\}) \leq C \eps + C T n^2 e^{-b}$.  Combining this result with Lemma \ref{Pistar} and (\ref{bdef}) gives $$P\bigg(\Pi_N^*\bigg(T - \frac{\tau_L}{a_N} \bigg) \neq \{\{1\}, \dots, \{n\}\}\bigg) \leq C n^2 \eps.$$  By (\ref{Ldef}), we have $T - \tau_L/a_N \leq 1 + 3/k_N$, so (\ref{tslem2}) implies $T - \tau_L^*/a_N \leq 1 + 3/k_N + 10 \delta T$.  Because each pair of lineages in the Bolthausen-Sznitman coalescent merges at rate $1$, it follows that for sufficiently large $N$,
$$P\bigg(\Pi \bigg(T - \frac{\tau_L^*}{a_N} \bigg) \neq \{\{1\}, \dots, \{n\}\}\bigg) \leq \binom{n}{2} \bigg( \frac{3}{k_N} + 10 \delta T \bigg) \leq C n^2 \delta T.$$

It follows from Lemma \ref{YYstar}, along with (\ref{problam}) and the fact that $\delta < \eps^2$ by (\ref{deldef}), that the probability that the second of the five events above occurs is at most $C \eps T^2$.  Likewise, it follows from Lemma \ref{smallco} and (\ref{problam}) that the probability that the third of the five events occurs is bounded above by $C T n^2 \eps$.

Consider next the fourth event listed above.  From the construction, this can only happen either if, for some $j \in I$, there are two points of $\Phi$ in $[T - \tau_{j+1}^*/a_N, T - \tau_j^*/a_N] \times [\eps, 1] \times [0, 1]^n$, or if there is some point $(u, y, z_1, \dots, z_n)$ in $\Phi$ in which $y \leq \eps$ but two of the points $z_1, \dots, z_n$ are less than or equal to $y$.  Recall that if $X$ has the Poisson distribution with mean $\lambda$, then $P(X \geq 2) \leq \lambda^2$.  Therefore, using also (\ref{tauspacing2}), the probability that, for some $j \in I$, there are two points of $\Phi$ in $[T - \tau_{j+1}^*/a_N, T - \tau_j^*/a_N] \times [\eps, 1] \times [0, 1]^n$ is bounded above by $$\sum_{j \in I} \bigg( \frac{\tau_{j+1}^* - \tau_j^*}{a_N} \cdot \frac{1 - \eps}{\eps} \bigg)^2 \leq \sum_{j \in I} \frac{1}{(\eps k_N)^2} \leq \frac{C T}{\eps^2 k_N},$$ which tends to zero as $N \rightarrow \infty$.  Note that if $y$ is the second coordinate of a point in $\Phi$, the probability that two of the points $z_1, \dots, z_n$ are less than or equal to $y$ is at most $\binom{n}{2} y^2$.  Therefore, the probability that there is a point $(u, y, z_1, \dots, z_n)$ in $\Phi$ in which $y \leq \eps$ but two of the points $z_1, \dots, z_n$ are less than or equal to $y$ is bounded above by $$t \int_0^{\eps} y^{-2} \cdot \binom{n}{2} y^2 \: dy = \binom{n}{2} t \eps \leq C T n^2 \eps.$$

Finally, consider the fifth of the possibilities above, which means that the coalescence at time $T - \tau_j/a_N$ in the process $\Pi_N^*$ does not match the coalescence that occurs at time $T_j^*$ in the process $\Pi$.  One way this could happen would be if the time interval $[T - \tau_{j+1}^*/a_N, T - \tau_j^*/a_N]$ is not entirely contained in the interval $[1, t+1]$.  By (\ref{tauspacing2}) and (\ref{jstarprime}), the number of $j \in I$ for which this interval is not contained in $[1, t+1]$ is at most $C \delta T k_N$.  By Lemmas \ref{kappabound}, \ref{bptail3}, and \ref{Ylem}, along with (\ref{problam}) and part 3 of Proposition \ref{newGq}, the probability that $Y_j > \eps$ for some such $j$ is at most
$$C \delta T k_N \cdot \frac{(1 - \eps)(1 + 13 \delta)}{(1 - 2 \delta) k_N \eps} + 2 \eps \leq \frac{C \delta T}{\eps} + C \eps.$$
The other way that the coalescence at time $T - \tau_j/a_N$ in the process $\Pi_N^*$ might not match the coalescence that occurs at time $T_j^*$ in the process $\Pi$ would be if one of the random variables $Z_{j,1}, \dots, Z_{j,n}$ is between $Y_j$ and $Y_j^*$.  By Lemma \ref{YdiffY}, the probability that this happens when $|Y_j - Y_j^*| > \eps^2$ is bounded above by $$\frac{C \delta T \log (1/\eps)}{\eps^2}.$$  Using Lemmas \ref{kappabound}, \ref{bptail3}, and \ref{Ylem}, we see that the probability that this happens when $|Y_j - Y_j^*| \leq \eps^2$ is at most $$\sum_{j \in I} \frac{C}{k_N \eps} \cdot n \eps^2 \leq C T n \eps.$$

Combining the bounds obtained above, we see that for sufficiently large $N$, the probability that (\ref{Picompare}) fails to hold is bounded above by
\begin{equation}\label{endub}
C T n^2 \eps + C n^2 \delta T + C \eps T^2 + \frac{C \delta T \log(1/\eps)}{\eps^2}.
\end{equation}
By Lemma \ref{Pistar}, we can replace $\Pi_N^*$ by $\Pi_N$ in (\ref{Picompare}) and conclude that the probability that
\begin{equation}\label{Picompare2}
\Pi\bigg(T - \frac{\tau_j^*}{a_N} \bigg) = \Pi_N\bigg(T - \frac{\tau_j}{a_N} \bigg)\hspace{.1in}\mbox{ for all } j \in I \mbox{ with }j \leq L
\end{equation}
fails to hold is also bounded above by the expression in (\ref{endub}) for sufficiently large $N$.

Now suppose that indeed (\ref{Picompare2}) holds and $\Lambda$ occurs.  Fix $i \in \{1, \dots, d\}$.  Then there exists $j \in I$ such that $T - \tau_{j+1}/a_N \leq t_i < T - \tau_j/a_N$.  By (\ref{tauspacing}) and (\ref{tslem2}), for sufficiently large $N$, $$T - \frac{\tau_j^*}{a_N} \leq t_i + \frac{2}{k_N} + 10 \delta T \leq t_i + 11 \delta T$$ and
$$T - \frac{\tau_{j+1}^*}{a_N} \geq t_i - \frac{2}{k_N} - 10 \delta T \geq t_i - 11 \delta T.$$  Thus, as long as $\Pi(t_i - 11 \delta T) = \Pi(t_i + 11 \delta T)$ and (\ref{Picompare2}) holds, we must have $\Pi(t_i) = \Pi_N(t_i)$.  However, because each pair of lineages in the Bolthausen-Sznitman coalescent merges at rate one, we have $$P\big(\Pi(t_i - 11 \delta T) \neq \Pi(t_i + 11 \delta T)\big) \leq \binom{n}{2} \cdot 22 \delta T.$$  Taking the union over $i \in \{1, \dots, d\}$ and using (\ref{endub}), it follows that for sufficiently large $N$, $$P(\Pi_N(t_i) \neq \Pi(t_i) \mbox{ for some }i \in \{1, \dots, d\}) \leq C T n^2 \eps + C d n^2 \delta T + C \eps T^2 + \frac{C \delta T \log(1/\eps)}{\eps^2}.$$  Since $\delta < \eps^3$ by (\ref{deldef}) and $\eps > 0$ can be chosen arbitrarily small for any fixed $T$, the theorem follows.
\end{proof}

\end{document}